\documentclass[11pt]{amsart}

\usepackage[a4paper,margin=1.08in]{geometry}
\usepackage{amsmath,amssymb,amsthm,mathtools}
\usepackage{mathrsfs}
\usepackage{enumitem}
\usepackage{xcolor}
\usepackage{hyperref}
\usepackage{microtype}
\usepackage{booktabs}
\usepackage{array}
\usepackage{xurl}

\hypersetup{
  colorlinks=true,
  linkcolor=blue!55!black,
  citecolor=blue!55!black,
  urlcolor=blue!55!black
}

\numberwithin{equation}{section}

\newtheorem{theorem}{Theorem}[section]
\newtheorem{proposition}[theorem]{Proposition}
\newtheorem{lemma}[theorem]{Lemma}
\newtheorem{corollary}[theorem]{Corollary}
\theoremstyle{definition}

\theoremstyle{remark}
\newtheorem{remark}[theorem]{Remark}

\newcommand{\PP}{\mathbb P}
\newcommand{\CC}{\mathbb C}
\newcommand{\OO}{\mathcal O}
\newcommand{\bT}{\overline T}
\newcommand{\bK}{\overline K}
\newcommand{\floor}[1]{\left\lfloor #1\right\rfloor}
\DeclareMathOperator{\Gr}{Gr}

\DeclareMathOperator{\Sym}{Sym}

\title[Two-component logarithmic hyperbolicity]{Invariant two-jets and effective hyperbolicity for complements of two plane curves}

\author{Lei Hou}
\address{Academy of Mathematics and Systems Science, Chinese Academy of Sciences, Beijing 100190, China}
\email{houlei@amss.ac.cn}

\author{Pengchao Wang}
\address{Academy of Mathematics and Systems Science, Chinese Academy of Sciences, Beijing 100190, China}
\email{2589355972@qq.com}

\author{Song-Yan Xie}
\address{State Key Laboratory of Mathematical Sciences, Academy of Mathematics and Systems Science, Chinese Academy of Sciences, Beijing 100190, China}
\address{School of Mathematical Sciences, University of Chinese Academy of Sciences, Beijing 100049, China}
\email{xiesongyan@amss.ac.cn}

\keywords{logarithmic jet differential, Demailly--Semple tower, entire curve, Second Main Theorem, hyperbolicity, exact finite-field computation}
\date{\today}

\begin{document}

\begin{abstract}
Let $D=C_1+C_2\subset\PP^2$ be a simple-normal-crossing union of smooth
plane curves of degrees $1\leqslant d_1\leqslant d_2$.  We prove an effective
Second Main Theorem on a nonempty Zariski-open set of ordered pairs whenever
\[
 d_1,d_2\geqslant3,
 \qquad\text{or}\qquad
 d_1=2,\ d_2\geqslant5,
 \qquad\text{or}\qquad
 d_1=1,\ d_2\geqslant8.
\]
After intersecting with the standard very-general
algebraic-hyperbolicity locus, we obtain Kobayashi hyperbolicity and
hyperbolic embedding of the complement. The proof uses logarithmic
invariant two-jet differentials, the corrected
Demailly--El Goul zero-locus argument, and a two-component adaptation of
slanted vector fields with base twist $\OO_{\PP^2}(3)$.  The adapted twist
three is uniform in $(d_1,d_2)$:
simultaneous cancellation of the two root equations requires third Taylor
differences, but no seventh-order polarization.  For a linear component the
third-remainder space vanishes, and ten compatible mixed generators replace
the generic Taylor package.  A numerical phase diagram shows that exact
finite computation is needed only for the six boundary pairs
$(3,3),(3,4),(2,6),(2,5),(1,9),(1,8)$.  Their key vanishings are certified by
full-column-rank matrices over finite fields.  For two cubics, a strengthened
range reduces to twelve theorem-relevant degree--twist cases; two low-weight
self-tests bring the computational total to $168$ character--eigenvalue
blocks.  The exact source, certificates, and
reproduction instructions form an accompanying computational archive.  In
the two-cubic case the resulting explicit estimate is
\[
 T_f(r)\leqslant57\bigl(N_f^{[1]}(r,C_1)+N_f^{[1]}(r,C_2)\bigr)
 +o(T_f(r))\ \|.
\]
For a conic and a quintic one may take the explicit coefficient $45$, and
for a line and an octic one may take $69$.
\end{abstract}

\maketitle
\tableofcontents

\section{Introduction}\label{sec:introduction}

\subsection{Historical setting}
The logarithmic hyperbolicity problem asks how the geometry of a divisor
$D\subset\PP^n$ constrains entire curves in $\PP^n\setminus D$.  In dimension
one, the model is Picard's theorem: the complement of three points in
$\PP^1$ admits no nonconstant entire curve.  In higher dimension the guiding
principle is that sufficiently positive logarithmic cotangent geometry should
force the same rigidity.  The Green--Griffiths construction and Demailly's
reparametrization-invariant refinement turn this principle into differential
equations on jets of entire curves.  In the logarithmic setting this leads to
the bundles
\[
 E_{k,m}T_X^*(\log D),
\]
whose negatively twisted sections vanish along every entire curve in the
complement; see \cite{Demailly1997,DethloffLu}.  The central difficulty is not
merely to produce one equation by Riemann--Roch, but to control its base locus
well enough to obtain an independent equation.

For logarithmic surfaces, two complementary ideas made this strategy
effective.  Miyaoka's zero-locus method (see also Lu--Yau), in the form used
by Demailly--El Goul~\cite{Miyaoka,LuYau,DemaillyElGoul}, produces a second equation intrinsically on an
irreducible component of the first equation's zero divisor.  Siu's slanted
vector-field method \cite{Siu}, developed effectively in dimension two by
P\u{a}un \cite{Paun} and in the logarithmic setting by Rousseau
\cite{Rousseau}, uses the natural action of invariant fields on the universal
jet space to turn the first equation $\omega_1$ into a genuine global second
jet differential $\omega_2$.
The gain from differentiation is governed by the pole order of the available
vector fields; this is why replacing a twist seven construction by twist
three has a direct quantitative consequence for a Second Main Theorem.

For the complement of one very general smooth plane curve, Rousseau obtained
the long-standing degree bound $d\geqslant14$
\cite[Theorem~2]{Rousseau}.  Recent work of
Hou--Huynh--Merker--Xie lowered it to $d\geqslant12$ by combining invariant
two-jets with a finite key-vanishing calculation
\cite{HouHuynhMerkerXie2026}; a subsequent symmetry method sharpened the
computational framework further \cite{CuiHouLiuXie2026}.  A related
three-component model---three general conics---admits a particularly simple
logarithmic frame and yields an effective Second Main Theorem
\cite{HouHuynhMerkerXie}.  These developments show that low-degree questions
are controlled by a delicate interface between intrinsic jet geometry,
explicit slanted fields, and exact finite linear algebra.

The chronological starting point is the three-conics paper
\cite{HouHuynhMerkerXie}.  Its decisive innovation is the exponent bound
that converts an a priori infinite jet-vanishing problem into finite exact
algebra; its Second Main Theorem argument also makes essential use of the
geometric structure of the Wronskian.  The later one-component paper
\cite{HouHuynhMerkerXie2026} develops this discovery into a one-component
theory, including the local invariant-two-jet normal form and the first
concrete implementation of the finite Key Vanishing strategy.  The
higher-direct-image, zero-locus, and algebraic-extension ingredients are
inherited from the preceding work; at each use we identify both the source
and the two-component specialization.  The symmetry organization and the
pole-order-three acceleration package follow
\cite[Section~2]{CuiHouLiuXie2026}; Section~\ref{sec:mixed-O3} gives the complete
two-component adaptation.  The principal new
technical result here is the Key Vanishing Lemma for two components: its
coupled normal form, unequal-degree divisibility conditions, symmetry
decomposition, and exact finite-rank certificates do not occur in any of the
three preceding papers.  This is the fourth paper of the program, but it is
not merely an application of the preceding strategy: the two-component
vanishing requires a new coupled proof with two residue directions,
degree-sensitive overlap equations, and exact symmetry blocks adapted to
ordered pairs.  The remaining ingredients require genuine two-component
adaptations, although we do not present the overall strategy itself as new.

The case of two components has its own history.  Let
\[
 D=C_1+C_2\subset\PP^2,
 \qquad 1\leqslant d_1:=\deg C_1\leqslant d_2:=\deg C_2,
\]
where $C_1$ and $C_2$ are smooth and transverse.  For very general ordered
pairs, Rousseau's first results
\cite[Th\'eor\`emes~1.3 and~2.14]{Rousseau2003} prove hyperbolicity in the ranges
$d_1\geqslant5$, $d_1=4,d_2\geqslant7$, $(d_1,d_2)=(4,4)$,
$d_1=3,d_2\geqslant9$, and $d_1=2,d_2\geqslant12$.  His later logarithmic
vector-field construction \cite[Theorem~3]{Rousseau}, also for very general
pairs, improves these to
$d_1\geqslant4$, $d_1=3,d_2\geqslant5$, $d_1=2,d_2\geqslant8$, and
$d_1=1,d_2\geqslant11$.  The present paper reaches every pair with
$d_1,d_2\geqslant3$, every pair with $d_1=2,d_2\geqslant5$, and every pair
with $d_1=1,d_2\geqslant8$.  In particular, it settles the two-cubic and
cubic--quartic cases, lowers the degree-two boundary from $8$ to $5$, and
lowers the linear-component boundary from $11$ to $8$.

\begin{remark}[The line--septic frontier]\label{rem:line-septic-frontier}
For $(d_1,d_2)=(1,e)$ one has
\[
 \bar c_1^2=(e-2)^2,\qquad
 \bar c_2=(e-1)^2,
 \qquad
 A(1,e)=4e^2-34e+43.
\]
At $e=10$ this gives
\[
 A(1,10)=103,\qquad
 \rho_{\mathrm{DEG}}(1,10)=\frac{103}{96}>1.
\]
Thus no finite Key Vanishing Lemma is needed for $e\geqslant10$.  For $e=9$
one has $\rho_{\mathrm{DEG}}=61/84$, leaving only
$(m,t)=(3,3),(4,3)$; for $e=8$ one has
$\rho_{\mathrm{DEG}}=3/8$, leaving
\[
 (3,2),(4,2),(5,2),(6,3),(7,3),(8,3).
\]
Lemma~\ref{lem:keyvanishing} certifies precisely these eight cases, and the
degree-one construction in Section~\ref{sec:mixed-O3} supplies the required
pole-three differentiation package.  This is why the theorem begins at
$e=8$.

At $e=7$ one has $A(1,7)=1$ and
$\rho_{\mathrm{DEG}}=1/60$; the required minimal profile is
\[
 \begin{array}{c|c}
  t& m\\ \hline
  1&3\leqslant m\leqslant60\\
  2&61\leqslant m\leqslant120\\
  3&121\leqslant m\leqslant180.
 \end{array}
\]
Since $A(1,6)=-17$, the pair $(1,7)$ is the formal Riemann--Roch endpoint of
this two-jet architecture.  It should not be confused with a predicted Key
Vanishing endpoint: $A(1,7)$ is barely positive, the list contains $178$
minimal cases, and the numerics make it plausible that at least one of the
corresponding negatively twisted spaces is nonzero.  No claim for $(1,7)$
is made here.
\end{remark}

The two-cubic case is especially close to the conjectural low-degree
boundary.  Its divisor has only two irreducible components and total degree
six.  Thus it has the same total degree as the three-conics model but one
fewer component, while total degree six is only one above the usual plane
Kobayashi benchmark $2\dim\PP^2+1=5$.  The reducible two-component family is
of course not the full parameter space of irreducible plane curves of degree
six, so these statements are not formally instances of the same genericity
conjecture.  Nevertheless, simultaneously lowering the number of components
and keeping the total degree this small removes much of the logarithmic
positivity available in hyperplane- or many-component configurations.  In
this precise sense, the two-cubic theorem illustrates the low-degree reach of
the invariant $2$-jet method.

Throughout, a property holds for a \emph{general} ordered pair if it holds
on some nonempty Zariski-open subset of the relevant parameter space.  It
holds for a \emph{very general} ordered pair if its failure locus is
contained in a countable union of proper Zariski-closed subsets.  We retain
the order even when $d_1=d_2$; no quotient by component exchange is taken.

\subsection{What must be adapted for two components}
Put
\[
 s=d_1+d_2,\qquad \kappa=s-3,
\]
and write
\[
 \bT^*=T_{\PP^2}^*(\log D),\qquad
 \bT=T_{\PP^2}(-\log D),\qquad
 \bK=K_{\PP^2}+D\simeq\OO_{\PP^2}(\kappa).
\]
We also write
\[
 h=c_1\!\left(\OO_{\PP^2}(1)\right),
 \qquad \bar c_i=c_i(\bT^*).
\]
Unless another space is displayed, $\OO(q)$ means $\OO_{\PP^2}(q)$; for an
affine variety $U$, the notation $\OO(U)$ denotes its ring of regular
functions.
The global Demailly--Semple strategy is inherited from the three works just
cited.  Its implementation for two components is nevertheless not obtained
by applying the one-component formulas twice.  We record the necessary
adaptations, distinguishing them from the main new key-vanishing argument.

First, the universal root cover has two logarithmic root variables, and
the mixed vector fields must satisfy the tangency equations for both
components with a common quadratic acceleration in the base variables.
We construct these fields by prescribing their jet components and
correcting their parameter polynomials by third-order Taylor differences.
For $d_1,d_2\geqslant2$, this correction gives parameter polynomials of
the required degrees without changing their Taylor data through order
two.  Coefficient normalization and the projective coordinate changes
then yield a finite family of global fields with base twist
$\OO_{\PP^2}(3)$.
To prove independence of the differentiated equation, we work on a
fixed jet fiber.  The simultaneous gradient- and Hessian-cancellation
equations determine the normal rows and the common quadratic
acceleration.  Third differences absorb the remaining parameter
components when $d_i\geqslant3$; for $d_i=2$, the corrected parameter
polynomial already vanishes once its Taylor data through order two
vanish.  When $d_1=1$, a compatibility condition replaces the degree
correction in the linear block, and ten explicit mixed generators
supply the required data.  After all non-jet components have been
cancelled on the fixed fiber, a commutator calculation proves
independence and hence selects a derivative by one of the original
global fields.  Combined with the zero-locus estimates and finite
vanishing results, this construction gives the stated degree ranges,
including $d_1=2,\ d_2\geqslant5$ and $d_1=1,\ d_2\geqslant8$.

Second, the numerical problem is genuinely two-dimensional.  The logarithmic
Chern numbers depend separately on $d_1$ and $d_2$, and the competition
between one-jets, invariant two-jets, and the zero-locus method produces a
phase diagram in the $(d_1,d_2)$-plane.  The one-jet Segre number is
\[
 \bar c_1^2-\bar c_2=d_1d_2-3(d_1+d_2)+6.
\]
Thus one-jets already treat $d_1\geqslant5$ and $d_1=4,d_2\geqslant7$, but they do
not see the low-degree corner addressed here.  Invariant two-jets reduce that
corner to six boundary pairs
$(3,3),(3,4),(2,6),(2,5),(1,9),(1,8)$.

Third, the key vanishing calculation is coupled.  For one component there is
one residue direction and one defining polynomial.  Here chart changes mix
two logarithmic residues, and unequal degrees carry different homogeneities.
We derive a finite normal form that records these coupled transition and
divisibility conditions as sparse integral matrices.  Equal-degree symmetry
splits the two-cubic matrices into character and involution blocks; the
unequal-degree pairs require separate normal forms and carefully chosen SNC
test divisors.  Exact full-column rank modulo one good prime proves rank over
$\mathbb Q$, and semicontinuity then gives vanishing on a parameter open.
This is an exact proof certificate, not a numerical floating-point test.
The symmetry decomposition and exact modular elimination make the finite
computations needed to obtain the coefficient $57$ in the two-cubic case
feasible on a standard laptop.
The supplementary material provides the computational records and the
instructions needed for independent verification.

The principal technical originality of the paper is the third point: the
coupled finite key-vanishing theorem and its proof certificates.  The first
point is a nontrivial two-component adaptation of the $\OO(3)$ mechanism of
\cite{CuiHouLiuXie2026}, and the second is the numerical bookkeeping needed
to deploy the inherited strategy in the $(d_1,d_2)$-plane.  Together these
ingredients retain pole order three in the two-component setting and make the
new degree range possible.

\subsection{The finite Key Vanishing Lemma}

For $d\geqslant1$ write $V_d=H^0(\PP^2,\OO_{\PP^2}(d))$.  For fixed degrees put
\[
 \mathcal S_{d_1,d_2}
 =\PP(V_{d_1})\times\PP(V_{d_2}),
\]
and denote by $\mathcal S_{d_1,d_2}^{\mathrm{snc}}$ its nonempty open
simple-normal-crossing locus.  Define
\begin{align*}
 A(d_1,d_2)
 &=4(d_1^2+d_2^2)+17d_1d_2-51(d_1+d_2)+90,\\
 \rho_{\mathrm{DEG}}(d_1,d_2)
 &=\frac{A(d_1,d_2)}{12(d_1+d_2-3)}.
\end{align*}
For a genuine second-level divisor only $m\geqslant3$ occurs.  For the five
boundary pairs other than $(3,3)$ set
\[
\mathfrak F_{\mathrm{bd}}(d_1,d_2)
 =\left\{(m,t)\in\mathbb Z_{\geqslant3}\times\{1,2,3\}:\
 \frac tm\geqslant\rho_{\mathrm{DEG}}(d_1,d_2)\right\}.
\]
For the conic--quintic pair we certify one additional minimal case and write
\[
\mathfrak F_{45}(2,5)
 =\left\{(m,t)\in\mathbb Z_{\geqslant3}\times\{1,2,3\}:\
 \frac tm>\frac13\right\}
 =\mathfrak F_{\mathrm{bd}}(2,5)\cup\{(8,3)\}.
\]
For two cubics we use the stronger range
\[
\mathfrak F_{57}(3,3)
 =\left\{(m,t)\in\mathbb Z_{\geqslant1}\times\{1,2,3\}:\
 1\leqslant m\leqslant5t-1\right\}.
\]
The two cases with $m=1,2$ are included as low-weight consistency tests; the
geometric argument uses only the twelve cases with $m\geqslant3$.

\begin{lemma}[Finite key vanishing]\label{lem:keyvanishing}
For each of the six degree pairs
\[
 (d_1,d_2)\in\{(3,3),(3,4),(2,6),(2,5),(1,9),(1,8)\}
\]
there exists a nonempty Zariski-open subset
$\mathcal U_{\mathrm{van}}(d_1,d_2)\subset
\mathcal S_{d_1,d_2}^{\mathrm{snc}}$ such that, for every parameter $a$ in
this subset, one has
\begin{equation}\label{eq:keyvanishing}
   H^0\!\left(\PP^2,
   E_{2,m}T_{\PP^2}^*(\log D_a)\otimes\OO_{\PP^2}(-t)\right)=0
\end{equation}
for every $(m,t)\in\mathfrak F_{57}(3,3)$ when
$(d_1,d_2)=(3,3)$, for every $(m,t)\in\mathfrak F_{45}(2,5)$ when
$(d_1,d_2)=(2,5)$, and for every
$(m,t)\in\mathfrak F_{\mathrm{bd}}(d_1,d_2)$ for the other four pairs.
\end{lemma}

The finite sets, after reduction to the least twist for each $m$, are shown
below.  The second column is the zero-locus threshold
$\rho_{\mathrm{DEG}}(d_1,d_2)$, and the third column lists the pairs
$(m,t)$ for which the vanishing \eqref{eq:keyvanishing} must be checked.
Thus the cubic row records fourteen computations: the four pairs $(m,1)$ with
$1\leqslant m\leqslant4$, the five pairs $(m,2)$ with $5\leqslant m\leqslant9$, and the five
pairs $(m,3)$ with $10\leqslant m\leqslant14$; its first two entries are the
self-tests just described.
\begin{center}
\begin{tabular}{cc>{\raggedright\arraybackslash}p{0.58\textwidth}}
\toprule
$(d_1,d_2)$ & $\rho_{\mathrm{DEG}}$ & minimal pairs $(m,t)$ \\
\midrule
$(3,3)$ & $1/4$ & $1\leqslant m\leqslant4$ for $t=1$; $5\leqslant m\leqslant9$ for $t=2$;
$10\leqslant m\leqslant14$ for $t=3$ \\
$(3,4)$ & $37/48$ & $(3,3)$ \\
$(2,6)$ & $23/30$ & $(3,3)$ \\
$(2,5)$ & $19/48$ & $(3,2),(4,2),(5,2),(6,3),(7,3),(8,3)$ \\
$(1,9)$ & $61/84$ & $(3,3),(4,3)$ \\
$(1,8)$ & $3/8$ & $(3,2),(4,2),(5,2),(6,3),(7,3),(8,3)$ \\
\bottomrule
\end{tabular}
\end{center}
Multiplication by a nonzero homogeneous polynomial handles all larger twists
with the same $m$.  Exact modular full rank at one SNC pair implies rank over
$\mathbb Q$, hence over $\mathbb C$; upper semicontinuity then promotes the
vanishing to the asserted parameter open.
The computational archive contains the case list, matrix dimensions,
ranks, nullities, and the data needed for independent
verification.\footnote{\label{fn:computational-archive}The computational archive is available at
\url{https://xiesongyan.github.io/codes/hou_wang_xie_two_plane_curves_cpp_evidence_2026.zip}.}

\subsection{Main statements}

\begin{theorem}[Effective Second Main Theorem]\label{thm:mainSMT}
Let $1\leqslant d_1\leqslant d_2$ be integers.  Assume that
\[
 d_1,d_2\geqslant3,
 \qquad\text{or}\qquad d_1=2,\ d_2\geqslant5,
 \qquad\text{or}\qquad d_1=1,\ d_2\geqslant8.
\]
There exist a
nonempty Zariski-open subset
$\mathcal U_{\mathrm{SMT}}(d_1,d_2)\subset
\mathcal S_{d_1,d_2}^{\mathrm{snc}}$ and an effectively computable constant
$\mathcal A_{d_1,d_2}>0$, both depending only on the ordered degree pair.
For every $a$ in this subset, write
$D_a=C_{1,a}+C_{2,a}$.  Then every algebraically nondegenerate entire curve
$f:\CC\to\PP^2$ satisfies
\begin{equation}\label{eq:mainSMT}
   T_f(r)
   \leqslant
   \mathcal A_{d_1,d_2}\,N_f^{[1]}(r,D_a)+o(T_f(r))\ \|,
\end{equation}
where, throughout the paper, truncation for a reduced divisor is taken
componentwise:
\[
   N_f^{[1]}(r,D_a)
   =N_f^{[1]}(r,C_{1,a})+N_f^{[1]}(r,C_{2,a}).
\]
For $(d_1,d_2)=(3,3)$ one may take $\mathcal A_{3,3}=57$; for
$(d_1,d_2)=(2,5)$ one may take $\mathcal A_{2,5}=45$; and for
$(d_1,d_2)=(1,8)$ one may take $\mathcal A_{1,8}=69$.
\end{theorem}

Let $\mathcal S_{\mathrm{Chen}}(d_1,d_2)$ denote the complement of the
countable union of proper closed subsets excluded by
Theorem~\ref{thm:Chen}.

\begin{corollary}[Hyperbolicity and hyperbolic embedding]\label{cor:hyperbolicity}
Under the degree assumptions of Theorem~\ref{thm:mainSMT}, the complement
$\PP^2\setminus D_a$ is Kobayashi hyperbolic and hyperbolically embedded
in $\PP^2$ for every
$a\in\mathcal U_{\mathrm{SMT}}(d_1,d_2)\cap
\mathcal S_{\mathrm{Chen}}(d_1,d_2)$.
In particular, this conclusion holds for a very general ordered pair of
smooth transverse curves in the stated degree range.
\end{corollary}

\begin{remark}[Genericity conditions and a subsequent strengthening]
Xie--Zhao have strengthened the hyperbolicity and hyperbolic-embedding
conclusions of Corollary~\ref{cor:hyperbolicity} from very general to
general ordered pairs in the same degree ranges
\cite[Theorem~F]{XieZhao2026}.

In our proof, the two genericity conditions arise at different stages.
The effective Second Main Theorem holds on the nonempty Zariski-open
subset $\mathcal U_{\mathrm{SMT}}$, whereas the application of
Xi Chen's theorem requires the very-general locus
\[
 \mathcal S_{\mathrm{Chen}}
 =
 \mathcal S_{d_1,d_2}^{\mathrm{snc}}
 \setminus\bigcup_{\nu\geqslant1}Z_\nu,
\]
where each $Z_\nu$ is a proper closed subset.
Since $\mathcal S_{d_1,d_2}^{\mathrm{snc}}$ is irreducible and
$\mathcal U_{\mathrm{SMT}}$ is nonempty and open, its complement
is also a proper closed subset.  Consequently,
\[
 \mathcal U_{\mathrm{SMT}}\cap\mathcal S_{\mathrm{Chen}}
 =
 \mathcal S_{d_1,d_2}^{\mathrm{snc}}\setminus
 \left(
   \bigl(\mathcal S_{d_1,d_2}^{\mathrm{snc}}
         \setminus\mathcal U_{\mathrm{SMT}}\bigr)
   \cup\bigcup_{\nu\geqslant1}Z_\nu
 \right)
\]
is a nonempty very-general subset over $\CC$.
Thus our effective estimate holds on a Zariski-open locus, while
our deduction of hyperbolicity also uses a very-general condition
to exclude algebraically degenerate entire curves.
\end{remark}

\subsection{Proof strategy}
Riemann--Roch and logarithmic $H^2$ vanishing first produce a negatively
twisted invariant two-jet differential.  The Demailly--El Goul reduction then
selects an irreducible factor $\omega_1$ of weighted degree $m$ and twist
$-t$, without decreasing the twist-to-weight ratio.  If
$f^*\omega_1\not\equiv0$, the jet Second Main Theorem gives the estimate
directly.  A zero component pulled back from the first Semple level is handled
by McQuillan's theorem.  For a genuine second-level component, $t>3$ is
handled by differentiating with the mixed $\OO(3)$ fields.  If
$t\leqslant3$, the numerical threshold either places $t/m$ in the zero-locus
range or reduces the problem to the finite set in
Lemma~\ref{lem:keyvanishing}; the restricted-bigness argument then supplies a
second section on the zero divisor.  Finally, Xi Chen's inequality excludes
algebraically degenerate entire curves from a very general complement.

The explicit endpoint constants use sharper versions of this dichotomy.  For
a conic and a quintic, the additional certified case $(m,t)=(8,3)$ moves the
split to $\rho=t/m=1/3$.  The branch $\rho\leqslant1/3$ has zero-locus
coefficient strictly smaller than $45$, whereas $\rho>1/3$ forces
$t\geqslant4$ and $m\leqslant3t-1$, giving differentiated coefficient at most
$11$.  Choosing the initial ratio above $1/10$ bounds the direct coefficient
strictly below $10$.

For two cubics the strengthened finite vanishing range moves the split to
$\rho=1/5$.  When $\rho\leqslant1/5$, the limiting zero-locus coefficient is
\[
 \frac{57+\sqrt{3189}}2<57.
\]
Thus the coefficient in the zero-locus estimate can be chosen below $57$.
When $\rho>1/5$, the finite theorem excludes $t\leqslant3$; hence $t\geqslant4$, and
integrality gives $m\leqslant5t-1$.  The $\OO(3)$ derivative therefore gives
\[
 \frac{m}{t-3}\leqslant\frac{5t-1}{t-3}\leqslant19.
\]
Together with the direct and McQuillan branches, this proves the stated
integer constant $57$.

Section~\ref{sec:constant-frontier} explains why the stated coefficients
suffice to demonstrate the method.

Section~\ref{sec:jets} collects the logarithmic-jet notation and external
analytic inputs.  Section~\ref{sec:numerical-thresholds} computes the
logarithmic Chern numbers and the two numerical thresholds.
Section~\ref{sec:first-differential} constructs the first differential.
Section~\ref{sec:mixed-O3} proves the simultaneous mixed $\OO(3)$
construction, which is the geometric core of the paper.
Section~\ref{sec:mechanisms} combines the zero-locus and differentiation
mechanisms, and Section~\ref{sec:main-proofs} assembles the main results.
Section~\ref{sec:finite-vanishing} proves the finite Key Vanishing Lemma by
symmetry-adapted exact linear algebra; the supporting computations and
verification instructions are provided in the supplementary material.  The final
Section~\ref{sec:constant-frontier} discusses the role of the explicit
constants.

\section{Logarithmic jets and external inputs}\label{sec:jets}

\subsection{Nevanlinna functions}
Let $\omega_{\mathrm{FS}}$ be the Fubini--Study form.  We use the standard Nevanlinna functions and normalization; see \cite{RuBook}.  For $f:\CC\to\PP^2$, set
\[
   T_f(r)=\int_1^r\frac{ds}{s}\int_{|z|<s}f^*\omega_{\mathrm{FS}}.
\]
If $E\subset\PP^2$ is an irreducible divisor and $f(\CC)\not\subset E$, let
\[
   N_f^{[1]}(r,E)=\int_1^r
   \sum_{|z|<s}\min\{1,\operatorname{ord}_z f^*E\}\frac{ds}{s}.
\]
For the reduced divisor $D=C_1+C_2$, we use the componentwise convention
\[
   N_f^{[1]}(r,D)
   :=N_f^{[1]}(r,C_1)+N_f^{[1]}(r,C_2).
\]
The symbol $\|$ means that the inequality holds for all $r>1$ outside a set of finite Lebesgue measure.  For an algebraically nondegenerate curve in $\PP^2$, the standard growth lemma gives
\[
   O(\log T_f(r)+\log r)=o(T_f(r))\ \|,
\]
so we shall replace the logarithmic error term by $o(T_f)$ after applying a jet-differential estimate.

\subsection{Invariant logarithmic two-jets}
We follow the standard invariant logarithmic jet notation of \cite{Demailly1997,DethloffLu}.  Locally on a logarithmic surface, an invariant two-jet differential of weighted degree $m$ has the form
\begin{equation}\label{eq:local-twojet}
   \sum_{0\leqslant j\leqslant\floor{m/3}}
   \sum_{\alpha_1+\alpha_2=m-3j}
   R_{\alpha_1,\alpha_2,j}
   (f_1')^{\alpha_1}(f_2')^{\alpha_2}
   (f_1'f_2''-f_2'f_1'')^j.
\end{equation}
Its associated graded bundle is
\begin{equation}\label{eq:filtration}
   \Gr^\bullet E_{2,m}\bT^*
   =\bigoplus_{0\leqslant j\leqslant\floor{m/3}}
   S^{m-3j}\bT^*\otimes\bK^j.
\end{equation}

\subsection{The logarithmic Demailly--Semple tower}
Set $X_0=\PP^2$ and $V_0=\bT$.  We use the convention that $\PP(V)$ parametrizes lines in $V$, so that $\OO_{\PP(V)}(-1)$ is the tautological subbundle.  Let
\[
   X_1=\PP(V_0),\qquad X_2=\PP(V_1),
\]
be the first two levels of the logarithmic Demailly--Semple tower, with projections $\pi_{2,1}:X_2\to X_1$ and $\pi_{2,0}:X_2\to X_0$.  Write
\[
   u_1=c_1\bigl(\pi_{2,1}^*\OO_{X_1}(1)\bigr),
   \qquad
   u_2=c_1\bigl(\OO_{X_2}(1)\bigr),
   \qquad
   h=\pi_{2,0}^*c_1\bigl(\OO_{\PP^2}(1)\bigr),
\]
and
\[
   \OO_{X_2}(a,b)
   :=\pi_{2,1}^*\OO_{X_1}(a)\otimes\OO_{X_2}(b).
\]
Let $\Gamma_2\subset X_2$ be the vertical divisor at the second stage.  With the above convention,
\[
   \OO_{X_2}(\Gamma_2)\simeq\OO_{X_2}(-1,1),
   \qquad
   \OO_{X_2}(3)\otimes\OO_{X_2}(-2\Gamma_2)
   \simeq\OO_{X_2}(2,1).
\]
Let $\OO_{X_2}(m)$ denote the standard level-two Demailly--Semple tautological bundle of weighted degree $m$.  The direct-image formula identifies invariant two-jets with its sections:
\begin{equation}\label{eq:direct-image}
   (\pi_{2,0})_*\OO_{X_2}(m)
   \simeq E_{2,m}\bT^*.
\end{equation}
We shall also use the following more precise direct-image statement for the
bundle $\OO_{X_2}(2,1)$.  Its two-stage filtration proof is the one used in
\cite[Proposition~6.1]{HouHuynhMerkerXie}.  Keeping every negative-degree
graded piece gives the sharp criterion in terms of the single defect
$a_1-2a_2$; we record the resulting statement and proof for later use.

\begin{proposition}[Sharp higher direct images]
\label{prop:correct-higher-direct-images}
Let $a_1,a_2\in\mathbb Z$.  If $a_2=-1$, then
\[
 R^q(\pi_{2,0})_*\OO_{X_2}(a_1,-1)=0
 \qquad(q\geqslant0).
\]
Suppose that $a_2\geqslant0$ and put $d=a_1-2a_2$.  Then
\[
 R^q(\pi_{2,0})_*\OO_{X_2}(a_1,a_2)=0\qquad(q\geqslant2),
\]
and
\[
 R^1(\pi_{2,0})_*\OO_{X_2}(a_1,a_2)=0
 \quad\Longleftrightarrow\quad d\geqslant-1.
\]
If $d\leqslant-2$, the nonzero $R^1$ has a finite filtration with successive
quotients
\begin{equation}\label{eq:correct-R1-quotients}
 \bK^{a_2-k-1}\otimes\Sym^{-d-3k-2}\bT,
 \qquad
 0\leqslant k\leqslant
 \min\left\{a_2,\floor{\frac{-d-2}{3}}\right\}.
\end{equation}
\end{proposition}

\begin{proof}
Write $q=\pi_{2,1}$ and $\pi=\pi_{1,0}$.  If $a_2=-1$, the restriction
to each $q$-fiber is $\OO_{\PP^1}(-1)$, so every direct image under $q$,
and hence under $\pi_{2,0}$, vanishes.  Assume $a_2\geqslant0$.  The
projective-bundle formula gives
\[
 R^j q_*\OO_{X_2}(a_1,a_2)=0\quad(j>0),
 \qquad
 q_*\OO_{X_2}(a_1,a_2)
 =\OO_{X_1}(a_1)\otimes\Sym^{a_2}V_1^*.
\]
The dual directed sequence
\[
 0\longrightarrow\OO_{X_1}(1)
 \longrightarrow V_1^*
 \longrightarrow\pi^*\bK\otimes\OO_{X_1}(-2)
 \longrightarrow0
\]
induces a filtration whose graded quotients are
\begin{equation}\label{eq:correct-Gk}
 \pi^*\bK^{a_2-k}\otimes\OO_{X_1}(d+3k),
 \qquad0\leqslant k\leqslant a_2.
\end{equation}
For the rank-two projective bundle of lines,
$R^1\pi_*\OO_{X_1}(n)=0$ for $n\geqslant-1$, while relative Serre duality
gives, for $n\leqslant-2$,
\[
 R^1\pi_*\OO_{X_1}(n)
 \simeq\Sym^{-n-2}\bT\otimes\det\bT
 =\Sym^{-n-2}\bT\otimes\bK^{-1}.
\]
Applying this to \eqref{eq:correct-Gk} gives precisely the quotients
\eqref{eq:correct-R1-quotients}.  In the bad range
$d+3k\leqslant-2$, the corresponding ordinary direct image is zero; ordering
the filtration from the good range to the bad range shows through the long
exact sequences that these $R^1$ terms assemble by successive extensions.
There are no direct images in degree at least two because both stages have
relative dimension one and the first push has no positive higher direct
images.  The vanishing criterion is therefore exactly $d\geqslant-1$.
\end{proof}

\begin{remark}[Fiberwise interpretation]
On a fiber $F\simeq\mathbb F_3$ of $\pi_{2,0}$,
\[
 \OO_{X_2}(a_1,a_2)|_F
 \simeq\OO_{\mathbb F_3}\bigl(a_2E+(a_1+a_2)f\bigr),
\]
and for $a_2\geqslant0$ its cohomology is controlled by the line-bundle degrees
$d+3k$, $0\leqslant k\leqslant a_2$.  Thus the defect
$a_1-2a_2$ is sharp; conditions involving $a_1$ and $a_2$ separately do
not suffice.
\end{remark}

\begin{lemma}[Direct images of $\OO_{X_2}(2p,p)$]\label{lem:direct-image-21}
For every integer $p\geqslant1$ one has
\begin{equation}\label{eq:direct-image-21}
   (\pi_{2,0})_*\OO_{X_2}(2p,p)
   \simeq E_{2,3p}\bT^*.
\end{equation}
Moreover,
\begin{equation}\label{eq:higher-direct-image-21}
   R^q(\pi_{2,0})_*\OO_{X_2}(2p,p)=0
   \qquad(q\geqslant1).
\end{equation}
\end{lemma}

\begin{proof}
The identification \eqref{eq:direct-image-21} is the standard
Demailly--Semple direct-image formula for the weight-$3p$ subbundle; see
\cite[Lemma~3.3(c)]{DemaillyElGoul}.  The higher direct images vanish by
Proposition~\ref{prop:correct-higher-direct-images}, because
$a_2=p\geqslant0$ and $a_1-2a_2=2p-2p=0$.
\end{proof}

Thus the $p$-th power of $\OO_{X_2}(2,1)$ corresponds canonically to
weighted degree $3p$.

We shall repeatedly use the following elementary fact about lifted curves.

\begin{lemma}[The regular lift avoids the vertical divisor]
\label{lem:regular-lift-not-vertical}
Let $f:\CC\to X_0$ be nonconstant and satisfy $f(\CC)\not\subset D$.
Then its logarithmic lift $f_{[2]}$ is well defined after holomorphic
extension across $f^{-1}(D)$ and the isolated zeros of $f'$, and
\[
   f_{[2]}(\CC)\not\subset\Gamma_2.
\]
\end{lemma}

\begin{proof}
On the nonempty open set on which $f$ avoids $D$ and $f'\neq0$, the first lift is
$f_{[1]}(z)=(f(z),[f'(z)])$.  Its derivative is tangent to $V_1$, and its
projection by $d\pi_{1,0}$ is $f'(z)\neq0$.  Hence the line
$[f_{[1]}'(z)]\subset V_{1,f_{[1]}(z)}$ is not vertical for
$\pi_{1,0}$.  By definition, $\Gamma_2$ consists precisely of the vertical
lines at the second stage, so $f_{[2]}(z)\notin\Gamma_2$ on this open set.
To extend the lifts, take a local coordinate $\tau$ at an omitted point.
In a logarithmic frame, a normal component of the first derivative is
$(w\circ f)'/(w\circ f)$, where $w=0$ defines a local branch of $D$;
the remaining components are holomorphic.  These components form a nonzero
meromorphic vector in $f^*\bT$.  Multiplying by a power of $\tau$ and
removing their common vanishing factor gives a holomorphic vector nonzero
at $\tau=0$, hence a unique extension of its projective line.
This extends $f_{[1]}$.  In a local Semple chart its derivative has
meromorphic horizontal logarithmic components and derivatives of
holomorphic projective-coordinate ratios in the vertical components.
It is therefore a nonzero meromorphic vector in $f_{[1]}^*V_1$:
membership in $V_1$ holds on the punctured disk, and $V_1$ is a locally
free subbundle.  The same cancellation
extends its projective line and thus $f_{[2]}$.  The omitted set is discrete
because $f$ is nonconstant and is not contained in $D$.  Since the
complement of $\Gamma_2$ has already met the lifted curve, its whole image
cannot be contained in $\Gamma_2$.
\end{proof}

\subsection{Jet Second Main Theorem and McQuillan's theorem}

We use the following standard jet-differential estimate; compare
\cite{Demailly1997} and
\cite[Theorem~1.2]{HouHuynhMerkerXie}.  The invariant logarithmic jet bundle
$E_{k,m}T_{\PP^2}^*(\log D)$ is naturally a subbundle of the
Green--Griffiths bundle appearing in the statement, so the estimate applies
to the invariant differentials constructed below.

\begin{theorem}[Jet differential Second Main Theorem]\label{thm:jetSMT}
Let $k,m,t\geqslant1$ be integers.  If
\[
   0\neq\omega\in H^0\!\left(\PP^2,
   E_{k,m}^{\mathrm{GG}}T_{\PP^2}^*(\log D)\otimes\OO(-t)\right)
\]
and $f:\CC\to\PP^2$ satisfies $f(\CC)\not\subset D$ and
$f^*\omega\not\equiv0$, then
\begin{equation}\label{eq:jetSMT}
   T_f(r)\leqslant\frac mtN_f^{[1]}(r,D)
   +O(\log T_f(r)+\log r)\ \|.
\end{equation}
\end{theorem}

We also use McQuillan's algebraic-degeneracy theorem in the following
quantitative form; see \cite{McQuillan} and
\cite[Theorem~1.1]{HouHuynhMerkerXie}.

\begin{theorem}[McQuillan]\label{thm:McQuillan}
Let $f:\CC\to\PP^2$ be algebraically nondegenerate.  If its first logarithmic lift is algebraically degenerate, then for a normal-crossing divisor of total degree $d\geqslant4$,
\[
   (d-3)T_f(r)\leqslant N_f^{[1]}(r,D)+o(T_f(r))\ \|.
\]
For $d=s=d_1+d_2$ this gives
\begin{equation}\label{eq:McQuillan-cubic}
   T_f(r)\leqslant\frac1{s-3}N_f^{[1]}(r,D)+o(T_f(r))\ \|.
\end{equation}
\end{theorem}

\subsection{Cohomological and zero-locus inputs}

We next record the vanishing and weighted-degree estimates that produce the
negatively twisted differentials used in the two main cases.

\begin{lemma}[Logarithmic Bogomolov vanishing]
\label{lem:numerical-Bogomolov}
Let $D=C_1+C_2$ be a transverse pair of smooth plane curves, put
$\kappa=d_1+d_2-3>0$, $\bK=K_{\PP^2}+D$, and set
$\bT=T_{\PP^2}(-\log D)$.  For integers $p\geqslant0$ and $q$,
\[
 H^0\!\left(\PP^2,S^p\bT\otimes\bK^q\right)=0
 \qquad\text{if}\qquad
 p>2q.
\]
\end{lemma}

\begin{proof}
This is the logarithmic Bogomolov-type vanishing theorem
\cite[Theorem~2.1]{HouHuynhMerkerXie}.
\end{proof}

\begin{theorem}[The required logarithmic $H^2$ vanishing]\label{thm:H2vanishing}
Let $D=C_1+C_2$ be a transverse pair of smooth plane curves with
$\kappa=d_1+d_2-3>0$, and put
$\bK=K_{\PP^2}+D\simeq\OO_{\PP^2}(\kappa)$.  For every rational number
$0\leqslant\delta<1/3$ and every sufficiently divisible $m\gg1$,
\begin{equation}\label{eq:H2vanishing}
   H^2\!\left(\PP^2,E_{2,m}T_{\PP^2}^*(\log D)
   \otimes\bK^{-\delta m}\right)=0.
\end{equation}
\end{theorem}

\begin{proof}
By Serre duality, the dual of the group in \eqref{eq:H2vanishing} is
\[
 H^0\!\left(\PP^2,
 K_{\PP^2}\otimes (E_{2,m}\bT^*)^\vee
 \otimes\bK^{\delta m}\right).
\]
The dual of the invariant-two-jet filtration has graded pieces
\[
 S^{m-3j}\bT\otimes\bK^{-j}
 \otimes K_{\PP^2}\otimes\bK^{\delta m},
 \qquad 0\leqslant j\leqslant\floor{m/3}.
\]
Because $m$ is sufficiently divisible, $\delta m$ is an integer.  Multiplication
by a nonzero section of $\OO_{\PP^2}(3)$ injects each graded piece into
\begin{equation}\label{eq:H2-graded-piece}
 S^{m-3j}\bT\otimes
 \bK^{\delta m-j}.
\end{equation}
For \eqref{eq:H2-graded-piece}, the inequality in
Lemma~\ref{lem:numerical-Bogomolov} is equivalent to
\[
 (m-3j)-2(\delta m-j)
 =(1-2\delta)m-j
 \geqslant (2/3-2\delta)m>0.
\]
Thus every graded piece has vanishing $H^0$.  Induction along the
filtration gives the vanishing of the whole Serre-dual group, proving
\eqref{eq:H2vanishing}.
\end{proof}

For the zero-locus argument, write
\[
 \bar c_i=c_i\!\left(T_{\PP^2}^*(\log D)\right),
 \qquad h=c_1\!\left(\OO_{\PP^2}(1)\right),
\]
and identify these classes with their pullbacks to the Semple tower.  For
$\rho=t/m$, set
\begin{equation}\label{eq:Q-DEG-definition}
 Q_{\mathrm{DEG}}(\rho,\tau)
 =3\tau^2+\bigl(9\rho-12(\bar c_1\cdot h)\bigr)\tau
 +(13\bar c_1^2-9\bar c_2)-12(\bar c_1\cdot h)\rho.
\end{equation}
When its value at $\tau=0$ is positive, denote its smaller positive root by
\begin{equation}\label{eq:tau1-general}
 \tau_1(\rho)
 =\frac{12\kappa-9\rho-
 \sqrt{(12\kappa-9\rho)^2-
 12\bigl((13\bar c_1^2-9\bar c_2)-12\kappa\rho\bigr)}}6.
\end{equation}
Section~\ref{sec:numerical-thresholds} derives these expressions from the
two-component Chern classes and the tower intersection relations.

\begin{lemma}[Tautological evaluation of a restricted section]
\label{lem:paper4-tautological-evaluation}
Let $Z\subset X_2$ be an effective Cartier divisor and let
\[
 \sigma\in H^0\!\left(
 Z,\OO_{X_2}(2p,p)\otimes
 \pi_{2,0}^*\OO_{\PP^2}(-\widetilde t)\big|_Z
 \right),
 \qquad p,\widetilde t\in\mathbb Z_{>0}.
\]
Let $f\colon\CC\to\PP^2$ be nonconstant, assume $f(\CC)\not\subset D$, and
suppose that its regular logarithmic second lift is contained in $Z$.  On
$\CC\setminus f^{-1}(D)$, regard
$df_{[1]}\colon T_{\CC}\to f_{[1]}^*V_1$ as a differential in the directed
structure.  Wherever it is nonzero, its image is the tautological line
$f_{[2]}^*\OO_{X_2}(-1)$.  It therefore defines, by meromorphic continuation
across the remaining discrete set, the intrinsic tautological derivative
\[
 \mathcal D_Df_{[1]}
 \in H^0_{\mathrm{mer}}\!\left(
 \CC,K_{\CC}\otimes f_{[2]}^*\OO_{X_2}(-1)
 \right).
\]
Then the canonical inclusion
\[
 \iota_p\colon
 \OO_{X_2}(2p,p)
 =\OO_{X_2}(3p)(-2p\Gamma_2)
 \hookrightarrow\OO_{X_2}(3p)
\]
defines the intrinsic meromorphic evaluation
\begin{equation}\label{eq:paper4-canonical-evaluation}
 s_{\sigma,f}:=
 \left\langle
 f_{[2]}^*(\iota_p\sigma),
 (\mathcal D_Df_{[1]})^{\otimes3p}
 \right\rangle
 \in H^0_{\mathrm{mer}}\!\left(
 \CC,K_{\CC}^{3p}\otimes
 f^*\OO_{\PP^2}(-\widetilde t)
 \right).
\end{equation}
This construction does not require an ambient lift of $\sigma$.  Moreover,
\begin{equation}\label{eq:paper4-canonical-pole-divisor}
 \operatorname{div}_{\infty}(s_{\sigma,f})
 \leqslant
 3p\sum_{i=1}^2(f^*C_i)_{\mathrm{red}},
\end{equation}
and
\begin{equation}\label{eq:paper4-canonical-proximity}
 \int_0^{2\pi}\log^+
 \|s_{\sigma,f}(re^{\sqrt{-1}\theta})\|\frac{d\theta}{2\pi}
 =O\!\left(\log^+T_f(r)+\log r\right)\ \|.
\end{equation}
At an isolated zero of $f'$ outside $f^{-1}(D)$, the evaluation has only a
removable singularity or a zero.  At a zero lying over $D$, no pole occurs
beyond the logarithmic boundary poles already counted in
\eqref{eq:paper4-canonical-pole-divisor}.
\end{lemma}

\begin{proof}
Pairing its $3p$-th tensor power with $f_{[2]}^*(\iota_p\sigma)$ gives
\eqref{eq:paper4-canonical-evaluation}.  Since this is a pairing of dual
tautological lines, the result is independent of all Semple coordinates and
all local representatives of $\sigma$.

Equip the line bundles with smooth Hermitian metrics.  Since $Z$ is
projective, the norm of the holomorphic section $\iota_p\sigma$ is bounded
on $Z$.  Hence
\[
 \|s_{\sigma,f}\|
 \leqslant C\,\|\mathcal D_Df_{[1]}\|^{3p}
\]
for a uniform constant $C$.

In coordinates in which $D=\{x_1\cdots x_q=0\}$, the horizontal boundary
components of the derivative of the regular first lift are
$(x_i\circ f)'/(x_i\circ f)$, while the vertical projective-coordinate
components are holomorphic.  Thus the tautological derivative has at most a
simple pole along each reduced pullback $(f^*C_i)_{\mathrm{red}}$.

By Lemma~\ref{lem:regular-lift-not-vertical}, the Semple lift extends across
the isolated zeros of $f'$.  At a critical point outside $f^{-1}(D)$, the
projective ratios extend and the evaluation has only a removable singularity
or a zero.  At a critical point over $D$, a logarithmic component may retain
a simple pole; that pole is already counted by the reduced boundary pullback,
and no additional Semple-type pole appears.  This proves
\eqref{eq:paper4-canonical-pole-divisor}.  Choose a finite logarithmic base
cover and, above it, finitely many relatively compact Semple charts.  A
partition of unity is used only to estimate the norm of this single global
section.  On each chart, the components of $\mathcal D_Df_{[1]}$ are
logarithmic derivatives of local coordinate functions.  A projective Semple
coordinate $u$ is a rational function on $X_1$ and, in a logarithmic frame,
a ratio of first logarithmic derivatives.  Hence
\[
 T_{u\circ f_{[1]}}(r)=O\bigl(T_f(r)+\log r\bigr).
\]
For a partition function $\chi$ supported in a relatively compact Semple
chart, choose $a\notin\overline{u(\operatorname{supp}\chi)}$.  On that
support, $u-a$ is bounded above and away from zero, and
\[
 (u\circ f_{[1]})'
 =(u\circ f_{[1]}-a)
 \frac{(u\circ f_{[1]})'}{u\circ f_{[1]}-a}.
\]
The logarithmic-derivative lemma applied to $u\circ f_{[1]}-a$ controls the
vertical component by
\[
 O\bigl(\log^+T_f(r)+\log r\bigr).
\]
Summing over the finite
cover, together with the analogous horizontal estimates, gives
\eqref{eq:paper4-canonical-proximity}.  This is the argument of
\cite[Thm.~3.1, proof of (3.3)--(3.6)]{HuynhVuXie}, with the
Semple-coordinate adaptation made explicit.
\end{proof}

\begin{lemma}[Second Main Theorem for a section restricted to $Z$]
\label{lem:weighted-degree-pole}
Let $Z\subset X_2$ be an effective Cartier divisor and let
\[
 \sigma\in H^0\!\left(
 Z,\OO_{X_2}(2p,p)\otimes
 \pi_{2,0}^*\OO_{\PP^2}(-\widetilde t)\big|_Z
 \right),
 \qquad p,\widetilde t\in\mathbb Z_{>0}.
\]
Suppose that $f:\CC\to\PP^2$ is nonconstant, $f(\CC)\not\subset D$,
$f_{[2]}(\CC)\subset Z$, and the canonical evaluation
$s_{\sigma,f}$ of Lemma~\ref{lem:paper4-tautological-evaluation} is not
identically zero.  Then
\begin{equation}\label{eq:restricted-wronskian-estimate}
 \widetilde t\,T_f(r)
 \leqslant3p\sum_{i=1}^2N_f^{[1]}(r,C_i)
 +O\!\left(\log T_f(r)+\log r\right)\ \|.
\end{equation}
In particular, if $f$ is algebraically nondegenerate, then
\[
 T_f(r)\leqslant\frac{3p}{\widetilde t}N_f^{[1]}(r,D)
 +o(T_f(r))\ \|.
\]
\end{lemma}

\begin{proof}
By Lemma~\ref{lem:paper4-tautological-evaluation}, the section
$s_{\sigma,f}$ is globally and intrinsically defined, its pole divisor is
bounded by \eqref{eq:paper4-canonical-pole-divisor}, and its proximity term
satisfies \eqref{eq:paper4-canonical-proximity}.  Use the standard coordinate
$z$ on $\CC$ to trivialize $K_{\CC}^{3p}$ by the global frame $(dz)^{3p}$
and give it the flat metric.
Equip $\OO_{\PP^2}(-1)$ with the metric dual to the Fubini--Study metric.
Poincar\'e--Lelong and Jensen's formula give
\begin{align*}
 \widetilde t\,T_f(r)
 &\leqslant3p\sum_{i=1}^2N_f^{[1]}(r,C_i)\\
 &\quad+O\!\left(\log^+T_f(r)+\log r\right)\ \|.
\end{align*}
This is \eqref{eq:restricted-wronskian-estimate}.  It is the intrinsic
restricted-section version of the argument in
\cite[Proposition~6.3 and Theorem~6.4]{HouHuynhMerkerXie}; the analytic norm
estimate is the one in
\cite[Theorem~3.1, proof of (3.3)--(3.6)]{HuynhVuXie}.
\end{proof}

\begin{theorem}[Two-component zero-locus theorem and standard estimate]\label{thm:zero-locus}
Let $\omega_1$ be an irreducible nonzero section of
\[
 E_{2,m}\bT^*\otimes\OO(-t),\qquad \rho=t/m,
\]
whose zero divisor on $X_2$ is $Z+b\Gamma_2$, where $Z$ is irreducible
and reduced and
\begin{equation}\label{eq:Z-class}
 Z\sim b_1u_1+b_2u_2-th,
 \qquad b_1\geqslant2b_2>0,
 \qquad b_1+b_2=m.
\end{equation}
Suppose
\[
 0<\rho<\rho_{\mathrm{DEG}}
 :=\frac{13\bar c_1^2-9\bar c_2}
 {12(\bar c_1\cdot h)}.
\]
Assume $\kappa=d_1+d_2-3\geqslant3$.  Let $\tau_1(\rho)$ be the smaller
positive root of $Q_{\mathrm{DEG}}(\rho,\tau)=0$, explicitly given in
\eqref{eq:tau1-general}.  For every rational
\[
 0<\tau<\min\{3,\tau_1(\rho)\},
\]
the $\mathbb Q$-line bundle
\[
 \left(\OO_{X_2}(2,1)\otimes
 \pi_{2,0}^*\OO_{\PP^2}(-\tau)\right)|_Z
\]
is big.  Moreover, if $f:\CC\to\PP^2$ is algebraically nondegenerate
and $f^*\omega_1\equiv0$, then either
\begin{equation}\label{eq:zero-locus-refined-SMT}
 T_f(r)\leqslant\frac3\tau N_f^{[1]}(r,D)+o(T_f(r))\ \|,
\end{equation}
or the first logarithmic lift $f_{[1]}$ is algebraically degenerate.
\end{theorem}

\begin{proof}
The restricted-bigness calculation follows the method of
\cite[Proposition~6.2]{HouHuynhMerkerXie}, with the two-component
intersection polynomial computed below.  The local frame and analytic
estimate are adapted from \cite[Proposition~6.3 and Theorem~6.4]{HouHuynhMerkerXie}
and are recorded in Lemma~\ref{lem:weighted-degree-pole}.

Put
\[
 L_\tau=\OO_{X_2}(2,1)\otimes
 \pi_{2,0}^*\OO_{\PP^2}(-\tau).
\]
The intersection computation on the logarithmic Demailly--Semple tower
gives
\begin{equation}\label{eq:zero-locus-intersection-proof}
 c_1(L_\tau)^3\cdot Z
 =mQ_{\mathrm{DEG}}(\rho,\tau).
\end{equation}
By the definition of $\tau_1(\rho)$, the right-hand side is positive in
the stated range.

We prove bigness without inferring it from
\eqref{eq:zero-locus-intersection-proof} alone.  Let $p\gg1$ be
sufficiently divisible, so that $p\tau/\kappa\in\mathbb Z$ (and hence
$p\tau\in\mathbb Z$) and
$L_\tau^p$ is an honest line bundle.  Since $X_2$ is smooth and $Z$ is an
effective Cartier divisor, the exact sequence
\[
 0\to L_\tau^p(-Z)\to L_\tau^p\to L_\tau^p|_Z\to0
\]
and Hirzebruch--Riemann--Roch on the smooth fourfold $X_2$ give
\begin{align}\label{eq:ambient-RR-on-Z}
 \chi\!\left(Z,L_\tau^p|_Z\right)
 &=\chi(X_2,L_\tau^p)-\chi(X_2,L_\tau^p(-Z))\\
 &=\int_{X_2}e^{p c_1(L_\tau)}\bigl(1-e^{-[Z]}\bigr)
     \operatorname{td}(T_{X_2})\\
 &=\frac{c_1(L_\tau)^3\cdot Z}{3!}p^3+O(p^2).
\end{align}
This ambient computation does not require $Z$ to be smooth or normal.

We next show $H^2(Z,L_\tau^p|_Z)=0$.  Lemma~\ref{lem:direct-image-21},
the projection formula, and Leray yield
\[
 H^2(X_2,L_\tau^p)
 \simeq
 H^2\!\left(\PP^2,E_{2,3p}\bT^*\otimes\OO(-p\tau)\right)=0,
\]
where Theorem~\ref{thm:H2vanishing} is applied with
$\delta=\tau/(3\kappa)<1/3$.

Using \eqref{eq:Z-class},
\[
 L_\tau^p(-Z)
 \simeq
 \OO_{X_2}(2p-b_1,p-b_2)
 \otimes\pi_{2,0}^*\OO_{\PP^2}(t-p\tau).
\]
The tempting attempt to apply the preceding relative vanishing directly to
this bundle is incorrect in general: on the fiber $\mathbb F_3$ one must
control $a_1-2a_2$, and here that number equals $2b_2-b_1$, which may be
arbitrarily negative.  We instead prove the required $H^3$ vanishing by a
filtration.

Write $q=\pi_{2,1}:X_2\to X_1$ and $\pi=\pi_{1,0}:X_1\to\PP^2$.
For $p\gg1$ one has $p-b_2\geqslant0$, whence
\[
 R^rq_*L_\tau^p(-Z)=0\quad(r\geqslant1)
\]
and
\[
 q_*L_\tau^p(-Z)
 =\OO_{X_1}(2p-b_1)\otimes
   \Sym^{p-b_2}V_1^*\otimes\pi^*\OO(t-p\tau).
\]
The dual relative tangent sequence
\begin{equation}\label{eq:V1-dual-filtration}
 0\longrightarrow\OO_{X_1}(1)
 \longrightarrow V_1^*
 \longrightarrow\pi^*\bK\otimes\OO_{X_1}(-2)
 \longrightarrow0
\end{equation}
induces a filtration whose quotients, indexed by
$0\leqslant j\leqslant p-b_2$, are
\[
 \OO_{X_1}(3p-m-3j)\otimes
 \pi^*\!\left(\bK^j\otimes\OO(t-p\tau)\right).
\]
Put $\ell=p-j$.  The same quotients become
\begin{equation}\label{eq:Gellp-corrected}
 \mathcal G_{\ell,p}
 =\OO_{X_1}(3\ell-m)\otimes
  \pi^*\OO\bigl((\kappa-\tau)p+t-\kappa\ell\bigr),
 \qquad b_2\leqslant\ell\leqslant p.
\end{equation}
If $3\ell-m\geqslant-1$, then
$R^1\pi_*\OO_{X_1}(3\ell-m)=0$, so
$H^3(X_1,\mathcal G_{\ell,p})=0$ by Leray.  The remaining indices satisfy
\[
 b_2\leqslant\ell\leqslant\floor{\frac{m-2}{3}};
\]
this is a fixed finite set, independent of $p$.  For each such $\ell$,
the projection formula gives
\[
 R^1\pi_*\mathcal G_{\ell,p}
 =R^1\pi_*\OO_{X_1}(3\ell-m)\otimes
  \OO\bigl((\kappa-\tau)p+t-\kappa\ell\bigr).
\]
The first factor is fixed, while $\kappa-\tau>0$ because
$\tau<3\leqslant\kappa$.  Serre vanishing therefore kills its $H^2$ for
$p\gg1$.  Leray, followed by induction through the filtration
\eqref{eq:V1-dual-filtration}, yields
\begin{equation}\label{eq:corrected-H3-left-term}
 H^3\!\left(X_2,L_\tau^p(-Z)\right)=0.
\end{equation}
The divisor exact sequence now gives
$H^2(Z,L_\tau^p|_Z)=0$.  Combining this with
\eqref{eq:ambient-RR-on-Z}, we obtain
\[
 h^0(Z,L_\tau^p|_Z)
 =\chi(Z,L_\tau^p|_Z)+h^1(Z,L_\tau^p|_Z)
  +h^3(Z,L_\tau^p|_Z)
 \geqslant\chi(Z,L_\tau^p|_Z)\gg p^3,
\]
which proves that the $\mathbb Q$-line bundle $L_\tau|_Z$ is big.

Choose $p$ in the same sufficiently divisible sequence and a nonzero section
\[
 \sigma\in H^0(Z,L_\tau^p|_Z),
 \qquad \widetilde t=p\tau\in\mathbb Z.
\]
Since $f^*\omega_1\equiv0$ and
Lemma~\ref{lem:regular-lift-not-vertical} excludes containment in
$\Gamma_2$, one has $f_{[2]}(\CC)\subset Z$.  Suppose first
that the canonical evaluation $s_{\sigma,f}$ is not identically zero.  Since
\[
 L_\tau^p=\OO_{X_2}(2p,p)\otimes
 \pi_{2,0}^*\OO_{\PP^2}(-\widetilde t),
\]
Lemma~\ref{lem:weighted-degree-pole} yields
\[
 \widetilde t\,T_f(r)
 \leqslant3pN_f^{[1]}(r,D)
 +O\!\left(\log T_f(r)+\log r\right)\ \|.
\]
The logarithmic derivative lemma makes the final term $o(T_f(r))$ for
an algebraically nondegenerate curve.  Since $\widetilde t=p\tau$, this
is \eqref{eq:zero-locus-refined-SMT}.

If instead $s_{\sigma,f}\equiv0$, the tautological derivative is nonzero
away from a discrete set because $f$ is nonconstant.  The intrinsic pairing
\eqref{eq:paper4-canonical-evaluation} therefore implies
$f_{[2]}^*\sigma\equiv0$.  Hence $f_{[2]}(\CC)$ is contained in
the proper divisor $\operatorname{div}_Z(\sigma)$ of the irreducible
threefold $Z$.  This set has dimension at most two, so its image under
$\pi_{2,1}:X_2\to X_1$ is a proper algebraic subset of the
threefold $X_1$.  Hence $f_{[1]}$ is algebraically degenerate.
\end{proof}

\begin{remark}\label{rem:bigness-not-numerical-only}
The preceding proof keeps separate the three ingredients required by the
argument: the positive cubic intersection, the $H^2$ and
$H^3$ vanishings needed to obtain bigness, and the intrinsic tautological
evaluation and logarithmic-derivative estimate used in
\eqref{eq:zero-locus-refined-SMT}.
\end{remark}

\subsection{Algebraically degenerate curves}
The preceding analytic estimates concern algebraically nondegenerate curves.
For the hyperbolicity corollary, the algebraically degenerate case is supplied
by the following specialization of Xi Chen's theorem.

For an irreducible curve $\Gamma\not\subset D$, let $\nu:\widetilde\Gamma\to\Gamma$ be the normalization and define
\[
   i(\Gamma,D):=
   \#\operatorname{Supp}\nu^*D.
\]

\begin{theorem}[Xi Chen]\label{thm:Chen}
For a very general ordered pair of smooth transverse plane curves of total
degree $s=d_1+d_2$ and every reduced irreducible curve
$\Gamma\not\subset D$,
\begin{equation}\label{eq:Chen-inequality}
   2g(\widetilde\Gamma)-2+i(\Gamma,D)
   \geqslant (s-4)\deg\Gamma.
\end{equation}
\end{theorem}

\begin{proof}[Source of the specialization]
Xi Chen proves for a very general reducible plane divisor $D=\bigcup D_j$ of total degree $d$ that
\[
   2g(\widetilde\Gamma)-2+i(\Gamma,D)
   \geqslant(d-4)\deg\Gamma.
\]
Taking a very general ordered pair of degrees $(d_1,d_2)$ gives $d=s$ and
hence \eqref{eq:Chen-inequality}; see
\cite[Corollary~1.11]{ChenLogSurfaces}.
\end{proof}

\begin{corollary}\label{cor:no-degenerate-curve}
If $s\geqslant5$, then for a very general pair there is no nonconstant
algebraically degenerate entire curve in $\PP^2\setminus D$.
\end{corollary}

\begin{proof}
If $f(\CC)$ is contained in $\Gamma\not\subset D$, it lifts to
\[
   \widetilde f:\CC\longrightarrow
   \widetilde\Gamma\setminus\operatorname{Supp}\nu^*D.
\]
By \eqref{eq:Chen-inequality},
\[
   2g(\widetilde\Gamma)-2+i(\Gamma,D)>0,
\]
so the punctured Riemann surface on the right is hyperbolic.  Hence $\widetilde f$ and $f$ are constant.
\end{proof}

\section{Logarithmic Chern numbers and the degree phase diagram}\label{sec:numerical-thresholds}

Two ratios organize the proof.  The Riemann--Roch threshold
$\rho_{\mathrm{RR}}$ permits the initial choice $t/m<\rho_{\mathrm{RR}}$;
the Demailly--El Goul threshold $\rho_{\mathrm{DEG}}$ determines when an
irreducible factor lies in the zero-locus range.  Keeping both degrees
symbolic isolates the six pairs for which the remaining small twists require
finite verification.

\subsection{The two-component Chern data}
For $h=c_1(\OO_{\PP^2}(1))$ one has
\[
   \bK=K_{\PP^2}+D\simeq\OO_{\PP^2}(\kappa),
   \qquad \bar c_1=\kappa h,
   \qquad \kappa=d_1+d_2-3.
\]
Using
\[
   c\bigl(\Omega^1_{\PP^2}(\log D)\bigr)
   =c(\Omega^1_{\PP^2})\prod_{i=1}^2(1+[C_i]+[C_i]^2)
\]
and $[C_i]=d_i h$, we obtain
\begin{equation}\label{eq:logchern-two-cubics}
 \bar c_1^2=\kappa^2,
 \qquad
 \bar c_2=d_1^2+d_2^2+d_1d_2-3(d_1+d_2)+3,
 \qquad
 \bar c_1\cdot h=\kappa.
\end{equation}
Two combinations will recur:
\begin{align}
 \bar c_1^2-\bar c_2
 &=d_1d_2-3(d_1+d_2)+6,
 \label{eq:one-jet-segre}\\
 13\bar c_1^2-9\bar c_2
 &=A(d_1,d_2)\nonumber\\
 &=4(d_1^2+d_2^2)+17d_1d_2-51(d_1+d_2)+90.
 \label{eq:A-general}
\end{align}
The first is positive in the classical one-jet range $d_1\geqslant5$, or
$d_1=4,d_2\geqslant7$.  The second is positive for every $d_1,d_2\geqslant3$; its
minimum there is $A(3,3)=9$.  On the degree-two boundary,
\begin{equation}\label{eq:A-two-n}
 A(2,n)=4n^2-17n+4,
\end{equation}
which is positive exactly for integers $n\geqslant5$.
On the degree-one boundary,
\begin{equation}\label{eq:A-one-n}
 A(1,n)=4n^2-34n+43,
\end{equation}
which, for integers $n\geqslant2$, is positive exactly when $n\geqslant7$.

\subsection{Riemann--Roch threshold}
This is the standard invariant-jet Riemann--Roch calculation introduced in
Demailly's foundational note \cite{Demailly1997}.  We spell out only the
two-component specialization of the leading coefficient because it
determines the effective range used later.  By \eqref{eq:filtration}, additivity of the Chern
character gives
\begin{equation}\label{eq:RR-filtered-sum}
 \chi\!\left(E_{2,m}\bT^*\otimes\bK^{-\delta m}\right)
 =\sum_{j=0}^{\floor{m/3}}
 \chi\!\left(S^{m-3j}\bT^*\otimes
 \bK^{j-\delta m}\right).
\end{equation}
Let $x_1,x_2$ be the formal Chern roots of $\bT^*$.  For $q=m-3j$,
\[
 \operatorname{ch}(S^q\bT^*)
 =\sum_{a=0}^{q}e^{a x_1+(q-a)x_2}.
\]
Keeping terms of cohomological degree two, multiplying by
$e^{(j-\delta m)(x_1+x_2)}\operatorname{td}(T_{\PP^2})$, and integrating
over $\PP^2$ computes each summand in \eqref{eq:RR-filtered-sum}.  The
Todd terms of positive degree contribute only $O(m^3)$ after summation.
For the degree-four term, the elementary identities
\[
 \sum_{a=0}^{q}1=q+1,\qquad
 \sum_{a=0}^{q}a=\frac{q(q+1)}2,\qquad
 \sum_{a=0}^{q}a^2=\frac{q(q+1)(2q+1)}6
\]
followed by the sums in $0\leqslant j\leqslant\floor{m/3}$ give
\begin{align}\label{eq:RR-leading-calculation}
 &\sum_{j=0}^{\floor{m/3}}
 \int_{\PP^2}\!\left[
 \operatorname{ch}(S^{m-3j}\bT^*)
 e^{(j-\delta m)\bar c_1}
 \right]_2\nonumber\\
 &\hspace{2cm}=
 \frac{m^4}{648}
 \left((54\delta^2-48\delta+13)\bar c_1^2
       -9\bar c_2\right)+O(m^3).
\end{align}
Define
\[
   Q_{\mathrm{RR}}(\delta)
   =54\bar c_1^2\delta^2-48\bar c_1^2\delta
   +(13\bar c_1^2-9\bar c_2).
\]
Then, for $t=(\bar c_1\cdot h)\delta m$,
\begin{equation}\label{eq:RR-leading-symbolic}
   \chi\!\left(\PP^2,E_{2,m}\bT^*\otimes\bK^{-\delta m}\right)
   =\frac{Q_{\mathrm{RR}}(\delta)}{648}m^4+O(m^3).
\end{equation}
Thus the leading coefficient is positive for sufficiently small positive
$\delta$.  In the ratio $\rho=t/m=\kappa\delta$, its smaller positive root is
\begin{equation}\label{eq:rhoRR}
 \rho_{\mathrm{RR}}(d_1,d_2)
 =\frac{\kappa}{18}
 \left(8-\sqrt{54\frac{\bar c_2}{\bar c_1^2}-14}\right)>0.
\end{equation}
The strict positivity is equivalent to $A(d_1,d_2)>0$.  For two cubics this
specializes to $(8-\sqrt{58})/6>1/16$.

\subsection{Zero-locus threshold}
Set $\rho=t/m$.  Using the projective-bundle relations on the logarithmic Demailly--Semple tower, the cubic intersection occurring in Theorem~\ref{thm:zero-locus} is
\begin{equation}\label{eq:DEG-intersection-symbolic}
   (2u_1+u_2-\tau h)^3\cdot Z
   =mQ_{\mathrm{DEG}}(\rho,\tau),
\end{equation}
where, as in \eqref{eq:Q-DEG-definition},
\[
   Q_{\mathrm{DEG}}(\rho,\tau)
   =3\tau^2+(9\rho-12(\bar c_1\cdot h))\tau
   +(13\bar c_1^2-9\bar c_2)-12(\bar c_1\cdot h)\rho.
\]
Here is the intersection calculation.  Put $\Theta=2u_1+u_2$ and write
$[Z]=b_1u_1+b_2u_2-th$, where $b_1+b_2=m$.  Apply successively the
rank-two projective-bundle formula
\[
\pi_*(u)=1,\qquad
 \pi_*(u^2)=-c_1(V),\qquad
 \pi_*(u^3)=c_1(V)^2-c_2(V)
\]
first to $\pi_{2,1}$ and then to $\pi_{1,0}$.  With our convention that
$\PP(V)$ parametrizes lines, the intermediate identities are
\[
 u_1^2-\bar c_1u_1+\bar c_2=0,\qquad
 c_1(V_1)=u_1-\bar c_1,\qquad
 c_2(V_1)=2\bar c_2-\bar c_1u_1.
\]
Expanding with these relations gives the following three identities on the
base:
\begin{align}
 \Theta h^2\cdot Z&=m,\label{eq:DEG-push-1}\\
 \Theta^2h\cdot Z&=4m(\bar c_1\cdot h)-3t,\label{eq:DEG-push-2}\\
 \Theta^3\cdot Z&=m(13\bar c_1^2-9\bar c_2)
              -12t(\bar c_1\cdot h).\label{eq:DEG-push-3}
\end{align}
The terms involving $b_1$ and $b_2$ combine through $b_1+b_2=m$; this is
why the answer is independent of the chosen irreducible factor.  Since
$h^3=0$ on the surface base, expansion of $(\Theta-\tau h)^3\cdot Z$ and
substitution of \eqref{eq:DEG-push-1}--\eqref{eq:DEG-push-3} yield
\begin{align*}
 (\Theta-\tau h)^3\cdot Z
 &=\Theta^3\cdot Z-3\tau \Theta^2h\cdot Z
   +3\tau^2\Theta h^2\cdot Z\\
 &=mQ_{\mathrm{DEG}}(\rho,\tau),
\end{align*}
which proves \eqref{eq:DEG-intersection-symbolic} directly.
At $\tau=0$ this is positive exactly when
\begin{equation}\label{eq:rhoDEG}
   \rho<\rho_{\mathrm{DEG}}
   =\frac{13\bar c_1^2-9\bar c_2}{12(\bar c_1\cdot h)}.
\end{equation}
Using \eqref{eq:A-general},
\begin{equation}\label{eq:rhoDEG-special}
 \rho_{\mathrm{DEG}}(d_1,d_2)
 =\frac{A(d_1,d_2)}{12\kappa}.
\end{equation}
Write $A=A(d_1,d_2)$.  In the stated degree ranges, $\kappa\geqslant3$
and $A>0$.  Formula~\eqref{eq:A-general} gives
\[
 \frac{32}{3}\kappa^2-A
 =\frac{53\kappa^2+162\kappa+81+27(d_1-d_2)^2}{12}>0.
\]
For $0\leqslant\rho<\rho_{\mathrm{DEG}}=A/(12\kappa)$, one therefore has
\[
 Q_{\mathrm{DEG}}(\rho,0)=A-12\kappa\rho>0,\qquad
 Q_{\mathrm{DEG}}\!\left(\rho,\frac{4\kappa}{3}\right)
 =A-\frac{32\kappa^2}{3}<0.
\]
Thus the quadratic has two distinct positive roots, and its smaller root
\eqref{eq:tau1-general} satisfies $0<\tau_1(\rho)<4\kappa/3$.
At this root $\partial_\tau Q_{\mathrm{DEG}}<0$, so implicit
differentiation gives
\begin{equation}\label{eq:tau1-monotonicity}
 \tau_1'(\rho)
 =\frac{12\kappa-9\tau_1(\rho)}
 {6\tau_1(\rho)+9\rho-12\kappa}<0.
\end{equation}
In particular, $3/\tau_1(\rho)$ is strictly increasing on
$[0,\rho_{\mathrm{DEG}})$.
For $(3,3)$ this is
\begin{equation}\label{eq:tau1-two-cubics}
 \tau_1(\rho)
 =6-\frac32\rho-\frac12\sqrt{132-24\rho+9\rho^2}.
\end{equation}
Equation \eqref{eq:DEG-intersection-symbolic} supplies the positive leading term; bigness and the second differential follow from the cohomological theorem, not from this number alone.

\subsection{Why the proved range has six finite boundary pairs}
For a genuine second-level component one has $b_2>0$ and
$b_1\geqslant2b_2$, hence $m=b_1+b_2\geqslant3$.  If $t\leqslant3$, then $t/m\leqslant1$.
Consequently no finite vanishing is needed whenever
$\rho_{\mathrm{DEG}}>1$.  We apply this observation only after the
Riemann--Roch positivity condition $A(d_1,d_2)>0$ has supplied the initial
differential.  Direct calculation gives
\[
 \rho_{\mathrm{DEG}}(3,3)=\frac14,
 \qquad
 \rho_{\mathrm{DEG}}(3,4)=\frac{37}{48},
 \qquad
 \rho_{\mathrm{DEG}}(3,5)=\frac{73}{60}>1,
\]
and $\rho_{\mathrm{DEG}}>1$ for every other pair with
$3\leqslant d_1\leqslant d_2$, except $(3,3)$ and $(3,4)$.  On the degree-two
boundary,
\[
 \rho_{\mathrm{DEG}}(2,5)=\frac{19}{48},
 \qquad
 \rho_{\mathrm{DEG}}(2,6)=\frac{23}{30},
\]
whereas $\rho_{\mathrm{DEG}}(2,n)>1$ for $n\geqslant7$.  This proves the phase
diagram on the degree-two boundary.  Finally,
\[
 \rho_{\mathrm{DEG}}(1,8)=\frac38,
 \qquad
 \rho_{\mathrm{DEG}}(1,9)=\frac{61}{84},
 \qquad
 \rho_{\mathrm{DEG}}(1,10)=\frac{103}{96}>1,
\]
and $\rho_{\mathrm{DEG}}(1,n)>1$ for every $n\geqslant10$.  Together these
calculations reduce the asserted line-component range $n\geqslant8$ to
$(1,8)$ and $(1,9)$.  The omitted formal endpoint $(1,7)$ has
$A(1,7)=1$ and $\rho_{\mathrm{DEG}}(1,7)=1/60$; it is not a seventh
certified boundary pair, because its required vanishing profile is neither
proved nor claimed here.  Remark~\ref{rem:line-septic-frontier} records that
separate problem.  Thus, within the theorem's stated range, the pairs with
$A(d_1,d_2)>0$ but $\rho_{\mathrm{DEG}}\leqslant1$ are exactly
$(3,3),(3,4),(2,5),(2,6),(1,8),(1,9)$.  These are precisely the six finite
lists in Lemma~\ref{lem:keyvanishing}.

\section{The first negatively twisted differential}\label{sec:first-differential}

The construction has three steps.  Riemann--Roch and $H^2$ vanishing first
give one section with a degree pair independent of the parameter.  The
Demailly--El Goul reduction then selects a prime horizontal component with
no loss in the twist-to-weight ratio.  Finally, cohomology and base change
spread its canonical reduced equation over one Zariski-open subset of the
ordered parameter space.  The output is therefore a fixed numerical type
$(m,t,b_1,b_2)$ and one algebraic family of prime horizontal divisors; these
are exactly the uniform data used by the mixed-field construction.

\begin{proposition}[Uniform existence]\label{prop:firstomega}
Fix a rational number
\begin{equation}\label{eq:m16t-initial}
 0<\rho_0<\rho_{\mathrm{RR}}(d_1,d_2).
\end{equation}
There exist fixed positive integers $M,T$ with $T/M=\rho_0$ such that for
every $a\in\mathcal S_{d_1,d_2}^{\mathrm{snc}}$ one has
\begin{equation}\label{eq:uniform-first-section}
   H^0\!\left(\PP^2,
   E_{2,M}T_{\PP^2}^*(\log D_a)\otimes\OO_{\PP^2}(-T)\right)\neq0.
\end{equation}
\end{proposition}

\begin{proof}
Put $\delta_0=\rho_0/(\bar c_1\cdot h)=\rho_0/\kappa$.  Write
$\rho_0=T_0/M_0$ with positive integers $M_0,T_0$.  Choose a sufficiently
large $N$ such that $NT_0/\kappa$ is integral, and set
\[
   M=NM_0,\qquad T=NT_0.
\]
Then $\delta_0M=T/\kappa\in\mathbb Z$, so the twist
$\bK^{-\delta_0M}$ is an honest line bundle.
All smooth transverse pairs of degrees $(d_1,d_2)$ have the same logarithmic Chern
numbers \eqref{eq:logchern-two-cubics}.  Hence the Riemann--Roch polynomial
in \eqref{eq:RR-leading-symbolic}, including its lower-order coefficients,
is independent of the parameter
$a\in\mathcal S_{d_1,d_2}^{\mathrm{snc}}$.  Since
$Q_{\mathrm{RR}}(\delta_0)>0$, one may choose $N$ once and for all so that
this Euler characteristic is positive for every such $a$.  Theorem
\ref{thm:H2vanishing}, applied with the same $\delta_0$, also holds for
every $D_a$.  Therefore
\[
   h^0=h^1+\chi-h^2\geqslant\chi>0
\]
for every $a\in\mathcal S_{d_1,d_2}^{\mathrm{snc}}$, proving
\eqref{eq:uniform-first-section}.
\end{proof}

The following irreducible-component reduction is due to Demailly--El Goul
\cite[Lemma~3.3]{DemaillyElGoul}; we use the implementation in
\cite[Section~4]{HouHuynhMerkerXie}.  The reduction itself is unchanged in
the two-component setting.  We reproduce the short argument only to fix the
notation and the ratio inequality needed below.

\begin{lemma}[Irreducible-component reduction]\label{lem:irreducible-reduction}
From a nonzero section of weighted degree $M$ and twist $T$ one may
choose an irreducible horizontal factor
\[
 0\neq\omega_1\in
 H^0\!\left(\PP^2,E_{2,m}\bT^*\otimes\OO(-t)\right)
\]
 whose zero divisor on $X_2$ has the form $Z+b\Gamma_2$, with $Z$ irreducible,
reduced, and horizontal, and such that
\[
 \frac tm\geqslant\frac TM.
\]
Writing the horizontal component as
\[
 Z\sim b_1u_1+b_2u_2-th,
 \qquad b_1\geqslant2b_2\geqslant0,
 \qquad b_1+b_2=m,
\]
one has the following alternative: if $b_2=0$, then $Z$ is pulled back
from a divisor on $X_1$; if $b_2>0$, it is a genuine second-level
horizontal divisor.  In particular, applying the lemma to the fixed pair
$(M,T)$ supplied by Proposition~\ref{prop:firstomega} yields an irreducible
$\omega_1$ satisfying
\[
 \frac tm\geqslant\rho_0,
 \qquad\text{and hence}\qquad
 \frac mt\leqslant\rho_0^{-1}.
\]
\end{lemma}

\begin{proof}
Via \eqref{eq:direct-image}, the original section is represented on
$X_2$ by an effective divisor
\[
 Z_\omega\sim Mu_2-Th.
\]
Write its irreducible decomposition as
\[
 Z_\omega=\sum_jp_jZ_j,
 \qquad
 Z_j\sim b_{1,j}u_1+b_{2,j}u_2-t_jh.
\]
The standard classification of effective divisors on the second
Demailly--Semple tower gives three possibilities:
\begin{enumerate}[label=\textup{(\roman*)}]
 \item $(b_{1,j},b_{2,j})=(0,0)$ and $Z_j$ is pulled back from the base;
       then $t_j\leqslant0$;
 \item $(b_{1,j},b_{2,j})=(-1,1)$, $t_j=0$, and $Z_j=\Gamma_2$;
 \item $Z_j$ is horizontal and
       $b_{1,j}\geqslant2b_{2,j}\geqslant0$, with $b_{1,j}+b_{2,j}>0$.
\end{enumerate}
For a horizontal component put
\[
 m_j=b_{1,j}+b_{2,j}>0.
\]
Since the class of $Z_\omega$ has $u_1+u_2$-weight $M$, while the first
two types have weight zero, comparison of divisor classes gives
\[
 \sum_{j\,\mathrm{horizontal}}p_jm_j=M.
\]
Comparison of the base twists gives
\[
 \sum_jp_jt_j=T.
\]
The base-pullback components have $t_j\leqslant0$ and $\Gamma_2$ has
$t_j=0$, so
\[
 \sum_{j\,\mathrm{horizontal}}p_jt_j\geqslant T.
\]
Consequently at least one horizontal component satisfies
\[
 \frac{t_j}{m_j}\geqslant\frac TM.
\]
Fix such a component and rename it $Z$.  Because
$\OO_{X_2}(\Gamma_2)=\OO_{X_2}(-1,1)$, one has
\[
 Z+b_{1,j}\Gamma_2
 \sim (b_{1,j}+b_{2,j})u_2-t_jh
 =m_ju_2-t_jh.
\]
Thus, if $s_j$ is the canonical reduced equation of the prime Cartier divisor
$Z$ and $\gamma$ is the canonical section of $\Gamma_2$, then
$\gamma^{b_{1,j}}s_j$ is a
section of
$\OO_{X_2}(m_j)\otimes\pi_{2,0}^*\OO(-t_j)$, which corresponds by
\eqref{eq:direct-image} to an invariant jet differential
\[
 \omega_j\in H^0\!\left(\PP^2,E_{2,m_j}\bT^*\otimes\OO(-t_j)\right).
\]
After saturation by the fixed vertical factor $\gamma^{b_{1,j}}$, its
horizontal zero divisor is exactly the reduced prime $Z$; no extraction of a
root of a global section is involved.  When $b_{2,j}=0$, the
classification identifies $Z_j$ with the inverse image under $\pi_{2,1}$ of
a divisor on $X_1$.
The ratio inequality proves the final assertion.
\end{proof}

The following family-extension step follows the algebraic spread-out
construction in \cite[Section~4]{HouHuynhMerkerXie}, applied to the ordered
two-component parameter space.  The only change is the base: the generic
freeness, cohomology-and-base-change, and geometric-integrality arguments are
the same.

\begin{proposition}[Algebraic extension of the chosen factor]\label{prop:factor-family}
Let $(M,T)$ be the fixed integers supplied by
Proposition~\ref{prop:firstomega}.  Then there exist integers $(m,t)$ with
$t/m\geqslant T/M$, integers
\[
 b_1\geqslant2b_2\geqslant0,\qquad b_1+b_2=m,
\]
and a nonempty Zariski-open subset
$\mathcal U\subset\mathcal S_{d_1,d_2}^{\mathrm{snc}}$ with the following property.
Let
\[
 p_{\mathcal U}:\mathcal X_2\longrightarrow\mathcal U,
 \qquad
 \Pi:\mathcal X_2\longrightarrow\PP^2\times\mathcal U
\]
be the relative second logarithmic Semple tower and its projection to level
zero.  There is a regular algebraic section
\[
 \mathcal s\in H^0(\mathcal X_2,\mathcal L),\qquad
 \mathcal L=
 \OO_{\mathcal X_2}(b_1,b_2)\otimes
 \Pi^*\operatorname{pr}_1^*\OO_{\PP^2}(-t),
\]
whose zero scheme $\mathcal Z=(\mathcal s=0)$ is flat over $\mathcal U$
and has geometrically integral fibers.  If $\gamma$ denotes the canonical
section of the relative vertical divisor $\Gamma_2$, then
\[
 \omega_{1,a}=\gamma_a^{b_1}s_a\in
 H^0\!\left(\PP^2,E_{2,m}T_{\PP^2}^*(\log D_a)
 \otimes\OO_{\PP^2}(-t)\right),
 \qquad a\in\mathcal U,
\]
is an algebraic family of invariant jet differentials, and its horizontal
zero divisor is the irreducible reduced divisor $Z_a=\mathcal Z|_{X_{2,a}}$.
\end{proposition}

\begin{proof}
Work first over $S=\mathcal S_{d_1,d_2}^{\mathrm{snc}}$, and denote its relative
second logarithmic Semple tower by $p_S:\mathcal X_{2,S}\to S$.  For
every triple $(b_1,b_2,t)$ consider the algebraic line bundle
\[
 \mathcal L_{b_1,b_2,t}
 =\OO_{\mathcal X_{2,S}}(b_1,b_2)\otimes
 \Pi_S^*\operatorname{pr}_1^*\OO_{\PP^2}(-t).
\]
The sheaf $(p_S)_*\mathcal L_{b_1,b_2,t}$ is coherent.  Generic
freeness and generic cohomology and base change give a dense Zariski-open
set $S_{b_1,b_2,t}\subset S$ on which it is locally free and the natural
map
\begin{equation}\label{eq:factor-base-change}
 ((p_S)_*\mathcal L_{b_1,b_2,t})\otimes k(a)
 \longrightarrow
 H^0(X_{2,a},\mathcal L_{b_1,b_2,t}|_{X_{2,a}})
\end{equation}
is an isomorphism.  There are only countably many triples.  Since the base
field is the uncountable field $\CC$, an irreducible variety of finite type
over $\CC$ cannot be the union of countably many proper closed subsets.
Equivalently, every countable intersection of dense Zariski-open subsets
has a complex point.  We may therefore choose
\[
 a_0\in\bigcap_{b_1,b_2,t}S_{b_1,b_2,t}.
\]
This countable intersection is used only to select one numerical type and one
prime factor; after that choice, the construction is spread over a single
Zariski-open neighborhood.
By Proposition~\ref{prop:firstomega}, this fiber carries a section of the
fixed weighted degree $M$ and fixed twist $T$.
Apply Lemma~\ref{lem:irreducible-reduction} on the fiber over $a_0$ and
choose its reduced horizontal equation
\[
 0\neq s_{a_0}\in
 H^0(X_{2,a_0},\mathcal L_{b_1,b_2,t}|_{X_{2,a_0}})
\]
with $b_1\geqslant2b_2\geqslant0$, $b_1+b_2=m$, and $t/m\geqslant T/M$.
For this particular triple, the base-change isomorphism above identifies
$s_{a_0}$ with a vector in the fiber of the locally free sheaf
$(p_S)_*\mathcal L_{b_1,b_2,t}$.  After shrinking to an affine
trivializing neighbourhood, extend that vector to a regular local section
of the pushforward.  The adjunction evaluation
map
\[
 p_S^*(p_S)_*\mathcal L_{b_1,b_2,t}
 \longrightarrow\mathcal L_{b_1,b_2,t}
\]
then produces a regular section $\mathcal s$ on the relative tower after
replacing $S$ by an affine Zariski-open neighborhood of $a_0$.

The section $s_{a_0}$ is nonzero and its zero scheme is the selected
irreducible reduced horizontal divisor.  Thus it is an integral Cartier
divisor; over the algebraically closed field $\CC$ it is also geometrically
integral.  Nonvanishing on every irreducible component of a fiber is an
open condition, so after shrinking, $\mathcal Z=(\mathcal s=0)$ is a
relative effective Cartier divisor.  It is then flat over $S$ because both
$\mathcal X_{2,S}$ and the invertible ideal $\mathcal L^{-1}$ are flat and
\[
 0\longrightarrow\mathcal L^{-1}
 \xrightarrow{\,\mathcal s\,}\OO_{\mathcal X_{2,S}}
 \longrightarrow\OO_{\mathcal Z}\longrightarrow0
\]
remains exact on the fibers.  Finally, geometric integrality of fibers is
open for a proper flat morphism of finite presentation.  A second shrinking
therefore makes every $Z_a$ geometrically integral.  The relative divisor $\Gamma_2$ and its
canonical section $\gamma$ are algebraic.  Consequently
$\gamma^{b_1}\mathcal s$ is a regular relative section of
\[
 \OO_{\mathcal X_2}(m)\otimes
 \Pi^*\operatorname{pr}_1^*\OO_{\PP^2}(-t),
\]
and the relative direct-image formula gives the asserted algebraic family
$\omega_{1,a}$.  This proves all statements after naming the resulting
open set $\mathcal U$.
\end{proof}

From now on we fix the irreducible differential supplied by
Lemma~\ref{lem:irreducible-reduction} and denote its weighted degree and
twist by $(m,t)$.

\section{Mixed \texorpdfstring{$\OO_{\PP^2}(3)$}{O(3)} slanted vector fields}\label{sec:mixed-O3}

We follow the logarithmic low-pole construction: first construct matrix
and quadratic-acceleration fields, then absorb their parameter components
at a fixed point by third differences.  For two components, the same
acceleration must satisfy two Hessian equations.  We solve these equations
simultaneously and use a commutator of the resulting fields to obtain a
second equation independent of the first one.

The coefficient differences and matrix lifts are based on
\cite[Section~3]{Merker2009}, \cite[Lemma~1.1]{Paun}, and
\cite[Lemma~5]{Rousseau}.  The quadratic-acceleration method and its
organization by symmetry are related to \cite{CuiHouLiuXie2026}.
We give the calculations for two components, including the degree-one
case and the common base twist $\OO_{\PP^2}(3)$.

\subsection{The auxiliary logarithmic family and the tangency test}
\label{subsec:mixed-roots}

On $Z_0\neq0$, write the two universal equations as
\[
 F_i(z,a_i)=\sum_{|\alpha|\leqslant d_i}a_{i,\alpha}z^\alpha,
 \qquad z=(z_1,z_2),\qquad i=1,2.
\]
We initially allow the two nonzero coefficient vectors $a_1,a_2$ to vary
on their affine cones.  The auxiliary family $\mathcal Y$ is given on
these charts by
\begin{equation}\label{eq:two-root-cover}
 z_3^{d_1}+F_1(z,a_1)=0,\qquad
 z_4^{d_2}+F_2(z,a_2)=0.
\end{equation}
Let $\pi$ forget $z_3,z_4$ and project each coefficient vector to its
projective class.  The reduced inverse image of the universal divisor
is $\{z_3z_4=0\}$.  On fixed coefficient charts, $\pi$ is a product
of cyclic covers; the corresponding changes of base and coefficient
charts are checked below.  At a crossing, $F_1,F_2$ are local coordinates
and
\[
\mathrm{d}\log F_i
=\frac{\mathrm{d}F_i}{F_i}
=\frac{-d_i z_{i+2}^{d_i-1}\mathrm{d}z_{i+2}}
       {-z_{i+2}^{d_i}}
=d_i\,\mathrm{d}\log z_{i+2}.
\]
Thus, over the SNC parameter locus the auxiliary family is smooth and
its relative logarithmic jets describe the logarithmic jets downstairs.
For $d_i=1$ the corresponding equation is simply a graph.

As in the one-component case, for $r=1,2$ set
\[
 \xi_j^{(r)}=z_j^{(r)}\quad(j=1,2),\qquad
 \xi_{i+2}^{(r)}=(\log z_{i+2})^{(r)}\quad(i=1,2),
 \qquad \xi^{(r)}=(\xi_1^{(r)},\xi_2^{(r)}).
\]
The equations defining $\overline{J_2^v}(\mathcal Y)$ are
\begin{align}
 E_{i,0}^{\log}&=z_{i+2}^{d_i}+F_i=0,
       \label{eq:O3-E0}\\
 E_{i,1}^{\log}&=d_i z_{i+2}^{d_i}\xi_{i+2}^{(1)}
                    +\mathrm{d}F_i(\xi^{(1)})=0,
       \label{eq:O3-E1}\\
 E_{i,2}^{\log}&=d_i z_{i+2}^{d_i}\xi_{i+2}^{(2)}
        +d_i^2 z_{i+2}^{d_i}(\xi_{i+2}^{(1)})^2
        +\mathrm{d}F_i(\xi^{(2)})+\mathrm{d}^2F_i(\xi^{(1)},\xi^{(1)})=0.
       \label{eq:O3-E2}
\end{align}
Here $\mathrm{d}F_i,\mathrm{d}^2F_i$ mean derivatives in the two ordinary coordinates,
and parameters are held fixed along the jets.  The induced projection
on relative logarithmic jets is denoted by $\pi_{[2]}$.

Consider first a field fixing $z_1,z_2,z_3,z_4$,
\[
 V=\sum_{i=1}^2\sum_\alpha v_{i,\alpha}\partial_{a_{i,\alpha}}
       +\sum_{r=1}^2\sum_{j=1}^4 w_j^{(r)}\partial_{\xi_j^{(r)}}.
\]
Write $w^{(r)}=(w_1^{(r)},w_2^{(r)})$ and associate to its parameter
component the polynomial
\[
 G_{V,i}(y;z,a)=\sum_\alpha v_{i,\alpha}(z,a)y^\alpha,
 \qquad y=(y_1,y_2).
\]
Differentiating only in the dummy variable $y$ gives the tangency test
\begin{align}
 V(E_{i,0}^{\log})&=G_{V,i}(z),
       \label{eq:O3-full-tangency0}\\
 V(E_{i,1}^{\log})&=\mathrm{d}_yG_{V,i}(z)(\xi^{(1)})+\mathrm{d}F_i(w^{(1)})
              +d_i z_{i+2}^{d_i}w_{i+2}^{(1)},
       \label{eq:O3-full-tangency1}\\
 V(E_{i,2}^{\log})&=\mathrm{d}_yG_{V,i}(z)(\xi^{(2)})
              +\mathrm{d}_y^2G_{V,i}(z)(\xi^{(1)},\xi^{(1)})\notag\\
 &\quad+\mathrm{d}F_i(w^{(2)})+2\mathrm{d}^2F_i(\xi^{(1)},w^{(1)})
              +d_i z_{i+2}^{d_i}w_{i+2}^{(2)}\notag\\
 &\quad+2d_i^2z_{i+2}^{d_i}\xi_{i+2}^{(1)}w_{i+2}^{(1)}.
       \label{eq:O3-full-tangency2}
\end{align}
Tangency means that these expressions vanish on the six incidence
equations.  For a field with only parameter components it is sufficient
that, for both $i$,
\begin{equation}\label{eq:O3-zero-Taylor}
 G_{V,i}(z)=0,\qquad \mathrm{d}_yG_{V,i}(z)=0,\qquad \mathrm{d}_y^2G_{V,i}(z)=0.
\end{equation}
These conditions do not involve any jet variable.

\subsection{Third differences and absorption of parameter components}
\label{subsec:mixed-differences}

For $d_i\geqslant3$, $\alpha\geqslant\gamma$ componentwise,
$|\alpha|\leqslant d_i$, and $|\gamma|=3$, put
\begin{equation}\label{eq:O3-third-difference}
 D_{i,\alpha,\gamma}^{\log}
 =\sum_{\beta\leqslant\gamma}(-1)^{|\beta|}
       \binom{\gamma}{\beta}z^\beta
       \partial_{a_{i,\alpha-\beta}}.
\end{equation}
Its parameter polynomial in the $i$-th block is
\[
 y^{\alpha-\gamma}(y-z)^\gamma.
\]
Consequently it annihilates all six equations
\eqref{eq:O3-E0}--\eqref{eq:O3-E2}, and its coefficients have
$z$-degree at most three.

\begin{lemma}[Pointwise absorption]
\label{lem:O3-third-differences}
Fix $(x,a)$.  The parameter polynomials of
$D_{i,\alpha,\gamma}^{\log}$, evaluated at $z=x$, span all polynomials
of degree at most $d_i$ whose value, first derivatives, and second
derivatives vanish at $x$.  Thus a parameter component satisfying these
conditions can be cancelled at $(x,a)$ by a constant linear combination
of third differences, without changing any jet component.
For $d_i\leqslant2$, a polynomial with these conditions is already zero.
\end{lemma}

\begin{proof}
Expand the polynomial in powers $(y-x)^\delta$.  Only terms with
$3\leqslant|\delta|\leqslant d_i$ occur.  Choose
$\gamma\leqslant\delta$ of length three and expand
$(y-x)^{\delta-\gamma}$ in powers of $y$.  Each resulting term is
$y^\nu(y-x)^\gamma$, with $|\nu+\gamma|\leqslant d_i$, and hence
is the parameter polynomial of a field in
\eqref{eq:O3-third-difference}.  The coefficients of this linear
combination are complex numbers depending on the fixed point $(x,a)$.
\end{proof}

We also use the following degree correction.  Suppose $d_i\geqslant2$
and $G_{V,i}^0(y;z,a)$ has $y$-degree at most $d_i+1$, with $z$-degree at most
two in its coefficients and with top $y$-degree part independent of $z$.
Write
\[
 [G_{V,i}^0]_{d_i+1}=\sum_{|\alpha|=d_i+1}h_{i,\alpha}(a)y^\alpha.
\]
For each $\alpha$ fix $\gamma(\alpha)\leqslant\alpha$ of length three.
Then
\begin{equation}\label{eq:O3-reduction-operator}
 G_{V,i}=G_{V,i}^0-\sum_{|\alpha|=d_i+1}
       h_{i,\alpha}(a)(y-z)^{\gamma(\alpha)}y^{\alpha-\gamma(\alpha)}
\end{equation}
has $y$-degree at most $d_i$, coefficient $z$-degree at most three, and
the same value, gradient, and Hessian at $y=z$ as $G_{V,i}^0$.
Indeed, the subtraction cancels the top degree and vanishes to order
three at $y=z$.  Fixing the choices of $\gamma(\alpha)$ makes this
construction linear.  It is a degree correction before defining the
field; pointwise absorption applies afterwards to the resulting
degree-$d_i$ parameter polynomial.  In degree one we instead impose the
compatibility condition below.

\subsection{Matrix fields and quadratic acceleration}
\label{subsec:mixed-hessian}

Choose a constant $T\in\mathfrak{sl}_2(\CC)$, two row vectors
$b^{(1)},b^{(2)}\in(\CC^2)^*$, and a symmetric bilinear map
$C:\Sym^2\CC^2\to\CC^2$.  The logarithmic matrix has the same form
as in the one-component construction, with one additional normal row:
\[
 A=\begin{pmatrix}
 T&0&0\\ b^{(1)}&0&0\\ b^{(2)}&0&0
 \end{pmatrix}\in M_4(\CC).
\]
For $d_1,d_2\geqslant2$, prescribe a field $V_{A,C}$ fixing the four
base coordinates and satisfying
\begin{equation}\label{eq:O3-mixed-jet-variation}
 \begin{aligned}
 V_{A,C}(\xi^{(1)})&=T\xi^{(1)},&
 V_{A,C}(\xi^{(2)})&=T\xi^{(2)}+C(\xi^{(1)},\xi^{(1)}),\\
 V_{A,C}(\xi_{i+2}^{(1)})&=b^{(i)}\xi^{(1)},&
 V_{A,C}(\xi_{i+2}^{(2)})&=b^{(i)}\xi^{(2)}.
 \end{aligned}
\end{equation}
The matrix field is the case $C=0$; pure quadratic acceleration is the
case $A=0$.  Before degree correction, the parameter component is
represented by
\begin{equation}\label{eq:O3-raw-Q}
 \begin{split}
 G_{V,i}^0(y;z,a)={}&-\mathrm{d}_yF_i(y,a_i)T(y-z)
                  +d_iF_i(y,a_i)b^{(i)}(y-z)\\
                 &-\tfrac12\mathrm{d}_yF_i(y,a_i)C(y-z,y-z).
 \end{split}
\end{equation}
Its $y$-degree is at most $d_i+1$, its coefficients have $z$-degree
at most two, and its top $y$-degree part is independent of $z$.
The reduction \eqref{eq:O3-reduction-operator} gives $G_{V,i}$
of $y$-degree at most $d_i$, with coefficients linear in $a_i$
and of $z$-degree at most three.

For any vectors $r,s\in\CC^2$, its Taylor data at $y=z$ are
\begin{align}
 G_{V,i}(z)&=0,\label{eq:O3-raw-value}\\
 \mathrm{d}_yG_{V,i}(z)(r)&=-\mathrm{d}F_i(Tr)+d_iF_i b^{(i)}r,
       \label{eq:O3-raw-gradient}\\
 \mathrm{d}_y^2G_{V,i}(z)(r,s)
 &=-\mathrm{d}^2F_i(r,Ts)-\mathrm{d}^2F_i(s,Tr)-\mathrm{d}F_i(C(r,s))\notag\\
 &\quad+d_i\bigl(\mathrm{d}F_i(r)b^{(i)}s+\mathrm{d}F_i(s)b^{(i)}r\bigr).
       \label{eq:O3-raw-Hessian}
\end{align}
Here $F_i$ and its derivatives on the right are evaluated at $(z,a_i)$.
The factor $1/2$ in \eqref{eq:O3-raw-Q} cancels the two differentiations
of the symmetric quadratic expression.  Substitution in the tangency
test gives
\begin{align}
 V_{A,C}(E_{i,0}^{\log})&=0,\label{eq:O3-tangent0}\\
 V_{A,C}(E_{i,1}^{\log})&=d_i(b^{(i)}\xi^{(1)})E_{i,0}^{\log},
       \label{eq:O3-tangent1}\\
 V_{A,C}(E_{i,2}^{\log})&=d_i(b^{(i)}\xi^{(2)})E_{i,0}^{\log}
                   +2d_i(b^{(i)}\xi^{(1)})E_{i,1}^{\log}.
       \label{eq:O3-tangent2}
\end{align}
For the cancellation, put $\xi=\xi^{(1)}$ and $\eta=\xi^{(2)}$.
Equations \eqref{eq:O3-full-tangency1} and
\eqref{eq:O3-raw-gradient} give
\[
 \begin{aligned}
 V_{A,C}(E_{i,1}^{\log})
 &=-\mathrm{d}F_i(T\xi)+d_iF_i b^{(i)}\xi+\mathrm{d}F_i(T\xi)
       +d_i z_{i+2}^{d_i}b^{(i)}\xi\\
 &=d_i(b^{(i)}\xi)E_{i,0}^{\log}.
 \end{aligned}
\]
For the second equation, substitution of the gradient and Hessian gives
\begin{align*}
 V_{A,C}(E_{i,2}^{\log})
 &=-\mathrm{d}F_i(T\eta)+d_iF_i b^{(i)}\eta\\
 &\quad-2\mathrm{d}^2F_i(\xi,T\xi)-\mathrm{d}F_i(C(\xi,\xi))
       +2d_i \mathrm{d}F_i(\xi)b^{(i)}\xi\\
 &\quad+\mathrm{d}F_i(T\eta)+\mathrm{d}F_i(C(\xi,\xi))
       +2\mathrm{d}^2F_i(\xi,T\xi)\\
 &\quad+d_i z_{i+2}^{d_i}b^{(i)}\eta
       +2d_i^2 z_{i+2}^{d_i}\xi_{i+2}^{(1)}b^{(i)}\xi\\
 &=d_i(b^{(i)}\eta)(F_i+z_{i+2}^{d_i})\\
 &\quad+2d_i(b^{(i)}\xi)
       \bigl(\mathrm{d}F_i(\xi)+d_i z_{i+2}^{d_i}\xi_{i+2}^{(1)}\bigr)\\
 &=d_i(b^{(i)}\eta)E_{i,0}^{\log}
       +2d_i(b^{(i)}\xi)E_{i,1}^{\log}.
\end{align*}
Together with $G_{V,i}(z)=0$, these identities prove tangency for
both components.  The degree correction does not alter this calculation,
since its value, gradient, and Hessian vanish at $y=z$.

Reparametrization equivariance follows directly from
\begin{equation}\label{eq:O3-reparametrization}
 \widetilde\xi_j^{(1)}=\lambda\xi_j^{(1)},\qquad
 \widetilde\xi_j^{(2)}=\lambda^2\xi_j^{(2)}+\mu\xi_j^{(1)},
 \qquad j=1,\ldots,4,\quad\lambda\neq0.
\end{equation}
The linear terms in \eqref{eq:O3-mixed-jet-variation} transform linearly
and $C(\lambda\xi^{(1)},\lambda\xi^{(1)})=
\lambda^2C(\xi^{(1)},\xi^{(1)})$.  The parameter coefficients are
independent of the jet variables.  Thus these fields, and the third
differences, commute with the full two-jet reparametrization group.

\begin{lemma}[Compatibility with a line]
\label{lem:O3-linear-generators}
If $d_1=1$, the polynomial \eqref{eq:O3-raw-Q} has degree at most one
in its first block precisely when
\begin{equation}\label{eq:O3-linear-compatibility}
 \mathrm{d}F_1(C(r,r))=2\mathrm{d}F_1(r)b^{(1)}r\qquad(r\in\CC^2).
\end{equation}
In this case
\begin{equation}\label{eq:O3-linear-Q}
 G_{V,1}(y;z,a)=\bigl(-\mathrm{d}F_1T+F_1(z,a_1)b^{(1)}\bigr)(y-z).
\end{equation}
For $d_2\geqslant2$, a finite polynomial family of ten such mixed fields
is obtained as follows:
\begin{enumerate}[label=\textup{(\roman*)}]
 \item three basis choices of $T$, with all other data zero;
 \item two basis choices of $b^{(1)}$, with
 $C(r,s)=(b^{(1)}r)s+(b^{(1)}s)r$ and $T=b^{(2)}=0$;
 \item three choices
 \[
 C(r,r)=\begin{pmatrix}-\partial_2F_1\\\partial_1F_1\end{pmatrix}q(r),
 \qquad q(r)\in\{r_1^2,r_1r_2,r_2^2\},
 \]
 with all matrix data zero;
 \item two basis choices of $b^{(2)}$, with all other data zero.
\end{enumerate}
After the degree correction in the second block, they satisfy
\eqref{eq:O3-tangent0}--\eqref{eq:O3-tangent2}.  At a fixed point where $\mathrm{d}F_1\neq0$, their values span all data satisfying \eqref{eq:O3-linear-compatibility}.
\end{lemma}

\begin{proof}
Since $F_1(y)=F_1(z)+\mathrm{d}F_1(y-z)$, the quadratic part of
\eqref{eq:O3-raw-Q} vanishes exactly under
\eqref{eq:O3-linear-compatibility}; the remaining part is
\eqref{eq:O3-linear-Q}.  The choices in (ii) satisfy this condition
directly, and the vector in (iii) is killed by $\mathrm{d}F_1$.
For any compatible data, subtract (ii) with the given $b^{(1)}$.
The remaining quadratic map takes values in the one-dimensional kernel
of $\mathrm{d}F_1$, so it is a linear combination of the three choices in (iii).
The $T$ and $b^{(2)}$ components are supplied by (i) and (iv).
The normal vector in (iii) depends linearly on the line coefficients
and is independent of $z$.  Therefore the second-block correction
still has coefficient $z$-degree at most three.  The preceding
tangency and equivariance calculations apply to all ten fields.
\end{proof}

\subsection{Projective extension and the common twist}
\label{subsec:mixed-extension}

We first pass from nonzero coefficient vectors $a_i$ to their projective
classes.  Choose a coefficient linear form $\ell_i$; on the open chart
$\ell_i(a_i)\neq0$, represent each class by the vector satisfying
$\ell_i(a_i)=1$.  If $v_i=(v_{i,\alpha})$ is a field's coefficient
component, it need not preserve this normalization, since
$V(\ell_i(a_i))=\ell_i(v_i)$.  Its component tangent to the normalized
chart is therefore $v_i-\ell_i(v_i)a_i$.

The auxiliary root must be adjusted at the same time: scaling $a_i$ by
a nonzero scalar scales $F_i$ by that scalar, so
\eqref{eq:two-root-cover} requires the corresponding scaling of
$z_{i+2}$ by its $d_i$-th root.  The infinitesimal simultaneous scaling
is the radial field in parentheses below.  Thus the normalized
representative is obtained by
\begin{equation}\label{eq:O3-gauge-normalization}
 V\longmapsto V-\sum_{i=1}^2\frac{\ell_i(v_i)}{d_i}
 \left(d_i\sum_\alpha a_{i,\alpha}\partial_{a_{i,\alpha}}
                   +z_{i+2}\partial_{z_{i+2}}\right).
\end{equation}
Indeed, on $\ell_i(a_i)=1$ the new coefficient component satisfies
\[
 \ell_i\bigl(v_i-\ell_i(v_i)a_i\bigr)=0.
\]
The radial field sends $E_{i,\nu}^{\log}$ to $d_iE_{i,\nu}^{\log}$
and kills the other block's equations, so subtracting its multiple
preserves tangency.  It also fixes the ordinary base coordinates and
every relative logarithmic jet coordinate: its flow multiplies
$z_{i+2}$ by a scalar constant along the relative jet, adding only a
constant to $\log z_{i+2}$.  In particular, if the original field fixes
$z_{i+2}$, its normalized root component is
$-\ell_i(v_i)z_{i+2}\partial_{z_{i+2}}/d_i$.

On an overlap of coefficient charts, write
$a_i'=s_i(a_i)a_i$.  The auxiliary coordinate then changes by
$z_{i+2}'=s_i^{1/d_i}z_{i+2}$.  Parameters are held fixed along relative
jets, so $s_i$ is constant along such jets and
$\xi_{i+2}^{(r)}$ is unchanged.  In contrast, the vector field may vary
the parameters, so $V(s_i)$ need not vanish.  If its root component is
$\tau_i z_{i+2}\partial_{z_{i+2}}$, the chain rule gives
\[
 V(z_{i+2}')
 =\left(\tau_i+\frac{V(s_i)}{d_i s_i}\right)z_{i+2}'.
\]
This is precisely the root correction required when changing the
normalized coefficient representative.  Thus
\eqref{eq:O3-gauge-normalization} is compatible with changes of
coefficient charts.  Different choices of the local $d_i$-th root
differ by a locally constant root of unity; the logarithmic jet
coordinates and the root field $z_{i+2}\partial_{z_{i+2}}$ are invariant
under this change.

We next make the coefficient homogeneities of the generators agree.
Here multidegree means the degree of the vector field, including its
coordinate derivatives: $\partial_{a_{i,\alpha}}$ has degree $-1$
in block $i$.  The mixed fields with constant data consequently have
multidegree $(0,0)$: their jet coefficients have degree zero, while
their $i$-th coefficient vectors have degree one in $a_i$, canceled
by the degree of $\partial_{a_{i,\alpha}}$.
For a line, the kernel accelerations in
Lemma~\ref{lem:O3-linear-generators}\textup{(iii)} have multidegree
$(1,0)$, because
$(-\partial_2F_1,\partial_1F_1)^{\mathsf T}$ is linear in $a_1$.
An $i$-th third difference has coefficient components independent of
the parameters, hence multidegree $-1$ in block $i$ and zero in the
other block.

Multiplying these three types respectively by the fixed coefficient
linear forms
\[
 \ell_1\ell_2,\qquad \ell_2,\qquad
 \ell_i^2\ell_j\quad(j\neq i)
\]
makes each resulting vector field homogeneous of multidegree $(1,1)$.
For example, a third difference in block one has degree $(-1,0)$,
and multiplication by $\ell_1^2\ell_2$ raises it to $(1,1)$.
After normalization on each chart by
\eqref{eq:O3-gauge-normalization}, this common homogeneity gives
sections with parameter twist
\[
 \OO_{\PP(V_{d_1})}(1)\boxtimes\OO_{\PP(V_{d_2})}(1).
\]
Indeed, under $a_i'=s_i a_i$, homogeneity contributes the factor
$s_1s_2$, in addition to the coordinate transformation of the vector
field; this is the transition rule for the indicated twist.
At a fixed point where a multiplier is nonzero, multiplication only
rescales the field's value by a nonzero scalar.  That scalar can be
absorbed into the constants in a pointwise linear combination, so
the preceding spanning statements still apply there.
\subsubsection{The base pole bound.}
As in the one-component calculation, take the overlap
\[
 w_0=1/z_1,\qquad w_2=z_2/z_1,
\]
and denote its ordinary jets by $\zeta_j^{(r)}=w_j^{(r)}$ for $j=0,2$.
The inverse jet transformation is
\begin{align}
 \xi_1^{(1)}&=-\frac{\zeta_0^{(1)}}{w_0^2},&
 \xi_2^{(1)}&=\frac{\zeta_2^{(1)}}{w_0}
                    -\frac{w_2\zeta_0^{(1)}}{w_0^2},
       \label{eq:O3-first-projective-jets}\\
 \xi_1^{(2)}&=-\frac{\zeta_0^{(2)}}{w_0^2}
                    +\frac{2(\zeta_0^{(1)})^2}{w_0^3},&
 \xi_2^{(2)}&=\frac{\zeta_2^{(2)}}{w_0}
       -\frac{w_2\zeta_0^{(2)}}{w_0^2}
       -\frac{2\zeta_0^{(1)}\zeta_2^{(1)}}{w_0^2}
       +\frac{2w_2(\zeta_0^{(1)})^2}{w_0^3}.
       \label{eq:O3-second-projective-jets}
\end{align}
For the acceleration part alone these identities give
\begin{equation}\label{eq:O3-projective-overlap}
 \begin{aligned}
 V_C(\zeta_0^{(2)})&=-w_0^2C^1(\xi^{(1)},\xi^{(1)}),\\
 V_C(\zeta_2^{(2)})&=w_0\bigl(C^2(\xi^{(1)},\xi^{(1)})
                       -w_2C^1(\xi^{(1)},\xi^{(1)})\bigr).
 \end{aligned}
\end{equation}
Each $w_0^4C^j(\xi^{(1)},\xi^{(1)})$ is polynomial in the new
coordinates and first jets.  The two displayed components therefore
have pole orders at most two and three, respectively.
The auxiliary coordinates change to $\widehat z_{i+2}=w_0z_{i+2}$.
For the second logarithmic jet, the acceleration contributes
\begin{equation}\label{eq:O3-root-acceleration-pole}
 V_C\bigl((\log\widehat z_{i+2})''\bigr)
       =-w_0C^1(\xi^{(1)},\xi^{(1)}),
\end{equation}
also of pole order at most three.  The first logarithmic jet is fixed
by the acceleration.

If $w=\phi(z)$ and $h=\log w_0$, the matrix
contributions to the new first and second ordinary jets are
\[
 D\phi(T\xi^{(1)}),\qquad
 D\phi(T\xi^{(2)})+2D^2\phi(\xi^{(1)},T\xi^{(1)}).
\]
Those to the new logarithmic jets are
\[
 \begin{aligned}
 &b^{(i)}\xi^{(1)}+\mathrm{d}h(T\xi^{(1)}),\\
 &b^{(i)}\xi^{(2)}+\mathrm{d}h(T\xi^{(2)})
                           +2\mathrm{d}^2h(\xi^{(1)},T\xi^{(1)}).
 \end{aligned}
\]
Here $D\phi,\mathrm{d}h$ have order at least one in $w_0$, and
$D^2\phi,\mathrm{d}^2h$ have order at least two.  Formulas
\eqref{eq:O3-first-projective-jets}--\eqref{eq:O3-second-projective-jets}
show pole order at most two for the terms containing these derivatives;
$b^{(i)}\xi^{(2)}$ has order at most three.  Interchanging $z_1,z_2$
handles the other projective overlap.

Every coefficient $v_{i,\alpha}(z,a)$ has $z$-degree at most three, so
\[
 w_0^3v_{i,\alpha}(w_0^{-1},w_2/w_0,a)
\]
is polynomial.  The scalar $\ell_i(v_i)$ in
\eqref{eq:O3-gauge-normalization} also has $z$-degree at most three.
Its root component is logarithmic, with coefficient $-\ell_i(v_i)/d_i$,
so its pole order is at most three as well.
Thus all components extend with base twist $\OO_{\PP^2}(3)$.

\subsubsection{Regularity along the divisor.}
At $C_1\cap C_2$ the matrix with rows $\mathrm{d}F_1,\mathrm{d}F_2$ is invertible.
Equations \eqref{eq:O3-E1} solve for $\xi^{(1)}$ in terms of
$\xi_3^{(1)},\xi_4^{(1)}$ by inverting this regular matrix, and
\eqref{eq:O3-E2} then solve for $\xi^{(2)}$.  The solutions are
regular, linear in the second logarithmic jets and polynomial in the
first ones.  Substitution into \eqref{eq:O3-mixed-jet-variation}
proves regularity at the crossing.  At a point of only $C_i$, use
$(F_i,z_j)$ as local coordinates and solve for the complementary normal
jet using a nonzero partial derivative of $F_i$.  The other auxiliary
coordinate is nonzero.  These calculations also cover vanishing first
jets.  All root components are proportional to $z_{i+2}$, so the
auxiliary boundary is preserved.  The relation
$\mathrm{d}\log F_i=d_i\,\mathrm{d}\log z_{i+2}$ and invariance under the constant root
choices prove regular descent to the logarithmic jets downstairs.

\begin{proposition}[The mixed low-pole fields]
\label{prop:mixed-O3-package}
For $d_1\geqslant1$, $d_2\geqslant2$, the constructions above give a
finite family of global logarithmic slanted fields on the relative
two-jet space over the SNC parameter locus, with common twist
\begin{equation}\label{eq:O3-common-twist}
 \OO_{\PP^2}(3)\boxtimes
 \OO_{\PP(V_{d_1})}(1)\boxtimes\OO_{\PP(V_{d_2})}(1),
\end{equation}
pulled back to that jet space.  They fix the base point of $\PP^2$
and commute with two-jet reparametrizations.
\end{proposition}

\begin{proof}
Take bases of $T,C,b^{(1)},b^{(2)}$, giving $3+6+2+2$ mixed fields
when both degrees are at least two; use the ten compatible fields of
Lemma~\ref{lem:O3-linear-generators} when $d_1=1$.  Add the finitely
many third differences and the indicated parameter multipliers.
Tangency, equivariance, and extension have been checked above.
\end{proof}

\subsection{Simultaneous cancellation at a fixed point}
\label{subsec:mixed-pointwise}

For general parameters $a$, take a nonempty open set $U_a$ on which
$F_1F_2\neq0$ and the two rows $\mathrm{d}F_1,\mathrm{d}F_2$ are independent.
Fix $x\in U_a$ and a lift $\widetilde x=(x,z_3,z_4,a)$ in the chosen
coefficient charts.  All coefficients in the following construction are
evaluated at this point.  In the notation of the logarithmic case, put
\[
 H_i^{\log}=F_i\,\mathrm{d}^2F_i-(\mathrm{d}F_i)^{\mathsf T}\mathrm{d}F_i,
 \qquad i=1,2.
\]
The value of $G_{V,i}$ at $y=x$ is already zero.  Its gradient vanishes
precisely when
\begin{equation}\label{eq:O3-N-choice}
 b^{(i)}=\frac{\mathrm{d}F_i(x,a_i)T}{d_iF_i(x,a_i)}.
\end{equation}
Substituting this expression into \eqref{eq:O3-raw-Hessian} gives
\begin{equation}\label{eq:O3-mixed-Hessian}
 F_i\,\mathrm{d}F_i\circ C+T^{\mathsf T}H_i^{\log}+H_i^{\log}T=0,
 \qquad i=1,2,
\end{equation}
as the condition~\eqref{eq:O3-zero-Taylor} for the Hessian of $G_{V,i}$
with respect to $y$ to vanish at $y=x$.  Here
$(\mathrm{d}F_i\circ C)(r,s)=\mathrm{d}F_i(C(r,s))$, and every term in
\eqref{eq:O3-mixed-Hessian} is evaluated at $(x,a)$.
The two equations determine a unique symmetric bilinear map $C_T$:
\begin{equation}\label{eq:O3-CT}
 C_T(r,r)=-
 \begin{pmatrix}F_1\mathrm{d}F_1\\ F_2\mathrm{d}F_2\end{pmatrix}^{-1}
 \begin{pmatrix}
 r^{\mathsf T}(T^{\mathsf T}H_1^{\log}+H_1^{\log}T)r\\
 r^{\mathsf T}(T^{\mathsf T}H_2^{\log}+H_2^{\log}T)r
 \end{pmatrix}.
\end{equation}
The inverse exists by the choice of $U_a$.

If $d_1=1$, then $H_1^{\log}=-(\mathrm{d}F_1)^{\mathsf T}\mathrm{d}F_1$.
The first equation of \eqref{eq:O3-mixed-Hessian} therefore reads
\[
 \mathrm{d}F_1(C_T(r,r))
   =2\mathrm{d}F_1(r)\,\frac{\mathrm{d}F_1(Tr)}{F_1}
   =2\mathrm{d}F_1(r)b^{(1)}r.
\]
This is exactly \eqref{eq:O3-linear-compatibility}.  Moreover,
\eqref{eq:O3-linear-Q} vanishes identically in $y$ at $(x,a)$,
by \eqref{eq:O3-N-choice}.  Thus the ten fields of
Lemma~\ref{lem:O3-linear-generators} supply the required values of
$T,b^{(1)},b^{(2)},C_T$ in this case as well.

\begin{lemma}[Cancellation on a fixed jet fiber]
\label{lem:O3-pointwise-cancellation}
For each $T\in\mathfrak{sl}_2(\CC)$ there is a constant linear
combination of the fields of Proposition~\ref{prop:mixed-O3-package}
whose base, parameter, and auxiliary-coordinate components vanish on
the whole jet fiber over $\widetilde x$, and whose jet components there
are those of \eqref{eq:O3-mixed-jet-variation}, with
$b^{(i)}$ and $C=C_T$ given by
\eqref{eq:O3-N-choice} and \eqref{eq:O3-CT}.
\end{lemma}

\begin{proof}
Choose the mixed fields, or the compatible degree-one fields, with
these prescribed values at $(x,a)$.  Their parameter polynomials have
zero value, gradient, and Hessian at $y=x$.
For $d_i\geqslant3$, cancel the entire polynomial in the $i$-th
block using Lemma~\ref{lem:O3-third-differences}.  For $d_i=2$ it is
already zero, and the preceding calculation handles $d_i=1$.
All combinations have constant coefficients once $(x,a)$ is fixed.

This cancellation is compatible with coefficient normalization.
Indeed, \eqref{eq:O3-gauge-normalization} is linear in the parameter
components, so cancellation before normalization also cancels the
auxiliary radial corrections afterwards.  The common-twist
multipliers are nonzero scalars at the chosen coefficient point and
can be absorbed into the constants in the linear combination.
Thus every non-jet component vanishes along the entire fixed fiber.
\end{proof}

Denote the resulting operator on this fixed jet fiber by
$\widetilde V_T$.  In subsequent compositions and commutators,
$x,a,T,b^{(i)},C_T$ are fixed: the operators differentiate only jet
variables.  In particular, no derivative of the pointwise inverse in
\eqref{eq:O3-CT} occurs.  Notice that if $T=0$, the two independent
rows $\mathrm{d}F_i$ force $C_T=0$; obtaining a nonzero acceleration direction
requires the commutator calculation below.

\subsection{A commutator at the same fixed point}
\label{subsec:mixed-curvature}

At the fixed point, choose the two independent covectors
$\mathrm{d}F_1/F_1$ and $\mathrm{d}F_2/F_2$ as a coframe, denoted by
$\mathrm{d}z_1,\mathrm{d}z_2$, so that
\[
 \frac{\mathrm{d}F_1}{F_1}=\mathrm{d}z_1,\qquad \frac{\mathrm{d}F_2}{F_2}=\mathrm{d}z_2.
\]
Express both ordinary jet vectors in the dual frame, retaining the
notation $\xi_j^{(r)}$ for their components.  All matrices and bilinear
maps below are expressed in this frame, which is held fixed under the
fiber operators.  Formula~\eqref{eq:O3-CT} becomes
\[
 C_T^i(r,r)=-\frac{r^{\mathsf T}
       (T^{\mathsf T}H_i^{\log}+H_i^{\log}T)r}{F_i^2},
 \qquad i=1,2.
\]
Set
\[
 W_{12}=\xi_1^{(1)}\xi_2^{(2)}-\xi_2^{(1)}\xi_1^{(2)}.
\]
Since $\operatorname{tr}T=0$, the action on invariant expressions is
\begin{equation}\label{eq:O3-residual}
 \widetilde V_T(\xi^{(1)})=T\xi^{(1)},\qquad
 \widetilde V_T(W_{12})=
 \xi_1^{(1)}C_T^2(\xi^{(1)},\xi^{(1)})
 -\xi_2^{(1)}C_T^1(\xi^{(1)},\xi^{(1)}).
\end{equation}
Use the three matrices
\[
 \mathsf H=\begin{pmatrix}1&0\\0&-1\end{pmatrix},\qquad
 \mathsf E=\begin{pmatrix}0&1\\0&0\end{pmatrix},\qquad
 \mathsf F=\begin{pmatrix}0&0\\1&0\end{pmatrix}.
\]
With $[V,V']=VV'-V'V$, their first-jet actions satisfy
$[\widetilde V_{\mathsf H},\widetilde V_{\mathsf E}]
 (\xi^{(1)})=-2\widetilde V_{\mathsf E}(\xi^{(1)})$.
On polynomials in $\xi_1^{(1)},\xi_2^{(1)},W_{12}$, the remaining
component is
\begin{align}
 &[\widetilde V_{\mathsf H},\widetilde V_{\mathsf E}]
       +2\widetilde V_{\mathsf E} \label{eq:O3-curvature}\\
 =&-2(\xi_2^{(1)})^2
 \left[
 \left(\frac{(H_1^{\log})_{11}}{F_1^2}
             +\frac{(H_2^{\log})_{12}}{F_2^2}\right)\xi_1^{(1)}
 +\left(\frac{(H_1^{\log})_{12}}{F_1^2}
             +\frac{(H_2^{\log})_{22}}{F_2^2}\right)\xi_2^{(1)}
 \right]\frac{\partial}{\partial W_{12}}. \notag
\end{align}
The two cubic expressions entering this calculation are
\begin{align*}
 \widetilde V_{\mathsf H}(W_{12})
 &=-\frac{2(H_2^{\log})_{11}}{F_2^2}(\xi_1^{(1)})^3
   +\frac{2(H_1^{\log})_{11}}{F_1^2}(\xi_1^{(1)})^2\xi_2^{(1)}\\
 &\quad+\frac{2(H_2^{\log})_{22}}{F_2^2}\xi_1^{(1)}(\xi_2^{(1)})^2
   -\frac{2(H_1^{\log})_{22}}{F_1^2}(\xi_2^{(1)})^3,\\
 \widetilde V_{\mathsf E}(W_{12})
 &=-\frac{2(H_2^{\log})_{11}}{F_2^2}(\xi_1^{(1)})^2\xi_2^{(1)}
   +2\left(\frac{(H_1^{\log})_{11}}{F_1^2}
               -\frac{(H_2^{\log})_{12}}{F_2^2}\right)
                  \xi_1^{(1)}(\xi_2^{(1)})^2\\
 &\quad+\frac{2(H_1^{\log})_{12}}{F_1^2}(\xi_2^{(1)})^3.
\end{align*}
For this calculation abbreviate
\[
 u=\xi_1^{(1)},\qquad v=\xi_2^{(1)},\qquad
 h_{i,jk}=\frac{(H_i^{\log})_{jk}}{F_i^2},
 \qquad
 S_{\mathsf H}=\widetilde V_{\mathsf H}(W_{12}),\quad
 S_{\mathsf E}=\widetilde V_{\mathsf E}(W_{12}).
\]
The coefficients $h_{i,jk}$ are constant on the fixed fiber, and
$S_{\mathsf H},S_{\mathsf E}$ are independent of $W_{12}$.  Hence
\begin{align*}
 (u\partial_u-v\partial_v)S_{\mathsf E}
 &=-2h_{2,11}u^2v
   -2(h_{1,11}-h_{2,12})uv^2-6h_{1,12}v^3,\\
 v\partial_u S_{\mathsf H}
 &=-6h_{2,11}u^2v+4h_{1,11}uv^2+2h_{2,22}v^3.
\end{align*}
Adding twice the displayed expression for $S_{\mathsf E}$ therefore gives
\begin{align*}
 \bigl([\widetilde V_{\mathsf H},\widetilde V_{\mathsf E}]
             +2\widetilde V_{\mathsf E}\bigr)(W_{12})
 &=(-2+6-4)h_{2,11}u^2v\\
 &\quad+\bigl[-2(h_{1,11}-h_{2,12})-4h_{1,11}
                    +4(h_{1,11}-h_{2,12})\bigr]uv^2\\
 &\quad+(-6h_{1,12}-2h_{2,22}+4h_{1,12})v^3\\
 &=-2v^2\bigl((h_{1,11}+h_{2,12})u
                    +(h_{1,12}+h_{2,22})v\bigr).
\end{align*}
The operator on the left annihilates $u$ and $v$, so this proves
\eqref{eq:O3-curvature} with the stated sign.  Only this commutator
is needed.

\begin{lemma}[Generic nonvanishing of the commutator coefficient]
\label{lem:O3-curvature-open}
For $d_1\geqslant1$ and $d_2\geqslant2$, there is a nonempty
Zariski-open subset of the SNC parameter space such that, for every
parameter $a$ in this subset, there is a nonempty Zariski-open set
$U_a\subset\PP^2$ on which
\[
 F_1F_2\det(\mathrm{d}F_1,\mathrm{d}F_2)\neq0
\]
and the coefficient of $\partial/\partial W_{12}$ in
\eqref{eq:O3-curvature} is a nonzero polynomial in
$\xi_1^{(1)},\xi_2^{(1)}$ at every point of $U_a$.
\end{lemma}

\begin{proof}
In the fixed coframe used in \eqref{eq:O3-curvature}, put
\[
 A_a(x):=
 \frac{(H_1^{\log})_{11}}{F_1^2}
 +\frac{(H_2^{\log})_{12}}{F_2^2},
 \qquad
 B_a(x):=
 \frac{(H_1^{\log})_{12}}{F_1^2}
 +\frac{(H_2^{\log})_{22}}{F_2^2}.
\]
These are rational functions of the base point and the parameters.
The commutator coefficient is
\[
 -2(\xi_2^{(1)})^2
 \bigl(A_a(x)\xi_1^{(1)}+B_a(x)\xi_2^{(1)}\bigr).
\]
At a fixed base point $x$, this is a nonzero polynomial in the
$1$-jet variables if and only if
\[
 (A_a(x),B_a(x))\neq(0,0).
\]

To verify generic nonvanishing, suppose first that
$d_1,d_2\geqslant2$ and take, at $x=0$,
\[
 F_1=1+z_1+z_1^2,\qquad
 F_2=1+z_2+z_2^2.
\]
Here $F_1=F_2=1$, the normalized gradients are the coordinate
rows, and
\[
 H_1^{\log}=\operatorname{diag}(1,0),\qquad
 H_2^{\log}=\operatorname{diag}(0,1).
\]
Thus $(A_a(0),B_a(0))=(1,1)$.
For $d_1=1$, take $F_1=1+z_1$ and the same $F_2$; then
$(A_a(0),B_a(0))=(-1,1)$.
Homogenization realizes these examples in every indicated degree.
Their projective zero sets need not be smooth: the examples are
used only to show that $A_a$ and $B_a$ do not both vanish
identically as rational functions.

After clearing denominators, including those arising from the
coframe change, the condition
\[
 A_a(x)\equiv0,\qquad B_a(x)\equiv0
\]
is a finite collection of polynomial conditions on the parameters.
The examples show that this is a proper closed condition.
Its complement therefore meets the dense SNC open subset of the
irreducible parameter space.  Restrict further so that
$F_1F_2\det(\mathrm{d}F_1,\mathrm{d}F_2)$ is not identically zero
on the base.

For each remaining parameter $a$, let $U_a^\circ$ be the nonempty
open subset of the chosen affine chart on which
\[
 F_1F_2\det(\mathrm{d}F_1,\mathrm{d}F_2)\neq0.
\]
The functions $A_a$ and $B_a$ are regular there.  Define
\[
 U_a
 :=
 U_a^\circ\setminus
 \{x\in U_a^\circ:A_a(x)=B_a(x)=0\}.
\]
Since $A_a$ and $B_a$ do not both vanish identically, $U_a$ is
nonempty.  At every point of $U_a$, the commutator coefficient
is a nonzero cubic polynomial in the $1$-jet variables.
\end{proof}

\subsection{Direct jet differentiation and the second global equation}
\label{subsec:mixed-differentiation}

We apply Siu differentiation
\cite{Siu,Siu2015,DiverioMerkerRousseau} directly to the family
$\omega_1$ of Proposition~\ref{prop:factor-family}.
Write locally $\omega_1=P e$, where $e$ is a frame of
$\OO_{\PP^2}(-t)$, and write a field of
Proposition~\ref{prop:mixed-O3-package} as $V=V_\tau\otimes\tau$
in a frame of its common twist.  Define
$D_V\omega_1=V_\tau(P)e\otimes\tau$.
Under $e'=g(z)e$, one has $P'=g^{-1}P$ and $V_\tau(g)=0$, because
the field fixes the base point; hence
$V_\tau(P')=g^{-1}V_\tau(P)$.
Under $\tau'=h\tau$, the local field changes to $h^{-1}V_\tau$,
which gives the same tensor.  Thus the derivatives glue.
Let $r_{\lambda,\mu}$ denote the reparametrization action in
\eqref{eq:O3-reparametrization}.  Equivariance gives
$r_{\lambda,\mu}^*(V_\tau(P))
 =V_\tau(r_{\lambda,\mu}^*P)=\lambda^mV_\tau(P)$,
so the weight remains $m$.  With
$\mathcal A=\OO_{\mathcal S}(1,1)|_{\mathcal U}$, this yields
\begin{equation}\label{eq:relative-Siu-derivative}
 D_V\omega_1\in H^0\!\left(\PP^2\times\mathcal U,
 E_{2,m}^{\mathrm{rel}}\otimes
 \bigl(\OO_{\PP^2}(-t+3)\boxtimes\mathcal A\bigr)\right),
\end{equation}
where $E_{2,m}^{\mathrm{rel}}$ is the relative invariant logarithmic
jet-differential bundle.  The relative direct-image formula
\eqref{eq:direct-image} gives a regular section on the full second
Semple space, including $\Gamma_2$.

We next show independence by a fixed-point argument.  On the complement
of the divisor, the local polynomial representing $\omega_1$ has the
usual form
\begin{equation}\label{eq:O3-fixed-point-polynomial}
 P(x,a,\xi^{(1)},\xi^{(2)})
 =\sum_{k=0}^{\lfloor m/3\rfloor}
       R_k(x,a,\xi^{(1)})W_{12}^k,
 \qquad \deg_{\xi^{(1)}}R_k=m-3k.
\end{equation}
Its pullback by $\pi_{[2]}$ has the same expression; the auxiliary
jets are determined by \eqref{eq:O3-E1}--\eqref{eq:O3-E2}, since
$F_1F_2\neq0$.  We use the same letter $P$ for this pullback.

For general $(x,a)$, this is a nonzero reduced polynomial in
$\xi_1^{(1)},\xi_2^{(1)},W_{12}$.  Indeed, away from $\Gamma_2$
the zero divisor of $\omega_{1,a}=\gamma_a^{b_1}s_a$ is the reduced
horizontal divisor $Z_a$.  In characteristic zero its fibers over
general base points are reduced.  On regular jets, the projection to
the Semple space is locally a smooth projection (use a nonzero
first-jet coordinate to normalize reparametrizations), so the
pulled-back equation is reduced there as well.  In the three displayed
variables the excluded set $\xi_1^{(1)}=\xi_2^{(1)}=0$ has
codimension two and cannot support a repeated polynomial factor.
This proves the assertion after shrinking the open set of base
points.  No irreducibility of the polynomial at a fixed point is
required.

Suppose that every derivative \eqref{eq:relative-Siu-derivative}
vanishes along the relative divisor $\mathcal Z$.
Choose $(x,a)$ satisfying the preceding reducedness assertion and
Lemma~\ref{lem:O3-curvature-open}.  Each local derivative at this
point vanishes on the zero set of the reduced polynomial
\eqref{eq:O3-fixed-point-polynomial}; it is therefore divisible by
that polynomial.  Both have weight $m$, and all three variable
weights $1,1,3$ are positive.  The quotient is consequently a scalar.
Constant linear combinations and the pointwise cancellation of
Lemma~\ref{lem:O3-pointwise-cancellation} give
\begin{equation}\label{eq:O3-pointwise-proportionality}
 \widetilde V_T\bigl(P(x,a,\xi^{(1)},\xi^{(2)})\bigr)
   =\rho_{\widetilde V_T}(x,a)
                 P(x,a,\xi^{(1)},\xi^{(2)}).
\end{equation}
The numbers $\rho_{\widetilde V_T}(x,a)$ are constant under all
the fixed-fiber operators.  This is the reason for cancelling the
parameter components before taking commutators.

\begin{proposition}[Fixed-point independence]\label{prop:O3-rigidity}
At a point where the coefficient in \eqref{eq:O3-curvature} is
nonzero, no nonzero polynomial of positive weight of the form
\eqref{eq:O3-fixed-point-polynomial} can satisfy
\eqref{eq:O3-pointwise-proportionality} for
$T=\mathsf H,\mathsf E,\mathsf F$.
\end{proposition}

\begin{proof}
The commutator
$[\widetilde V_{\mathsf H},\widetilde V_{\mathsf E}]$
annihilates $P$, since the two proportionality factors in \eqref{eq:O3-pointwise-proportionality} are constants
on the fixed fiber.  Thus \eqref{eq:O3-curvature}, applied to $P$,
is $2\rho_{\widetilde V_{\mathsf E}}(x,a)P$.
If $P$ has positive degree in $W_{12}$, its left side is the nonzero
coefficient in \eqref{eq:O3-curvature} multiplied by
\[
 \frac{\partial P}{\partial W_{12}}
   =\sum_{k\geqslant1}kR_k(x,a,\xi^{(1)})W_{12}^{k-1}.
\]
It is nonzero and has smaller $W_{12}$-degree than $P$, which is
impossible for a scalar multiple of $P$.  Therefore $P$ is independent
of $W_{12}$.

It is now a binary form in $\xi_1^{(1)},\xi_2^{(1)}$ of degree
$m>0$.  The remaining actions of $\widetilde V_{\mathsf E}$ and
$\widetilde V_{\mathsf F}$ are
$\xi_2^{(1)}\partial_{\xi_1^{(1)}}$ and
$\xi_1^{(1)}\partial_{\xi_2^{(1)}}$.
Both are nilpotent on binary forms of degree $m$, so their eigenvalues
in \eqref{eq:O3-pointwise-proportionality} are zero.
It follows that both partial derivatives of $P$ vanish.  Hence $P$
is constant, contrary to its positive weight.
\end{proof}

The commutator is used only on the fixed fiber to prove independence.
The second global equation will be a derivative by one of the original
fields, with the twist in \eqref{eq:O3-common-twist}.

\begin{proposition}[$\OO(3)$ slanted derivative]
\label{prop:globalized-O3}
Consider the family supplied by
Proposition~\ref{prop:factor-family}, with its fixed integers
$b_1,b_2,m,t$, and assume that $b_2>0$.
After shrinking its nonempty parameter open, write the irreducible
reduced horizontal divisor of $\omega_{1,a}$ as
\[
 Z_a\sim b_1u_1+b_2u_2-th,\qquad
 b_1\geqslant2b_2>0,\qquad m=b_1+b_2.
\]
There is a further nonempty parameter open $\mathcal U$ and a relative
section, with $\mathcal A=\OO_{\mathcal S}(1,1)|_{\mathcal U}$,
\[
 \Omega_2\in H^0\!\left(\mathcal X_2,
 \OO_{\mathcal X_2}(m)\otimes
 \Pi^*\operatorname{pr}_1^*\OO_{\PP^2}(-t+3)
 \otimes p_{\mathcal U}^*\mathcal A\right)
\]
whose restriction to $Z_a$ is nonzero for every $a\in\mathcal U$.
After trivializing $\mathcal A|_a$, its fiber restrictions correspond to
\[
 \omega_{2,a}\in H^0\!\left(\PP^2,
 E_{2,m}T_{\PP^2}^*(\log D_a)\otimes\OO_{\PP^2}(-t+3)\right).
\]
The common zero set of $\omega_{1,a}$ and $\omega_{2,a}$ on $X_{2,a}$
has codimension at least two away from $\Gamma_{2,a}$.
\end{proposition}

\begin{proof}
Intersect the parameter opens in Proposition~\ref{prop:factor-family}
and Lemma~\ref{lem:O3-curvature-open}.  Each of the finitely many
fields in Proposition~\ref{prop:mixed-O3-package} gives the global
derivative \eqref{eq:relative-Siu-derivative} and its corresponding
section $\Omega_V$ on $\mathcal X_2$.
If every $\Omega_V$ vanished along $\mathcal Z$, the fixed-point
argument would give \eqref{eq:O3-pointwise-proportionality} for all
three matrices, contradicting Proposition~\ref{prop:O3-rigidity}.
Consequently one of these global fields satisfies
\begin{equation}\label{eq:relative-nonzero-Siu}
 \Omega_V|_{\mathcal Z}\neq0.
\end{equation}

The family $\mathcal Z\to\mathcal U$ is proper.  After shrinking
$\mathcal U$, generic freeness and cohomology and base change identify
the spaces of sections of the relevant line bundle on $Z_a$ with the
fibers of a vector bundle on $\mathcal U$.
Equation~\eqref{eq:relative-nonzero-Siu} gives a generically nonzero
section of that bundle.  Removing its proper closed zero locus makes
$\Omega_V|_{Z_a}$ nonzero for every remaining parameter.  Set
$\Omega_2=\Omega_V$.

Trivializing the parameter factor on each fiber and applying
\eqref{eq:direct-image} yields $\omega_{2,a}$ with the asserted
weight and twist.  Away from $\Gamma_{2,a}$ the zero divisor of
$\omega_{1,a}=\gamma_a^{b_1}s_a$ is the single reduced component
$Z_a$, on which $\Omega_2$ does not vanish identically.
Thus the two sections have no common divisorial component there.
\end{proof}

\begin{corollary}\label{cor:O3-slanted-alternative}
Let $f:\CC\to\PP^2$ be algebraically nondegenerate and suppose
$t>3$ and $f^*\omega_{1,a}\equiv0$.  Then either
$f^*\omega_{2,a}\not\equiv0$ for the differential in
Proposition~\ref{prop:globalized-O3}, or the first logarithmic lift
$f_{[1]}$ is algebraically degenerate.
\end{corollary}

\begin{proof}
If both pullbacks vanish, set
$Y=\overline{f_{[2]}(\CC)}^{\,\mathrm{Zar}}\subset X_2$.
By Lemma~\ref{lem:regular-lift-not-vertical}, $Y$ is not contained in
$\Gamma_2$.  Its dense open part outside $\Gamma_2$ lies in the
codimension-two common zero set of the two equations.  Since
$\dim X_2=4$, one has $\dim Y\leqslant2$.
The proper projection $\pi_{2,1}$ therefore maps $Y$ to a proper
algebraic subset of the threefold $X_1$ containing $f_{[1]}(\CC)$.
\end{proof}

\section{The two complementary mechanisms}\label{sec:mechanisms}
The argument now follows a fixed decision tree.  If the chosen factor does
not vanish on the lifted curve, its original twist gives the estimate.  If
it vanishes, the case $b_2=0$ is handled on $X_1$; for $b_2>0$, a
pole-three derivative treats $t>3$, while the key vanishing lemma and the
zero-locus theorem treat $t\leqslant3$.

For one of the six boundary pairs in Lemma~\ref{lem:keyvanishing}, let
$\mathcal U_{\mathrm{key}}=\mathcal U_{\mathrm{van}}(d_1,d_2)$; for every
other degree pair in our range, put
$\mathcal U_{\mathrm{key}}=\mathcal S_{d_1,d_2}^{\mathrm{snc}}$.
Let $\mathcal U_0$ be the intersection of $\mathcal U_{\mathrm{key}}$ with
the parameter open furnished by Proposition~\ref{prop:factor-family}.
It is nonempty because the parameter space is irreducible.  Work first over
$a\in\mathcal U_0$, and put
$D=D_a$ and $\omega_1=\omega_{1,a}$.  Let $f:\CC\to\PP^2$ be
algebraically nondegenerate.  If
$f^*\omega_1\not\equiv0$, Theorem~\ref{thm:jetSMT} and
$m/t\leqslant\rho_0^{-1}$ give the required estimate directly.  It remains to
treat the vanishing case
\begin{equation}\label{eq:fomega1zero}
   f^*\omega_1\equiv0.
\end{equation}
Write the selected horizontal component as
\[
 Z\sim b_1u_1+b_2u_2-th.
\]
If $b_2=0$, then $Z=\pi_{2,1}^{-1}(Z_1)$ for a proper divisor
$Z_1\subset X_1$.  By Lemma~\ref{lem:regular-lift-not-vertical}, the
regular lift is not contained in $\Gamma_2$.  Hence
\eqref{eq:fomega1zero} implies
\[
 f_{[1]}(\CC)\subset Z_1,
\]
so $f_{[1]}$ is algebraically degenerate and Theorem~\ref{thm:McQuillan}
applies.  This branch uses no slanted field.  Hence in the two cases below
we may assume $b_2>0$.  Proposition~\ref{prop:globalized-O3}, applied after
restricting the factor family to $\mathcal U_0$, supplies a nonempty open
subset $\mathcal U_1\subset\mathcal U_0$ on which the chosen differential
and its slanted derivative are defined simultaneously.  In the remainder
of this section we work over $a\in\mathcal U_1$.

\subsection{Case 1: \texorpdfstring{$t>3$}{t greater than 3}}
Corollary~\ref{cor:O3-slanted-alternative} supplies an independent
\[
   \omega_{2,a}\in H^0\!\left(\PP^2,E_{2,m}\bT^*\otimes\OO(-t+3)\right).
\]
If $f^*\omega_{2,a}\not\equiv0$, Theorem~\ref{thm:jetSMT} gives
\[
   T_f(r)\leqslant\frac{m}{t-3}N_f^{[1]}(r,D)+o(T_f(r))\ \|.
\]
Since $m/t\leqslant\rho_0^{-1}$ and $t\geqslant4$,
\begin{equation}\label{eq:case1-bound}
 \frac{m}{t-3}
 \leqslant\frac1{\rho_0}\frac{t}{t-3}
 \leqslant\frac4{\rho_0}.
\end{equation}
If $f^*\omega_{2,a}\equiv0$, Corollary~\ref{cor:O3-slanted-alternative}
makes $f_{[1]}$ algebraically degenerate, and
Theorem~\ref{thm:McQuillan} applies.

\subsection{Case 2: \texorpdfstring{$t\leqslant3$}{t at most 3}}
Since $b_2>0$ and $b_1\geqslant2b_2$, one has $m\geqslant3$.  If
$\rho_{\mathrm{DEG}}>1$, then automatically
$t/m\leqslant1<\rho_{\mathrm{DEG}}$.  Otherwise the degree pair is one of the
six boundary pairs, and Lemma~\ref{lem:keyvanishing} implies that the
existing nonzero $\omega_1$ cannot have
$t/m\geqslant\rho_{\mathrm{DEG}}$.  Thus in every case
\begin{equation}\label{eq:mgreater4t}
 \rho=\frac tm<\rho_{\mathrm{DEG}}(d_1,d_2).
\end{equation}
Theorem~\ref{thm:zero-locus} applies.  For every rational
$0<\tau<\min\{3,\tau_1(t/m)\}$ it gives either
\[
   T_f(r)\leqslant\frac3\tau N_f^{[1]}(r,D)+o(T_f(r))\ \|
\]
or the algebraic degeneracy of $f_{[1]}$, followed by \eqref{eq:McQuillan-cubic}.

\subsection{The effective constant}\label{subsec:explicit-constant}
The discrete gap below the zero-locus threshold is explicit.  Set
\begin{equation}\label{eq:rho-star-general}
 \rho_*=
 \max_{1\leqslant q\leqslant3}
 \frac{q}{\max\{3,\floor{q/\rho_{\mathrm{DEG}}}+1\}},
 \qquad
 \tau_*=\min\{3,\tau_1(\rho_*)\}.
\end{equation}
For integers $1\leqslant t\leqslant3$, $m\geqslant3$ satisfying
$t/m<\rho_{\mathrm{DEG}}$, one has $t/m\leqslant\rho_*$.  The function
$3/\tau_1(\rho)$ is strictly increasing on $[0,\rho_{\mathrm{DEG}})$
by \eqref{eq:tau1-monotonicity}.  Consequently
any integer strictly larger than
\begin{equation}\label{eq:effective-A-general}
 \max\left\{\rho_0^{-1},\frac4{\rho_0},\frac3{\tau_*},
 \frac1{\kappa}\right\}
\end{equation}
is an admissible $\mathcal A_{d_1,d_2}$ in
Theorem~\ref{thm:mainSMT}.  In the order displayed, the four entries account
for the original section, its pole-three derivative, the zero-locus
section, and McQuillan's branch.

\subsection{The sharper conic--quintic dichotomy}
For $(d_1,d_2)=(2,5)$ one has
\[
 \kappa=4,\qquad \bar c_1^2=16,\qquad \bar c_2=21,\qquad
 \rho_{\mathrm{DEG}}=\frac{19}{48}.
\]
The extra case $(m,t)=(8,3)$ in $\mathfrak F_{45}(2,5)$ changes the useful
split from $3/8$ to $1/3$.  Put $\rho=t/m$ and assume $b_2>0$.

If $\rho\leqslant1/3$, then $\rho<19/48$, so
Theorem~\ref{thm:zero-locus} applies.  Since $3/\tau_1(\rho)$ is increasing
on $[0,19/48)$ by \eqref{eq:tau1-monotonicity}, it is enough to
evaluate the endpoint.  Formula
\eqref{eq:tau1-general} gives
\begin{equation}\label{eq:conic-quintic-tau}
 \tau_1(1/3)=\frac{45-\sqrt{1989}}6,
 \qquad
 \frac3{\tau_1(1/3)}
 =\frac{45+\sqrt{1989}}2<45.
\end{equation}
The last inequality follows from $1989<45^2$.
For the actual ratio $\rho\leqslant1/3$, one may therefore choose a rational
number $1/15<\tau<\tau_1(\rho)$; the zero-locus coefficient $3/\tau$ is then
strictly smaller than $45$.

Suppose instead that $\rho>1/3$.  If $t\leqslant3$, then
$m<3t$ and hence $(m,t)\in\mathfrak F_{45}(2,5)$, contradicting the
nonzero choice of $\omega_1$ and Lemma~\ref{lem:keyvanishing}.  Thus
$t\geqslant4$, and the pole-three derivative is available.  Integrality
gives $m\leqslant3t-1$, whence
\begin{equation}\label{eq:conic-quintic-O3-eleven}
 \frac{m}{t-3}
 \leqslant\frac{3t-1}{t-3}
 =3+\frac8{t-3}\leqslant11.
\end{equation}
If the differentiated equation vanishes as well, the first lift is
algebraically degenerate and McQuillan's estimate applies.

Finally,
\begin{equation}\label{eq:conic-quintic-RR}
 \rho_{\mathrm{RR}}(2,5)
 =\frac29\left(8-\sqrt{\frac{455}{8}}\right)>\frac1{10};
\end{equation}
indeed, after squaring, the last inequality is $22750<22801$.  Choose a
rational number
\[
 \frac1{10}<\rho_0<\rho_{\mathrm{RR}}(2,5).
\]
The direct branch then has coefficient $\rho_0^{-1}<10$,
\eqref{eq:conic-quintic-tau} bounds the zero-locus branch by a number
strictly smaller than $45$, and
\eqref{eq:conic-quintic-O3-eleven} bounds the differentiated branch by $11$.
McQuillan's coefficient is $1/4$.  Therefore the integer $45$ covers every
branch for a conic and a quintic.

\subsection{The explicit line--octic constant}
For $(d_1,d_2)=(1,8)$ one has $\kappa=6$ and
$\rho_{\mathrm{DEG}}=3/8$.  Formula~\eqref{eq:rho-star-general} therefore
gives $\rho_*=1/3$.  Substitution in \eqref{eq:tau1-general} yields
\begin{equation}\label{eq:line-octic-tau}
 \tau_1(1/3)=\frac{23-5\sqrt{21}}2,
 \qquad
 \frac3{\tau_1(1/3)}=\frac32\bigl(23+5\sqrt{21}\bigr)<69.
\end{equation}
Put $\rho=t/m$ and retain the assumption $b_2>0$.
For $\rho\leqslant1/3<\rho_{\mathrm{DEG}}$, the monotonicity on
$[0,3/8)$ proved in \eqref{eq:tau1-monotonicity} gives a rational
choice $3/69<\tau<\tau_1(\rho)$ for the zero-locus theorem; hence its
coefficient $3/\tau$ is strictly smaller than $69$.
If $\rho>1/3$ and $t\leqslant3$, then $m<3t$: there are no cases at $t=1$,
and the possibilities are $m=3,4,5$ at $t=2$ and $m=3,\ldots,8$ at $t=3$.
Lemma~\ref{lem:keyvanishing} excludes the first three weights at twist $2$
and the last three at twist $3$; multiplication by a nonzero linear form
also excludes the first three at twist $3$.  Hence, for $\rho>1/3$,
one has $t\geqslant4$, and integrality gives $m\leqslant3t-1$.  Thus
\[
 \frac{m}{t-3}\leqslant\frac{3t-1}{t-3}
 =3+\frac8{t-3}\leqslant11,
\]
as in \eqref{eq:conic-quintic-O3-eleven}.
If the differentiated equation also vanishes, the first lift is
algebraically degenerate and McQuillan's estimate applies.
For the direct branch,
\begin{equation}\label{eq:line-octic-RR}
 \rho_{\mathrm{RR}}(1,8)
 =\frac13\left(8-\sqrt{\frac{119}{2}}\right)>\frac2{23};
\end{equation}
indeed, after squaring, it reduces to $62951<63368$.  Choose a rational number
\[
 \frac2{23}<\rho_0<\rho_{\mathrm{RR}}(1,8).
\]
The direct branch has coefficient $\rho_0^{-1}<23/2$, the differentiated
branch has coefficient at most $11$, and the zero-locus coefficient is
strictly smaller than $69$.
McQuillan's coefficient is $1/6$.  Thus the integer $69$ covers every branch
for a line and an octic.

\subsection{The sharper ratio dichotomy for two cubics}
For $(d_1,d_2)=(3,3)$, the stronger part of
Lemma~\ref{lem:keyvanishing} improves the preceding general estimate.
Put $\rho=t/m$ and retain the assumption $b_2>0$.

Suppose first that $\rho\leqslant1/5$.  Since
$1/5<\rho_{\mathrm{DEG}}(3,3)=1/4$, Theorem~\ref{thm:zero-locus}
 applies, irrespective of the absolute size of $t$.  The function
 $3/\tau_1(\rho)$ is increasing on $[0,1/4)$ by
 \eqref{eq:tau1-monotonicity}, and
\begin{equation}\label{eq:Astar-calculation}
 \tau_1(1/5)=\frac{57-\sqrt{3189}}{10},
 \qquad
 \frac3{\tau_1(1/5)}
 =\frac{57+\sqrt{3189}}2
 \approx56.73561581<57.
\end{equation}
Thus one can choose a rational number
 $3/57<\tau<\tau_1(\rho)$, and the zero-locus branch has coefficient
 strictly smaller than $57$.

Suppose next that $\rho>1/5$.  If $t\leqslant3$, integrality gives
$m\leqslant5t-1$, contradicting the cubic finite vanishing in
Lemma~\ref{lem:keyvanishing}.  Hence $t\geqslant4$, so the mixed
$\OO(3)$ derivative is available.  Since $m\leqslant5t-1$,
\begin{equation}\label{eq:cubic-O3-nineteen}
 \frac{m}{t-3}\leqslant\frac{5t-1}{t-3}\leqslant19.
\end{equation}
If the differentiated equation also vanishes on the lifted curve,
Corollary~\ref{cor:O3-slanted-alternative} makes the first lift algebraically
degenerate, and Theorem~\ref{thm:McQuillan} applies.
Consequently the two branches under $f^*\omega_1\equiv0$ and $b_2>0$ are
bounded by $57$; the direct and $b_2=0$ branches are included in the final
assembly.

\section{Proofs of the main statements}\label{sec:main-proofs}

We now assemble the preceding alternatives.  No new numerical or geometric
input is needed at this stage.

\begin{proof}[Proof of Theorem~\ref{thm:mainSMT}]
The fixed pair $(M,T)$ of Proposition~\ref{prop:firstomega} works throughout
$\mathcal S_{d_1,d_2}^{\mathrm{snc}}$.  Intersect the parameter open of
Proposition~\ref{prop:factor-family} with $\mathcal U_{\mathrm{key}}$ and
call the resulting nonempty open set $\mathcal U_0$.  The integers
$b_1,b_2,m,t$ are fixed on this family.  If $b_2=0$, put
$\mathcal U_{\mathrm{SMT}}(d_1,d_2)=\mathcal U_0$.  If $b_2>0$, apply
Proposition~\ref{prop:globalized-O3} to the restricted family and let
$\mathcal U_{\mathrm{SMT}}(d_1,d_2)\subset\mathcal U_0$ be the nonempty open set it
provides.

Fix $a\in\mathcal U_{\mathrm{SMT}}(d_1,d_2)$ and choose $\omega_{1,a}$ as in
Lemma~\ref{lem:irreducible-reduction} and
Proposition~\ref{prop:factor-family}.  If
$f^*\omega_{1,a}\not\equiv0$, the jet Second Main Theorem gives the estimate
with constant at most $\rho_0^{-1}$.  Under \eqref{eq:fomega1zero}, the branch
$b_2=0$ makes $f_{[1]}$ algebraically degenerate and is handled directly by
Theorem~\ref{thm:McQuillan}.  Suppose therefore that $b_2>0$.  For a
degree pair other than $(3,3)$, Case~1 applies when $t>3$ and gives
\eqref{eq:case1-bound} or McQuillan's stronger estimate.  When $t\leqslant3$,
Case~2 and \eqref{eq:rho-star-general} give the zero-locus estimate or
McQuillan's estimate.  Hence any integer strictly larger than
\eqref{eq:effective-A-general} is admissible.

For $(d_1,d_2)=(1,8)$ choose $\rho_0$ as in
\eqref{eq:line-octic-RR}; the preceding line--octic calculation shows that
the integer $69$ is admissible.

For $(d_1,d_2)=(2,5)$ choose $\rho_0$ as in
\eqref{eq:conic-quintic-RR}; use the sharper $1/3$ dichotomy of the preceding
conic--quintic subsection in place of the coarser general split.  The direct,
zero-locus, differentiated, and algebraically degenerate branches are then
bounded respectively by $10$, $45$, $11$, and $1/4$, with strict inequality
where needed.  Hence the integer $45$ is admissible.

For $(d_1,d_2)=(3,3)$ choose
$1/16<\rho_0<\rho_{\mathrm{RR}}(3,3)$, so the direct branch has
coefficient less than $16$.  Under \eqref{eq:fomega1zero}, use instead the
sharper ratio dichotomy of the preceding subsection: equation
\eqref{eq:Astar-calculation} gives a coefficient less than $57$ when
$t/m\leqslant1/5$, and \eqref{eq:cubic-O3-nineteen} gives a coefficient at
most $19$ when $t/m>1/5$.  Thus the integer $57$ works for two cubics.
This proves \eqref{eq:mainSMT}.
\end{proof}

\begin{proof}[Proof of Corollary~\ref{cor:hyperbolicity}]
Fix
\[
   a\in\mathcal U_{\mathrm{SMT}}(d_1,d_2)
   \cap\mathcal S_{\mathrm{Chen}}(d_1,d_2)
\]
and let
\[
   f:\CC\longrightarrow\PP^2\setminus D_a.
\]
If $f$ is algebraically nondegenerate, then
$N_f^{[1]}(r,D_a)=0$, and Theorem~\ref{thm:mainSMT} gives
$T_f(r)\leqslant o(T_f(r))$, impossible for a nonconstant curve.  If $f$ is
algebraically degenerate, Corollary~\ref{cor:no-degenerate-curve} makes it
constant.  Hence the complement is Brody hyperbolic.

It remains to pass from Brody hyperbolicity to the asserted metric
statement.  The one-dimensional boundary strata are
\[
 C_1\setminus C_2,\qquad C_2\setminus C_1.
\]
Since the curves meet transversally in $d_1d_2$ points, their logarithmic
Euler characteristics satisfy
\[
 2g(C_i)-2+d_1d_2>0\qquad(i=1,2)
\]
throughout the stated degree range.  Thus both punctured curves are Brody
hyperbolic, and the zero-dimensional stratum $C_1\cap C_2$ is finite.
We use the following form of Green's criterion: if $X$ is projective,
$D=\bigcup_{i=1}^rD_i$ is a simple-normal-crossing divisor, and both
$X\setminus D$ and every open stratum
\[
 \bigcap_{i\in I}D_i\setminus\bigcup_{j\notin I}D_j
 \qquad(\varnothing\neq I\subset\{1,\ldots,r\})
\]
are Brody hyperbolic, then $X\setminus D$ is hyperbolically embedded in
$X$; see \cite{Green1977} and the precise restatement
\cite[Theorem~4.1]{JavanpeykarLevin2022}.  The strata in the present case are
exactly the complement, the two punctured curves above, and the finite set
$C_1\cap C_2$.  The criterion therefore shows that the complement is
hyperbolically embedded in $\PP^2$; in particular it is Kobayashi hyperbolic.
The set
$\mathcal U_{\mathrm{SMT}}(d_1,d_2)
\cap\mathcal S_{\mathrm{Chen}}(d_1,d_2)$ is very general by
the discussion following Corollary~\ref{cor:hyperbolicity}.
\end{proof}

\section{Proof of the finite key vanishing lemma: exact finite computation}
\label{sec:finite-vanishing}

We now prove Lemma~\ref{lem:keyvanishing}.  It is enough to establish the
vanishing at one transverse pair, because upper semicontinuity will then
give the required nonempty parameter open.  After multiplication by a
nonzero homogeneous polynomial reduces to the least twist for each $m$, the
finite range is
\[
   H^0\!\left(\PP^2,E_{2,m}T_{\PP^2}^*(\log D_a)\otimes\OO(-t)\right)=0
\]
for the pairs listed in the table following Lemma~\ref{lem:keyvanishing}.
The guiding principle is to generate candidate two-jet differentials on one
large ordinary chart, test their logarithmic regularity on one mixed
logarithmic chart, and impose the available symmetries before constructing
any matrix.
This chart-and-divisibility strategy starts from the one-component normal
form in \cite[Proposition~3.1 and Section~3]{HouHuynhMerkerXie2026}.  The
argument below is not a restatement of that calculation: two defining
polynomials produce two residue directions, their projective weights must be
balanced when $d_1\ne d_2$, and the overlap equations are coupled.
For $(3,3)$ we impose the $(\mu_3)^2$ action and coordinate exchange before
constructing the matrices.  The pairs $(2,e)$, $e=5,6$, have a different
cyclic action with coordinate exchange.  The line pairs $(1,e)$, $e=8,9$,
retain the cyclic action but not the involution, while $(3,4)$ is small enough
to run as one unsplit block.
The cubic organization is related to the symmetry method developed in
\cite{CuiHouLiuXie2026}.  That citation concerns the organizational idea of
splitting a finite calculation by symmetry; the two-component normal form,
divisibility criterion, matrices, and verification program used here are
derived independently in this section.  All operations are performed with
integral polynomials.  Every characteristic-zero conclusion is certified by
full column rank at an explicitly recorded good prime; a second prime is used
as an independent reproducibility check whenever practical.

The proof has four stages.  We first choose explicit SNC test pairs.  We then
derive a finite ambient normal form and translate regularity on the mixed
chart into divisibility equations.  Next we decompose these equations by the
available finite symmetries.  Finally, exact modular row reduction certifies
that every resulting block has zero kernel.

\subsection{Degree-uniform test pairs}
We use the following integral pairs; in each row, $A=0$ and $B=0$ are the two
components of the test divisor:
\begin{equation}\label{eq:all-test-pairs}
\begin{array}{c|c|c}
(d_1,d_2)&A&B\\ \hline
(3,3)&X^3+2Y^3+Z^3&2X^3+Y^3+Z^3\\
(3,4)&X^3+Y^3+Z^3&X^4+Y^4+Z^4\\
(2,e),\ e=5,6&XY+Z^2&X^e+Y^e+Z^e\\
(1,8)&Y&X^8+Y^8+Z^8+X^4Y^4\\
(1,9)&Y&X^9+Y^9+Z^9+X^3Y^3Z^3.
\end{array}
\end{equation}
Each row gives an SNC divisor and remains SNC after adjoining the line
$L=\{Z=0\}$.  For $(3,4)$, proportional gradients at a common point would
force all nonzero coordinates to be equal; neither two nor three equal
nonzero coordinates can satisfy the cubic equation.  On $L$, the two
equations would give simultaneously $(X/Y)^3=-1$ and $(X/Y)^4=-1$, which
is impossible; each curve meets $L$ transversally.  The assertions for
$(3,3)$ are proved below.

For the rows $(2,e)$, the conic $A=XY+Z^2$ and the Fermat curve
$B_e=X^e+Y^e+Z^e$ are smooth.  Suppose that they meet nontransversally.
If $Z=0$, then $A=0$ forces $XY=0$, which is incompatible with $B_e=0$.
Thus we may scale $Z=1$, so $XY=-1$.  Writing
$\nabla B_e=\lambda\nabla A$, the third component gives
$\lambda=e/2$, and the first two components give
\[
 2X^{e-1}=Y,\qquad 2Y^{e-1}=X.
\]
Multiplication by $X$ and $Y$, respectively, yields
$X^e=Y^e=-1/2$.  Hence
\[
 (XY)^e=(-1)^e=X^eY^e=\frac14,
\]
a contradiction for $e=5,6$.  On $L$, the conic meets the line at
$[1:0:0]$ and $[0:1:0]$, neither of which lies on $B_e$; the Fermat curve
also meets $L$ transversally.  Thus $A+B_e+L$ is SNC.

For $(1,8)$, let
\[
 B_8=X^8+Y^8+Z^8+X^4Y^4.
\]
At a singular point of $B_8$, the equation $(B_8)_Z=8Z^7=0$ gives $Z=0$.
Neither $X$ nor $Y$ can vanish, while the other two partial derivatives give
\[
 2X^4+Y^4=0,\qquad X^4+2Y^4=0.
\]
Their determinant is $3$, a contradiction.  Thus $B_8$ is smooth.  On the
line $A=\{Y=0\}$ its intersection equation is $X^8+Z^8=0$, so $X$ and $Z$
are nonzero and the intersection is transverse.  Smoothness likewise implies
transversality with $L=\{Z=0\}$, because $(B_8)_Z$ vanishes there.  Finally,
$A\cap L=[1:0:0]\notin B_8$, and hence $A+B_8+L$ is SNC.

For $(1,9)$, put
\[
 B_9=X^9+Y^9+Z^9+X^3Y^3Z^3.
\]
A singular point with a zero coordinate is impossible from the remaining
partial derivatives.  If all coordinates are nonzero, set
$a=X^3$, $b=Y^3$, and $c=Z^3$.  The three singularity equations become
\[
 3a^2+bc=3b^2+ac=3c^2+ab=0.
\]
Multiplication gives $27=-1$, again a contradiction.  Hence $B_9$ is smooth.
Its intersections with $A$ and $L$ are transverse by the same restriction
argument, and $A\cap L=[1:0:0]\notin B_9$.  Thus $A+B_9+L$ is SNC.

\subsection{A special symmetric pair of cubic curves}
We work with the following special transverse pair:
\begin{equation}\label{eq:symmetric-pair}
   A=X^3+2Y^3+Z^3,
   \qquad
   B=2X^3+Y^3+Z^3.
\end{equation}
Let
\[
   D=\{AB=0\}.
\]
Both cubics are smooth, and the two curves meet transversally.  Indeed, if
\[
   A=B=0,
\]
then
\[
   X^3=Y^3,
   \qquad
   3X^3+Z^3=0.
\]
Thus all intersection points have nonzero coordinates.  The two gradients
cannot be proportional: their $Z$-components would force the scalar to be
$1$, whereas the $X$-components would then require $3X^2=6X^2$.
Moreover, $A=B=Z=0$ has no projective solution, and each cubic meets
$L=\{Z=0\}$ transversally.  Hence $A+B+L$ is SNC.

The pair has two useful symmetries.  First, the involution
\begin{equation}\label{eq:sigma-involution}
   \sigma(X,Y,Z)=(Y,X,Z)
\end{equation}
satisfies
\[
   \sigma^*A=B,
   \qquad
   \sigma^*B=A.
\]
Second, for \(\alpha^3=\beta^3=1\), the diagonal transformation
\begin{equation}\label{eq:mu3-action}
   [X:Y:Z]\longmapsto[\alpha X:\beta Y:Z]
\end{equation}
fixes both \(A\) and \(B\).  Thus the finite-dimensional spaces appearing in the computation decompose into characters of \((\mu_3)^2\), and the involution \(\sigma\) exchanges the two character indices.

For the pairs $(2,e)$ in \eqref{eq:all-test-pairs}, the corresponding
symmetry is
\[
 [X:Y:Z]\longmapsto[\zeta X:\zeta^{-1}Y:Z],
 \qquad \zeta^e=1,
\]
together with the involution $X\leftrightarrow Y$.  The monomial
\[
 x^r y^s(x')^i(y')^jW_{xy}^k
\]
has character $r-s+i-j\pmod e$, because $W_{xy}$ has character zero.
Coordinate exchange sends the character $c$ to $-c$ and multiplies the
$k$-th Wronskian layer by $(-1)^k$.  Consequently the $e=5$ candidate
space splits into six blocks indexed by
$\{0\},\{1,4\},\{2,3\}$ and the two involution signs; for $e=6$ it
splits into eight blocks indexed by
$\{0\},\{1,5\},\{2,4\},\{3\}$ and the two signs.

For either line pair $(1,e)$ the same cyclic transformation
\[
 [X:Y:Z]\longmapsto[\zeta X:\zeta^{-1}Y:Z],
 \qquad \zeta^e=1,
\]
preserves both components.  There is no coordinate-exchange involution,
because the distinguished component is the line $Y=0$.  Since $a=y$ has
character $-1$ on the $Z$-chart and $b$ is invariant, a source monomial in
the $k$-th Wronskian layer has the shifted character
\begin{equation}\label{eq:line-character-weight}
 r-s+i-j+n_k\pmod e,
 \qquad n_k=m-2k.
\end{equation}
Thus the complete candidate space splits into exactly $e$ cyclic character
blocks.  The shift by $n_k$ is essential: it is contributed by the factor
$a^{-n_k}$ in the normal form.

\subsection{The ordinary and mixed logarithmic charts}
The ordinary chart below supports the finite ansatz and makes the finite
symmetries transparent.  The mixed chart is where we test logarithmic
regularity, the negative twist, and the overlaps between the two boundary
components.

Let
\begin{equation}\label{eq:U0-def}
   U_0=\{ZAB\neq0\}.
\end{equation}
On \(U_0\), put
\[
   x=\frac XZ,
   \qquad
   y=\frac YZ,
\]
and
\begin{equation}\label{eq:a0b0}
   a=x^3+2y^3+1,
   \qquad
   b=2x^3+y^3+1.
\end{equation}
The ordinary invariant two-jet frame on \(U_0\) is
\begin{equation}\label{eq:ordinary-frame-U0}
   x',\qquad y',\qquad W_{xy}:=x'y''-x''y'.
\end{equation}

The mixed logarithmic chart is taken in \(\{X\neq0\}\).  Put
\[
   u=\frac YX,
   \qquad
   z=\frac ZX.
\]
Then
\begin{equation}\label{eq:axbX}
   a_X=1+2u^3+z^3,
   \qquad
   b_X=2+u^3+z^3.
\end{equation}
Set
\begin{equation}\label{eq:qX-def}
   q=\left(\log\frac{a_X}{b_X}\right)'.
\end{equation}
A direct computation gives
\begin{equation}\label{eq:qX-explicit}
   q=
   \frac{3u^2(3+z^3)}{a_Xb_X}\,u'
   +
   \frac{3z^2(1-u^3)}{a_Xb_X}\,z'.
\end{equation}
Relative to the ordered ordinary frame \((z',u')\), the determinant of
\((z',q)\) is \(\Delta_X/(a_Xb_X)\), where the numerator is
\begin{equation}\label{eq:DeltaX-def}
   \Delta_X:=a_{X,u}b_X-a_Xb_{X,u}=3u^2(3+z^3).
\end{equation}
We define
\begin{equation}\label{eq:UX-def}
   U_X=\{X\neq0,\ u(3+z^3)\neq0\}.
\end{equation}
On this chart, the logarithmic invariant two-jet frame is
\begin{equation}\label{eq:mixed-frame-UX}
   z',\qquad q,
   \qquad
   W_X:=qz''-q'z'.
\end{equation}
The two sets \(U_0\) and \(U_X\) cover the complement of a finite set.  Indeed,
\[
   \PP^2\setminus(U_0\cup U_X)
   \subset
   \{ZAB=0\}\cap\{XY(3X^3+Z^3)=0\},
\]
and the right-hand side is finite because none of the curves \(Z=0\), \(A=0\), or \(B=0\) is a component of \(XY(3X^3+Z^3)=0\).  Since the jet differential sheaves used here are locally free on the normal crossing locus, it is enough to test regularity on \(U_0\cup U_X\).

For the unequal-degree rows of \eqref{eq:all-test-pairs}, the same
construction uses the projectively invariant logarithmic form.  Put
\[
 g=\gcd(d_1,d_2),\qquad \alpha=\frac{d_2}{g},\qquad
 \beta=\frac{d_1}{g},
\]
and on the $X$-chart define
\begin{equation}\label{eq:q-general-degrees}
 q=\alpha\,d\log a_X-\beta\,d\log b_X,
 \qquad
 \Delta=\alpha a_{X,u}b_X-\beta a_Xb_{X,u},
 \qquad
 E=\alpha a_{X,z}b_X-\beta a_Xb_{X,z}.
\end{equation}
Then
\begin{equation}\label{eq:q-general-solve}
 q=\frac{\Delta u'+Ez'}{a_Xb_X},
 \qquad
 u'=\frac{a_Xb_X}{\Delta}q-\frac E\Delta z'.
\end{equation}
The frame generator may replace \(q\) by any nonzero rational scalar
multiple without changing the kernel problem over \(\mathbb Q\).  The
implementation uses this freedom to keep the mixed-chart numerators
integral.  For the diagonal nonlinear models it uses
\[
 q_{\rm cmp}=\frac1{d_1}d\log a_X-
             \frac1{d_2}d\log b_X
           =\frac{g}{d_1d_2}\,q.
\]
For the mixed conic model \(a_X=u+z^2\), of bidegree \((2,e)\), it instead
uses
\[
 q_{\rm cmp}=d\log a_X-\frac2e d\log b_X=\frac ge q,
 \qquad g=\gcd(2,e).
\]
Thus $q_{\rm cmp}=q/5$ for $e=5$, whereas $q_{\rm cmp}=q/3$ for $e=6$.
This is the normalization whose degree factors cancel in the corresponding
mixed numerators.  For the two line models the program uses the integral
form \(q\) displayed below.  Thus in every model the resulting
\(\Delta_{\rm cmp}\), \(E_{\rm cmp}\), and every matrix entry produced by
the program are integers.  Each rescaling is an invertible change of frame
over \(\mathbb Q\); it is not being asserted to be invertible after
reduction modulo every prime.
Instead, the program first constructs its own integral matrix in this
normalized frame.  Full column rank of that integral matrix modulo one prime
implies full column rank over \(\mathbb Q\), and the rational change of frame
then identifies its kernel with the kernel of the invariant form used in
\eqref{eq:q-general-degrees}.  In particular, the proof-effective prime
\(5\) for the conic--quintic case is legitimate even though \(5\) divides
the denominator in the rational rescaling \(q_{\rm cmp}=q/5\).
The order is important: the change of frame is made over characteristic
zero, the normalized matrix is then reconstructed with integral entries, and
only this integral matrix is reduced modulo $5$.  No occurrence of the
rational scalar $1/5$ is reduced in $\mathbb F_5$.

Likewise, the program's cyclic blocks are defined directly by congruences of
integer monomial exponents and, where present, by the explicit exchange
involution.  Their construction does not invoke semisimplicity of a
finite-group representation over the proof field.  Semisimplicity is used
only in characteristic zero to show that these blocks exhaust the candidate
space.  Hence a proof prime may divide the order of the displayed symmetry
group without invalidating the integral block computation.
For the two line models these formulas are especially simple.  When $e=8$,
\begin{equation}\label{eq:line-octic-overlap}
 \begin{gathered}
  a_X=u,\qquad b_X=1+u^8+z^8+u^4,\qquad
  q=8\,d\log a_X-d\log b_X,\\
  \Delta=8+8z^8+4u^4,\qquad E=-8uz^7.
 \end{gathered}
\end{equation}
When $e=9$,
\begin{equation}\label{eq:line-nonic-overlap}
 \begin{gathered}
  a_X=u,\qquad b_X=1+u^9+z^9+u^3z^3,\qquad
  q=9\,d\log a_X-d\log b_X,\\
  \Delta=9+9z^9+6u^3z^3,\qquad
  E=-9uz^8-3u^4z^2.
 \end{gathered}
\end{equation}
In both cases $z,u,b_X,\Delta$ are pairwise coprime.  For $e=8$, put
$\delta_0=\Delta/4$.  The identity
\[
 2b_X-\delta_0=u^4(2u^4+1)
\]
reduces the only nonobvious gcd to $u$ and the factors of $2u^4+1$; the
former does not divide $b_X$, and no nonconstant polynomial in $u$ divides
the polynomial $b_X$, which is monic in $z$.  For $e=9$, a common factor of
$b_X$ and $\Delta=9b_X-u(b_X)_u$ would, after excluding $u$, divide
$3u^6+z^3$.  On each component $z=\lambda u^2$, $\lambda^3=-3$, the
restriction of $b_X$ is
\[
 1-2u^9-27u^{18},
\]
which is not identically zero.  This proves the asserted coprimality.

The coefficients $\alpha,\beta$ are exactly what make $q$ independent of
the projective trivialization.  For every test pair in
\eqref{eq:all-test-pairs}, none of $Z,A,B$ is a component of the homogeneous
polynomial $X\Delta$.  Hence
\[
 U_0=\{ZAB\neq0\},\qquad U_X=\{X\Delta\neq0\}
\]
again cover the complement of a finite set.  The general logarithmic frame
is $(z',q,W_X)$ with $W_X=qz''-q'z'$.  Formulas
\eqref{eq:qX-def}--\eqref{eq:DeltaX-def} are the specialization
$d_1=d_2=3$.

\subsection{An auxiliary logarithmic compactification and a finite ambient normal form}
The ordinary frame on $U_0$ is convenient for symmetry, but it is not adapted to the line at infinity.  In particular, the negative twist cannot be encoded by simply subtracting $t$ from the degree of each numerator in the ordinary frame: cancellations between different ordinary basis terms may remove the apparent poles at infinity.  We therefore separate the finite ambient bound from the twisting condition.

Let
\[
   L:=\{Z=0\},
   \qquad
   D^+:=D+L.
\]
For every test pair in \eqref{eq:all-test-pairs}, the divisor $D^+$ is
simple normal crossing.  Multiplication by the section
$Z^t\in H^0(\PP^2,\OO_{\PP^2}(t))$ gives an injection
\begin{equation}\label{eq:twist-injection}
   E_{2,m}T_{\PP^2}^*(\log D)\otimes\OO_{\PP^2}(-t)
   \xrightarrow{\ \cdot Z^t\ }
   E_{2,m}T_{\PP^2}^*(\log D)
   \hookrightarrow
   E_{2,m}T_{\PP^2}^*(\log D^+).
\end{equation}
The image of the first arrow consists of logarithmic two-jet differentials for $D$ that vanish to order at least $t$ along $L$.  Since $Z=1$ on the affine chart used to define $U_0$, multiplication by $Z^t$ does not change the local expression on $U_0$.  Thus we may first generate a finite ambient space inside the larger logarithmic bundle for $D^+$ and impose the removal of the extra logarithmic pole along $L$, together with the order-$t$ vanishing, on $U_X$.

We begin with the local pole estimate used for all three components of $D^+$.

\begin{lemma}[Layerwise logarithmic pole bound]\label{lem:local-layer-pole}
Let $s=0$ be a smooth local branch of a logarithmic divisor on a surface, and let $v$ be a transverse coordinate.  Put
\[
   \lambda=\frac{s'}s,
   \qquad
   W_{sv}=s'v''-s''v'.
\]
Then the logarithmic Wronskian satisfies the triangular identity
\begin{equation}\label{eq:local-log-W}
   \lambda v''-\lambda'v'
   =\frac1sW_{sv}+\frac{(s')^2v'}{s^2}.
\end{equation}
Consequently, in weighted degree $m$, the coefficient of the ordinary layer $W_{sv}^k$ has pole order at most
\[
   n_k:=m-2k
\]
along $s=0$.  More precisely, the coefficient of
\[
   (s')^p(v')^{m-3k-p}W_{sv}^k
\]
has pole order at most $p+k$.
\end{lemma}

\begin{proof}
Consider a logarithmic monomial
\[
   \lambda^a(v')^b(\lambda v''-\lambda'v')^r,
   \qquad
   a+b=m-3r.
\]
A contribution to the ordinary Wronskian layer $W_{sv}^k$ can occur only when $r\geqslant k$.  Choose the leading term $s^{-1}W_{sv}$ in exactly $k$ of the $r$ Wronskian factors and the lower term $s^{-2}(s')^2v'$ in the remaining $r-k$ factors.  The pole order is then
\[
   a+k+2(r-k).
\]
The resulting number of factors $s'$ outside $W_{sv}^k$ is
\[
   p=a+2(r-k),
\]
so the pole order equals $p+k$.  Since $p\leqslant m-3k$, this is at most $m-2k$.  The statement follows by linearity.
\end{proof}

Write
\[
   K=\floor{m/3},
   \qquad
   n_k=m-2k,
   \qquad
   s_k=m-3k,
   \qquad
   d_k^+=(s-1)m-(2s-1)k.
\]

\begin{proposition}[Finite ambient normal form on $U_0$]\label{prop:finite-ambient-normal-form}
Every global section of
\[
 E_{2,m}T_{\PP^2}^*(\log D^+)
\]
restricts to $U_0$ in the form
\begin{equation}\label{eq:U0-normal-form}
   \omega_0
   =
   \sum_{k=0}^{K}
   \sum_{i+j=s_k}
   \frac{P_{i,j,k}(x,y)}{a^{n_k}b^{n_k}}
   (x')^i(y')^jW_{xy}^{k},
\end{equation}
where
\begin{equation}\label{eq:degree-bound-P}
   P_{i,j,k}\in\CC[x,y],
   \qquad
   \deg P_{i,j,k}\leqslant d_k^+=(s-1)m-(2s-1)k.
\end{equation}
The exponent $n_k=m-2k$ and the degree bound $d_k^+$ are necessary ambient bounds.  Arbitrary polynomials satisfying \eqref{eq:degree-bound-P} need not define a global section; the missing compatibility conditions are imposed on $U_X$ in the next subsection.
\end{proposition}

\begin{proof}
Apply Lemma~\ref{lem:local-layer-pole} to the branches $a=0$ and $b=0$.  The coefficient of $W_{xy}^k$ has pole order at most $n_k$ along each branch.  After multiplication by $a^{n_k}b^{n_k}$, every coefficient is regular on the affine plane, hence is a polynomial.  This proves the denominator assertion.

It remains to control the degree by using the extra logarithmic component $L$.  On the $X$-chart write
\[
   z=\frac ZX,
   \qquad
   u=\frac YX,
   \qquad
   x=\frac1z,
   \qquad
   y=\frac uz.
\]
The ordinary jets satisfy
\begin{equation}\label{eq:ordinary-chart-transition}
   x'=-\frac{z'}{z^2},
   \qquad
   y'=\frac{zu'-uz'}{z^2},
   \qquad
   W_{xy}=-\frac{W_{zu}}{z^3},
\end{equation}
where $W_{zu}=z'u''-z''u'$.  Moreover,
\[
   a=\frac{a_X}{z^{d_1}},
   \qquad
   b=\frac{b_X}{z^{d_2}}.
\]
Thus the $k$-th layer of \eqref{eq:U0-normal-form} becomes
\begin{align}\label{eq:k-layer-at-infinity}
   &\frac{z^{(s-2)m-(2s-3)k}}{a_X^{n_k}b_X^{n_k}}
   \sum_{i+j=s_k}
   (-1)^{i+k}P_{i,j,k}(z^{-1},uz^{-1})
   (z')^i(zu'-uz')^jW_{zu}^k.
\end{align}
For $0\leqslant\ell\leqslant s_k$, the coefficient of
\[
   (z')^{s_k-\ell}(u')^\ell W_{zu}^k
\]
is a triangular linear combination of the polynomials $P_{s_k-j,j,k}(z^{-1},uz^{-1})$ with $j\geqslant\ell$; its diagonal term is a nonzero scalar multiple of
\[
   z^{(s-2)m-(2s-3)k+\ell}
   P_{s_k-\ell,\ell,k}(z^{-1},uz^{-1}).
\]
By the refined part of Lemma~\ref{lem:local-layer-pole}, a section logarithmic along $L=\{z=0\}$ may have pole order at most
\[
   (s_k-\ell)+k
\]
in this coefficient.  Starting with $\ell=s_k$ and descending in $\ell$, the triangularity shows that no polynomial $P_{s_k-\ell,\ell,k}$ can have degree exceeding
\[
   (s-2)m-(2s-3)k+\ell+(s_k-\ell)+k
   =(s-1)m-(2s-1)k.
\]
This proves \eqref{eq:degree-bound-P}.
\end{proof}

The important point is that the negative twist $t$ does not appear in \eqref{eq:degree-bound-P}.  The degree bound describes a finite ambient space with an allowed logarithmic pole along $L$.  The twist is imposed by requiring the transformed coefficients to lose that pole and to vanish to order $t$ along $L$.

\subsection{Regularity and twisting test on \texorpdfstring{$U_X$}{UX}}
On $U_0\cap U_X$ one has
\begin{equation}\label{eq:U0UX-coordinates}
   x=\frac1z,
   \qquad
   y=\frac uz,
   \qquad
   a=\frac{a_X}{z^{d_1}},
   \qquad
   b=\frac{b_X}{z^{d_2}}.
\end{equation}
The first derivatives transform as
\begin{equation}\label{eq:xprime-yprime-to-UX}
   x'=-\frac{z'}{z^2},
   \qquad
   y'=\frac{u'z-uz'}{z^2}.
\end{equation}
Set
\[
   H:=a_Xb_X,
   \qquad
   E:=\alpha a_{X,z}b_X-\beta a_Xb_{X,z},
   \qquad
   \Delta_X:=\alpha a_{X,u}b_X-\beta a_Xb_{X,u}.
\]
Equation \eqref{eq:qX-explicit} gives
\begin{equation}\label{eq:solve-uprime}
   u'=\frac{H}{\Delta_X}q-\frac{E}{\Delta_X}z'.
\end{equation}
The pair $(z',q)$ is a logarithmic one-jet frame for $D$ on $U_X$, and
\[
   W_X=qz''-q'z'
\]
is the associated invariant Wronskian.  Substituting \eqref{eq:U0UX-coordinates}, \eqref{eq:xprime-yprime-to-UX}, and \eqref{eq:solve-uprime} into \eqref{eq:U0-normal-form} and collecting in the mixed frame gives
\begin{equation}\label{eq:UX-collected}
   \omega_0|_{U_X}
   =
   \sum_{k=0}^{K}
   \sum_{j=0}^{s_k}
   R_{j,k}(u,z)
   (z')^{s_k-j}q^jW_X^k.
\end{equation}

\begin{proposition}[Exact extension and twisting criterion]\label{prop:extension-criterion}
A candidate \eqref{eq:U0-normal-form} is the restriction of a section
\[
   \omega\in
   H^0\!\left(\PP^2,
   E_{2,m}T_{\PP^2}^*(\log D)\otimes\OO_{\PP^2}(-t)\right)
\]
if and only if every coefficient in \eqref{eq:UX-collected} satisfies
\begin{equation}\label{eq:ring-criterion}
   R_{j,k}\in z^t\CC[u,z,\Delta_X^{-1}].
\end{equation}
Equivalently, after clearing a fixed bounded denominator, the condition is a finite collection of polynomial divisibility equations that are homogeneous and linear in the coefficients of the $P_{i,j,k}$.
\end{proposition}

\begin{proof}
The mixed frame $(z',q,W_X)$ is regular for the logarithmic jet bundle associated with $D$ on $U_X$.  Hence a section of the untwisted bundle is regular on $U_X$ exactly when its scalar coefficients lie in
\[
   \OO(U_X)=\CC[u,z,\Delta_X^{-1}].
\]
Under the injection \eqref{eq:twist-injection}, the image of a section twisted by $\OO(-t)$ vanishes to order at least $t$ along $L=\{z=0\}$.  This is precisely condition \eqref{eq:ring-criterion}.

Since $U_0\cup U_X$ misses only finitely many points and the jet bundle is locally free, a section satisfying this condition extends uniquely across the missing codimension-two set.  This proves the equivalence.
\end{proof}

For computation, the denominator clearing can be made uniform and explicit.  Write
\[
   u'=Aq+Bz',\qquad A=H/\Delta_X,\qquad B=-E/\Delta_X.
\]
Differentiating with respect to the source variable and substituting
$u'=Aq+Bz'$ gives the exact identity
\begin{align}\label{eq:u-second-exact}
 u''={}&Aq'+Bz''+AA_uq^2\\
 &+(A_uB+A_z+AB_u)qz'
 +(BB_u+B_z)(z')^2.
\end{align}
Consequently,
\begin{equation}\label{eq:Wzu-to-WX-triangular}
 \begin{split}
 W_{zu}={}&-A W_X+AA_uq^2z'\\
 &+(A_uB+A_z+AB_u)q(z')^2
 +(BB_u+B_z)(z')^3.
 \end{split}
\end{equation}
This is an identity in the localized differential-polynomial ring.  Since
$A,B$ have one power of $\Delta_X$ in the denominator, their first
partial derivatives have at most two, and every coefficient in
\eqref{eq:Wzu-to-WX-triangular} has $\Delta_X$-denominator order at most
three.  A source layer $W_{zu}^r$ can contribute to the target layer
$W_X^k$ only when $r\geqslant k$.  The $m-3r$ first-jet factors contribute at
most one power of $\Delta_X^{-1}$ each, the $k$ leading Wronskian factors
contribute one power each, and the remaining $r-k$ lower Wronskian
factors contribute at most three powers each.  Therefore the total
exponent of $\Delta_X$ is bounded by
\begin{equation}\label{eq:Delta-bound}
   (m-3r)+k+3(r-k)=m-2k=n_k.
\end{equation}
No substitution introduces $a_X$ or $b_X$ in a denominator beyond the
original source exponent, and $n_r\leqslant n_k$ for $r\geqslant k$.

We also record the $z$-bound explicitly.  Before replacing $u'$ and
$W_{zu}$ by the mixed frame, a monomial from source layer $r$ with a
numerator of degree $d$ carries the factor
\[
 z^{sn_r-2(m-3r)-3r-d}
 =z^{(s-2)m-(2s-3)r-d}.
\]
Since $d\leqslant(s-1)m-(2s-1)r$, this exponent is at least
\[
 (s-2)m-(2s-3)r-\bigl((s-1)m-(2s-1)r\bigr)
 =-(m-2r)=-n_r.
\]
The substitutions \eqref{eq:solve-uprime} and
\eqref{eq:Wzu-to-WX-triangular} introduce no additional powers of
$z^{-1}$.  Hence the pole order in $z$ is at most $n_r\leqslant n_k$.
Consequently,
\begin{equation}\label{eq:cleared-polynomial}
   \mathcal N_{j,k}
   :=z^{n_k}a_X^{n_k}b_X^{n_k}\Delta_X^{n_k}R_{j,k}
   \in\CC[u,z].
\end{equation}
The same valuation argument applies to every row of
\eqref{eq:all-test-pairs}, including the unequal-degree rows.  Indeed,
$A+B+L$ is SNC, so $z,a_X,b_X$ are pairwise coprime, and the check following
\eqref{eq:q-general-solve} says precisely that
\[
 \gcd(z\,a_Xb_X,\Delta_X)=1.
\]
Thus localizing at $\Delta_X$ changes none of the valuations along the prime
factors of $z\,a_Xb_X$.  This is the required bridge from the general
coprimality $\gcd(A,B)=1$ to the concrete divisibility test.

For example, for the special pair \eqref{eq:symmetric-pair}, the four factors
\[
   z,\qquad a_X,\qquad b_X,\qquad \Delta_X
\]
are pairwise coprime in $\CC[u,z]$.  Indeed, $z$ divides none of the
other three.  If an irreducible factor divided both $a_X$ and $b_X$, it
would divide $a_X-b_X=u^3-1$; after setting $u^3=1$, however, both
polynomials become $3+z^3$, so no factor of $u^3-1$ divides them
identically.  Moreover, $u$ divides neither $a_X$ nor $b_X$.  Finally,
for every root $\zeta$ of $\zeta^3=-3$ one has
\[
   a_X(u,\zeta)=2(u^3-1),
   \qquad
   b_X(u,\zeta)=u^3-1,
\]
which are not the zero polynomial in $u$; hence no irreducible factor of
$3+z^3$ divides $a_X$ or $b_X$.  Since $\CC[u,z]$ is a unique
factorization domain, this verifies the general coprimality statement
directly in the cubic case.  The valuation conditions along $z=0$, $a_X=0$,
and $b_X=0$ can therefore be combined for every test pair.  Hence
\eqref{eq:ring-criterion} is equivalent to the concrete divisibility
condition
\begin{equation}\label{eq:divisibility-criterion}
   z^{n_k+t}a_X^{n_k}b_X^{n_k}
   \ \bigm|\ 
   \mathcal N_{j,k}
   \qquad
   \text{in }\CC[u,z].
\end{equation}
For the conic and cubic models, the fixed product in
\eqref{eq:divisibility-criterion} may be handled by the corresponding exact
polynomial-remainder test.  For a line pair, however, $a_X=u$, and treating
the whole
product as a univariate polynomial in $z$ would give a nonconstant leading
coefficient.  We instead use pairwise coprimality and impose the following
three equivalent conditions on every cleared numerator $\mathcal N_{j,k}$:
\begin{enumerate}[label=\textup{(\roman*)}]
 \item no monomial has $z$-degree below $n_k+t$;
 \item no monomial has $u$-degree below $n_k$;
 \item the remainder modulo the monic polynomial $b_X^{n_k}$ in
 $K[u][z]$ is zero, where $K=\CC$ for the characteristic-zero system and
 $K=\mathbb F_p$ for a certificate matrix.
\end{enumerate}
The three row families are kept disjoint in the certificate.  Hence all
divisibility conditions are imposed simultaneously, and each is homogeneous
and linear in the coefficients of the $P_{i,j,k}$.  The triangularity in the
Wronskian degree permits the equations to be generated from the largest $k$
downward.

\subsection{The involution and character reductions}
The involution $\sigma(X,Y,Z)=(Y,X,Z)$ preserves $U_0$ and acts on the ordinary frame by
\[
   x\leftrightarrow y,
   \qquad
   x'\leftrightarrow y',
   \qquad
   W_{xy}\longmapsto -W_{xy}.
\]
If one fixes an eigenvalue $\varepsilon=\pm1$ and imposes
\begin{equation}\label{eq:sigma-eigencondition}
   \sigma^*\omega_0=\varepsilon\omega_0,
\end{equation}
then the coefficients in \eqref{eq:U0-normal-form} satisfy
\begin{equation}\label{eq:sigma-coefficients}
   P_{i,j,k}(x,y)=\varepsilon(-1)^kP_{j,i,k}(y,x).
\end{equation}
Thus only one representative from each orbit
\[
   (i,j,x^r y^s)\longleftrightarrow(j,i,x^s y^r)
\]
needs to be generated.  If an orbit is fixed and $\varepsilon(-1)^k=-1$, its coefficient is forced to be zero.

The diagonal $(\mu_3)^2$-action gives a further decomposition.  For
\[
   [X:Y:Z]\longmapsto[\alpha X:\beta Y:Z],
   \qquad
   \alpha^3=\beta^3=1,
\]
one has
\[
   x\mapsto\alpha x,
   \qquad
   y\mapsto\beta y,
   \qquad
   x'\mapsto\alpha x',
   \qquad
   y'\mapsto\beta y',
   \qquad
   W_{xy}\mapsto\alpha\beta W_{xy}.
\]
Therefore the monomial
\[
   x^r y^s(x')^i(y')^jW_{xy}^{k}
\]
has character
\begin{equation}\label{eq:character-weight}
   (r+i+k,\;s+j+k)\pmod 3.
\end{equation}
For a fixed character $(\chi_1,\chi_2)\in(\mathbb Z/3\mathbb Z)^2$, we generate only the monomials satisfying
\begin{equation}\label{eq:character-filter}
   r+i+k\equiv\chi_1\pmod3,
   \qquad
   s+j+k\equiv\chi_2\pmod3.
\end{equation}
The involution $\sigma$ exchanges the character $(\chi_1,\chi_2)$ with $(\chi_2,\chi_1)$.  Hence the nine characters split into six orbits:
\begin{equation}\label{eq:character-orbits}
   (0,0),\ (1,1),\ (2,2),\quad
   (0,1)\leftrightarrow(1,0),\quad
   (0,2)\leftrightarrow(2,0),\quad
   (1,2)\leftrightarrow(2,1).
\end{equation}
For diagonal characters one imposes $\sigma$-eigenconditions inside the character space.  For an off-diagonal orbit, an eigenvector is determined by its component in either one of the two exchanged character spaces.

\begin{lemma}[The cubic symmetry blocks exhaust the candidate space]
\label{lem:symmetry-exhaustion}
Let $V$ be a finite-dimensional candidate space in the above normal form
that is stable under both $(\mu_3)^2$ and the involution $\sigma$, and let
$K\subset V$ be the kernel of the extension and twisting equations.
Then $K=0$ if and only if its intersection with every one of the twelve
character--eigenvalue blocks indexed by
\eqref{eq:character-orbits} and $\varepsilon\in\{+1,-1\}$ is zero.
\end{lemma}

\begin{proof}
Over $\CC$, the finite abelian group $(\mu_3)^2$ acts semisimply, so
$V=\bigoplus_\chi V_\chi$.  By Proposition~\ref{prop:extension-criterion},
the extension and twisting conditions characterize global sections; since
the symmetries preserve $D$ and $L$, their kernel is stable and has the same
character decomposition.  The involution sends $V_\chi$ to
$V_{\sigma\chi}$.  If $\chi=\sigma\chi$, the involution is diagonalizable
with eigenvalues $\pm1$.  If $\chi\neq\sigma\chi$ and
$0\neq v\in K\cap V_\chi$, then at least one of
$v+\sigma v$ and $v-\sigma v$ is nonzero; these vectors lie respectively
in the $+1$ and $-1$ eigenspaces of
$V_\chi\oplus V_{\sigma\chi}$.  Thus every nonzero kernel contains a
nonzero vector in one of the listed blocks.  The converse is immediate.
\end{proof}

\begin{lemma}[The conic-pair symmetry blocks exhaust the candidate space]
\label{lem:conic-symmetry-exhaustion}
For either model $(2,e)$, let $V$ be a finite candidate space stable under
the cyclic action $\mu_e$ and coordinate exchange, and let $K\subset V$
be the kernel of the extension and twisting equations.  Then
$K=0$ if and only if its intersection with every character--eigenvalue block
indexed by an orbit of $c\mapsto-c$ in $\mathbb Z/e\mathbb Z$ and by
$\varepsilon\in\{+1,-1\}$ is zero.
\end{lemma}

\begin{proof}
Over $\CC$, the cyclic group $\mu_e$ acts semisimply.  The cyclic action and
coordinate exchange preserve $D$ and $L$, so
Proposition~\ref{prop:extension-criterion} makes $K$ stable under both and gives
\(K=\bigoplus_c(K\cap V_c)\).  Coordinate exchange sends $V_c$ to
$V_{-c}$ and preserves $K$.  On a fixed character ($c=-c$), it is
diagonalizable with eigenvalues $\pm1$.  On a two-element orbit
$\{c,-c\}$, every nonzero $v\in K\cap V_c$ yields a nonzero vector among
$v+\sigma v$ and $v-\sigma v$ in the corresponding $+1$ or $-1$ block.
Thus a nonzero kernel meets one of the enumerated blocks, and the converse is
immediate.
\end{proof}

\begin{lemma}[The line-pair character blocks exhaust the candidate space]
\label{lem:line-symmetry-exhaustion}
For either line model $(1,e)$, let $V$ be a finite candidate space stable
under the cyclic action $\mu_e$, and let $K\subset V$ be the kernel of the
extension and twisting equations.  Then
$K=0$ if and only if its intersection with each of the $e$ character spaces
defined by \eqref{eq:line-character-weight} is zero.
\end{lemma}

\begin{proof}
The cyclic group $\mu_e$ acts semisimply in characteristic zero.  It preserves
$D$ and $L$, so Proposition~\ref{prop:extension-criterion} makes $K$ stable.
Both $V$ and $K$ are
therefore direct sums of their $e$ character spaces.
\end{proof}

\begin{proposition}[Finite matrix equivalence]\label{prop:finite-matrix-equivalence}
Fix one of the test pairs in \eqref{eq:all-test-pairs} and integers
$m,t\geqslant1$.  Let $V_m$ be the finite coefficient space in
\eqref{eq:U0-normal-form}--\eqref{eq:degree-bound-P}, and let
\[
   \Phi_{m,t}\colon V_m\longrightarrow W_{m,t}
\]
record the coefficients of all forbidden remainders in
\eqref{eq:divisibility-criterion}; for a line pair, $\Phi_{m,t}$ records
the three row families listed after that equation.  Restriction to $U_0$
induces a natural isomorphism
\[
 H^0\!\left(\PP^2,
 E_{2,m}T_{\PP^2}^*(\log D)\otimes\OO_{\PP^2}(-t)\right)
 \simeq \ker\Phi_{m,t}.
\]
Moreover, the symmetry decompositions above identify this kernel with the
direct sum of the kernels of the corresponding block matrices.
\end{proposition}

\begin{proof}
Proposition~\ref{prop:finite-ambient-normal-form} sends every global section
to a unique vector of $V_m$.  Proposition~\ref{prop:extension-criterion},
together with the coprimality and denominator bounds proved above, says that
such a vector comes from a twisted global section exactly when all the
remainders recorded by $\Phi_{m,t}$ vanish.  The resulting section on
$U_0\cup U_X$ extends uniquely across the finite complement.  This proves
the first assertion.  The complete space $V_m$ satisfies the stability
hypotheses of the symmetry lemmas: the diagonal actions multiply each
normal-form monomial by a scalar, while coordinate exchange permutes the
monomials, preserves their total-degree bounds, and preserves the denominator
$a^{n_k}b^{n_k}$ whenever that involution is used.  Its kernel is also stable,
since the symmetries preserve $D$ and $L$ and hence preserve the global
extension and twisting conditions.  These facts and
Lemmas~\ref{lem:symmetry-exhaustion},
\ref{lem:conic-symmetry-exhaustion}, and
\ref{lem:line-symmetry-exhaustion} prove the second.
\end{proof}

The implementation also checks equivariance after the full chart
transformation.  If a cleared target row contains the monomial $u^Uz^Z$, has
$q$-degree $Q$, and lies in target Wronskian layer $k$, its character in the
source convention is
\begin{equation}\label{eq:line-target-character}
 -2U-Z+2(m-2k)+Q\pmod e.
\end{equation}
For every complete transformed column in the $m=3$ self-test, each nonzero
target row has the same character as the source column.  This detects, in
particular, omission of the denominator shift in
\eqref{eq:line-character-weight}.

\subsection{Counting the ambient variables}
Let
\[
   M(d)=\binom{d+2}{2}
\]
be the number of monomials in two variables of total degree at most $d$.  Since the negative twist is imposed by the equations on $U_X$, rather than by deleting monomials on $U_0$, the ambient variable count depends on $m$ but not on $t$:
\begin{equation}\label{eq:raw-variable-count}
   N_{\mathrm{raw}}^+(m;s)
   =
   \sum_{k=0}^{\floor{m/3}}
   (m-3k+1)
   \binom{(s-1)m-(2s-1)k+2}{2}.
\end{equation}
For a line pair $(1,e)$, so that $s=e+1$, this specializes to
\begin{equation}\label{eq:line-raw-variable-count}
 N_e(m)=
 \sum_{0\leqslant k\leqslant\floor{m/3}}
 (m-3k+1)\binom{em-(2e+1)k+2}{2}.
\end{equation}
The exact counts for the eight line-pair cases are
\begin{equation}\label{eq:line-variable-count-table}
\begin{array}{c|c|r|r|r}
 (d_1,d_2)&(m,t)&N_e(m)&\text{blocks}&\text{largest block}\\ \hline
 (1,8)&(3,2)&1{,}336&8&174\\
 (1,8)&(4,2)&3{,}077&8&398\\
 (1,8)&(5,2)&6{,}066&8&779\\
 (1,8)&(6,3)&10{,}807&8&1{,}382\\
 (1,8)&(7,3)&17{,}876&8&2{,}279\\
 (1,8)&(8,3)&27{,}849&8&3{,}543\\
 (1,9)&(3,3)&1{,}669&9&186\\
 (1,9)&(4,3)&3{,}857&9&429
\end{array}
\end{equation}
Thus the line--octic calculation has $48$ blocks and $67{,}011$ ambient
columns in total; the line--nonic calculation has $18$ blocks and $5{,}526$
columns.

The remainder of this subsection specializes to $(d_1,d_2)=(3,3)$, so
$s=6$ and the degree bound becomes $5m-11k$.

For a fixed character $(\chi_1,\chi_2)$, define
\begin{equation}\label{eq:M-character}
   M_d(a,b)
   :=
   \#\{(r,s)\in\mathbb Z_{\geqslant0}^2:
   r+s\leqslant d,
   \ r\equiv a\pmod3,
   \ s\equiv b\pmod3\}.
\end{equation}
Then
\begin{equation}\label{eq:character-variable-count}
   N_{\chi_1,\chi_2}^+(m)
   =
   \sum_{k=0}^{\floor{m/3}}
   \sum_{i+j=m-3k}
   M_{5m-11k}(\chi_1-i-k,\chi_2-j-k),
\end{equation}
where all congruences are understood modulo $3$.  For an off-diagonal character orbit, the number of variables in a $\sigma$-eigensystem is the number in either exchanged character.  For a diagonal character $(\chi,\chi)$, the $\sigma$-even and $\sigma$-odd dimensions are
\begin{equation}\label{eq:diagonal-eigenspace-count}
   N_{\chi}^{\pm,+}(m)
   =
   \frac12\left(N_{\chi,\chi}^+(m)\pm T_\chi^+(m)\right),
\end{equation}
where
\begin{equation}\label{eq:trace-term}
   T_\chi^+(m)
   =
   \sum_{\substack{0\leqslant k\leqslant\floor{m/3}\\ m-3k\text{ even}}}
   (-1)^k
   \#\left\{r\geqslant0:
   2r\leqslant5m-11k,
   \ r+\frac{m-3k}{2}+k\equiv\chi\pmod3
   \right\}.
\end{equation}

For each twist $t$, the following table records the endpoint $m=5t-1$ of the
certified cubic range.  The last column is the largest block after both the
character and involution decompositions.
\begin{equation}\label{eq:variable-count-table}
\begin{array}{c|c|r|r}
 t & m & N_{\mathrm{raw}}^+ & \max N_{\mathrm{sym}}^+ \\ \hline
 1 & 4  & 1,265   & 156 \\
 2 & 9  & 16,511  & 1,924 \\
 3 & 14 & 76,905  & 8,830
\end{array}
\end{equation}
Thus the largest ambient candidate space occurs at $(m,t)=(14,3)$ and has
$76,905$ variables.  After the $(\mu_3)^2$-character decomposition and the
involution, the largest block has $8,830$ variables.  The raw ambient total
over all pairs
\[
   1\leqslant t\leqslant3,
   \qquad 1\leqslant m\leqslant5t-1
\]
is $326,532$.  The certificate computation instead uses one least twist
$t_0(m)=\lceil(m+1)/5\rceil$ for each $1\leqslant m\leqslant14$, including
the low-weight self-tests $m=1,2$; the raw ambient total for these fourteen
cases is $280,798$, as recorded in Table~\ref{tab:finite-rank-certificates}.
The least-twist vanishing covers the larger twists by the injection explained
after \eqref{eq:cubic-minimal-pairs}.
These variables are never placed in one matrix: the
computation splits by $(m,t)$, character orbit, and involution eigenvalue.
Within each such block, the Wronskian layers organize the downward generation
of equations, but the rank test is performed on the full coupled matrix
containing all layers.  Although the ambient variables do not depend on $t$,
the divisibility equations \eqref{eq:divisibility-criterion} become stronger
as $t$ increases.

\subsection{Exact modular rank verification}\label{subsec:modular-rank}

For each degree pair and each $(m,t)$ in its finite range, the computation is
as follows.
\begin{enumerate}[label=\textup{(\arabic*)}]
   \item Use the injection \eqref{eq:twist-injection} and generate the finite
   ambient normal form \eqref{eq:U0-normal-form} on $U_0$, with denominator
   exponent $m-2k$ and degree bound $(s-1)m-(2s-1)k$.  Do not subtract $t$
   from this degree bound.
   \item For $(3,3)$, split the ambient variables according to the
   $(\mu_3)^2$-characters and impose the involution eigencondition.  For
   $(2,e)$ use the cyclic character and involution blocks described above.
   For $(1,e)$ use exactly the $e$ shifted cyclic characters in
   \eqref{eq:line-character-weight}, without an involution.  Retain one full
   unsplit sparse basis for $(3,4)$.
   \item Transform the candidate to $U_X$ using
   \eqref{eq:U0UX-coordinates}, \eqref{eq:xprime-yprime-to-UX},
   \eqref{eq:solve-uprime}, and the triangular Wronskian formula
   \eqref{eq:Wzu-to-WX-triangular}.
   \item Collect the expression in the regular logarithmic frame
   $(z',q,W_X)$ and form the cleared polynomials $\mathcal N_{j,k}$ in
   \eqref{eq:cleared-polynomial}.
   \item Impose the exact divisibility conditions
   \eqref{eq:divisibility-criterion}.  For a line pair use the three separate
   row families listed after that equation.  These conditions simultaneously
   remove the spurious logarithmic pole along $L$, impose the twist
   $\OO(-t)$, and cancel all forbidden poles along $A$ and $B$.
   \item Choose a prime for which the adopted normalized integral matrix is
   defined, reduce that integral matrix to the corresponding finite field,
   and verify full column rank.  If the frame was obtained by a rational
   rescaling such as $q_{\rm cmp}=q/5$, perform that rescaling over
   characteristic zero and reconstruct the integral matrix before reduction;
   do not reduce the rational change-of-frame scalar itself.  Repeat over a
   second prime as an implementation check; only one proof-effective prime is
   logically necessary.
\end{enumerate}

For $(3,3)$ it suffices to check the twelve genuine second-level pairs
\begin{equation}\label{eq:cubic-minimal-pairs}
 3\leqslant m\leqslant14,
 \qquad t=t_0(m):=\left\lceil\frac{m+1}{5}\right\rceil.
\end{equation}
Indeed, if $t\geqslant t_0(m)$ and a section with twist $-t$ existed,
multiplication by a nonzero form of degree $t-t_0(m)$ would inject it into
the space with twist $-t_0(m)$.  The implementation also checks $m=1,2$ as
low-degree self-tests.  The remaining minimal pairs are exactly those
displayed after Lemma~\ref{lem:keyvanishing}.

The passage from modular rank to characteristic zero is elementary but
decisive.  Every cleared matrix has integral entries.  Full column rank
modulo a prime $p$ exhibits a maximal minor that is nonzero modulo $p$;
that minor is a nonzero integer and remains nonzero over $\mathbb Q$.

The certificate program constructs the matrices from
\eqref{eq:divisibility-criterion}.  For every cubic pair in
\eqref{eq:cubic-minimal-pairs}, all twelve character--eigenvalue matrices
have full column rank.  Lemma~\ref{lem:symmetry-exhaustion} then gives zero
kernel for the entire cubic candidate space.  For $(2,e)$, every cyclic
character--eigenvalue block has full column rank; the $(3,4)$ unsplit matrix
does as well.  For $(1,e)$, every shifted cyclic character block has full
column rank; Lemma~\ref{lem:line-symmetry-exhaustion} therefore gives zero
kernel.  The proof primes are recorded in
Table~\ref{tab:finite-rank-certificates}.  For each block, the computational
record gives the degree pair, $(m,t)$, the numbers of rows and columns, the
rank, the nullity, the nonzero pivot product, and the prime modulus.
The complete computational records and instructions for independent
verification are supplied in the supplementary material cited in
footnote~\ref{fn:computational-archive}.  A separate check confirms that all
required cases and blocks are present and that their recorded data are
internally consistent.  This check is distinct from recomputing the ranks;
the supplementary material also provides the means to reproduce the rank
computations.

\begin{table}[ht]
\centering
\caption{Finite-rank certificate summary.  ``Columns'' is summed over the
indicated independent blocks; in every row the summed rank equals the summed
number of columns, so the kernel dimension is zero.  The proof modulus is the
prime used for the characteristic-zero inference.  For $(3,3)$, the totals
include the $12$ theorem-relevant pairs in \eqref{eq:cubic-minimal-pairs} and
the two low-weight implementation self-tests $m=1,2$.}
\label{tab:finite-rank-certificates}
\small
\begin{tabular}{ccrrrrc}
\toprule
pair & cases & blocks & columns & rank & kernel dim. & proof modulus \\
\midrule
$(3,3)$ & $14$ & $168$ & $280{,}798$ & $280{,}798$ & $0$ & $5$ \\
$(3,4)$ & $1$ & $1$ & $781$ & $781$ & $0$ & $5$ \\
$(2,6)$ & $1$ & $8$ & $1{,}040$ & $1{,}040$ & $0$ & $7$ \\
$(2,5)$ & $6$ & $36$ & $38{,}286$ & $38{,}286$ & $0$ & $5$ \\
$(1,9)$ & $2$ & $18$ & $5{,}526$ & $5{,}526$ & $0$ & $5$ \\
$(1,8)$ & $6$ & $48$ & $67{,}011$ & $67{,}011$ & $0$ & $5$ \\
\bottomrule
\end{tabular}
\end{table}

The proof moduli displayed in the table already suffice.  The archive also
contains reruns at other primes to test the implementation independently;
those reruns are reproducibility checks, not additional mathematical
hypotheses.  For $(3,3)$, the newly required case $(m,t)=(9,2)$ was replayed
in full at prime $7$.  For $(2,5)$, the additional case $(m,t)=(8,3)$ has
$15{,}849$ columns distributed among six character--eigenvalue blocks.  All
six blocks have full column rank modulo $5$, and an independent replay gives
full column rank modulo $7$ as well.

Prime $5$ is a bad reduction for one $(2,6)$ block: there the rank is
$92/93$.  This computation serves only as a diagnostic and is not used in
the proof.  The proof-effective $(2,6)$ certificate is the full-rank
prime-$7$ run, independently reproduced at prime $11$.

For both line pairs, prime $5$ supplies the proof-effective full-pivot
certificates and prime $11$ supplies an independent full-rank replay.  Prime
$7$ is excluded: some relevant integral minors vanish modulo $7$, and the
degree-nine test curve has bad reduction there because $1^3=-27$ in
$\mathbb F_7$.  None of the prime-$7$ diagnostics is used in the proof.

By the maximal-minor argument above, the proof-effective ranks in
Table~\ref{tab:finite-rank-certificates} give zero kernel over $\mathbb Q$,
and hence over $\mathbb C$.  This proves the required vanishing for every
test pair in \eqref{eq:all-test-pairs}.

To pass to general pairs of the fixed degrees, let $\mathcal U$ be the
Zariski-open parameter space of ordered smooth transverse pairs, let
$\mathcal D\subset\PP^2\times\mathcal U$ be the relative
simple-normal-crossing divisor, and let $\mathcal E_{m,t}$ be the relative
logarithmic invariant two-jet bundle twisted by $\OO_{\PP^2}(-t)$.
Properness of the projection $\PP^2\times\mathcal U\to\mathcal U$ and upper
semicontinuity of fiberwise cohomology imply that
\[
   (C_1,C_2)\longmapsto h^0\!\left(\PP^2,
   E_{2,m}T_{\PP^2}^*(\log(C_1+C_2))\otimes\OO(-t)\right)
\]
is upper semicontinuous.  Vanishing at the special pair therefore holds on
a nonempty Zariski-open subset of $\mathcal U$.  There are only finitely
many pairs $(m,t)$ under consideration; intersecting their nonempty open
loci proves Lemma~\ref{lem:keyvanishing}.

\section{The role of the explicit constants}
\label{sec:constant-frontier}

The constants $57$, $45$, and $69$ are explicit choices sufficient to
demonstrate the method and its effective Second Main Theorem.  More powerful
computing equipment and an enlarged certified finite-vanishing range may
yield smaller constants, but we do not intend to pursue these improvements.
Our aim is to establish the method and the stated degree range; optimizing
the numerical coefficients is not essential to that aim.

\section*{AI-assisted preparation}
Generative artificial-intelligence tools were used for language polishing
and, under detailed author direction, to assist in drafting portions of the
$\OO(3)$ argument and writing the C++ code for the exact finite computations.
The authors supplied the core ideas, precise mathematical constraints, and
iterative guidance.  The ideas and arguments are due to the authors, who
checked the C++ code and independently verified every argument in its final
form.  The authors accept full responsibility for the manuscript and the
code.

\section*{Funding}
S.-Y. Xie acknowledges partial support from National Key R\&D Program of
China under Grants No.~2023YFA1010500 and No.~2021YFA1003100, and NSFC
Grants No.~12288201 and No.~12471081, as well as Xiaomi Young Talents
Program.

\end{document}